\documentclass[10pt]{article}
\usepackage[left=1.0in,top=1.0in,right=1.0in,bottom=1.0in]{geometry}

\usepackage{graphicx}
\usepackage{subcaption}
\usepackage{amsmath}
\usepackage{amssymb}
\usepackage{amsthm}
\usepackage{latexsym}
\usepackage{bm}
\usepackage{array,supertabular}
\usepackage{color}
\usepackage{framed}
\usepackage{setspace}
\usepackage{fancyhdr}
\usepackage{stmaryrd}
\usepackage{mathrsfs}
\usepackage{mdframed}
\usepackage{url}
\usepackage{multirow}
\usepackage{algorithm}
\usepackage{algpseudocode}
\usepackage{booktabs}
\newcolumntype{P}[1]{>{\centering\arraybackslash}p{#1}}

\SetSymbolFont{stmry}{bold}{U}{stmry}{m}{n}

\newtheorem{proposition}{Proposition}
\newtheorem{remark}{Remark}

\numberwithin{equation}{section}

\begin{document}
\title{A structure-preserving integrator for incompressible finite elastodynamics based on a grad-div stabilized mixed formulation with particular emphasis on stretch-based material models}
\author{Jiashen Guan$^{\textup{a}}$, Hongyan Yuan$^{\textup{a}}$, and Ju Liu$^{\textup{ a,b,*}}$ \\
$^a$ \textit{\small Department of Mechanics and Aerospace Engineering,}\\
\textit{\small Southern University of Science and Technology,}\\
\textit{\small 1088 Xueyuan Avenue, Shenzhen, Guangdong 518055, China}\\
$^b$ \textit{\small Guangdong-Hong Kong-Macao Joint Laboratory for Data-Driven Fluid Mechanics and Engineering Applications,}\\
\textit{\small Southern University of Science and Technology}\\
\textit{\small 1088 Xueyuan Avenue, Shenzhen, Guangdong 518055, China}\\
$^{*}$ \small \textit{E-mail address:} liuj36@sustech.edu.cn, liujuy@gmail.com
}
\date{}
\maketitle

\section*{Abstract}
We present a structure-preserving scheme based on a recently-proposed mixed formulation for incompressible hyperelasticity formulated in principal stretches. Although there exist Hamiltonians introduced for quasi-incompressible elastodynamics based on different variational formulations, the one in the fully incompressible regime has yet been identified in the literature. The adopted mixed formulation naturally provides a new Hamiltonian for fully incompressible elastodynamics.  Invoking the discrete gradient formula, we are able to design fully-discrete schemes that preserve the Hamiltonian and momenta. The scaled mid-point formula, another popular option for constructing algorithmic stresses, is analyzed and demonstrated to be non-robust numerically. The generalized Taylor-Hood element based on the spline technology conveniently provides a higher-order, robust, and inf-sup stable spatial discretization option for finite strain analysis. To enhance the element performance in volume conservation, the grad-div stabilization, a technique initially developed in computational fluid dynamics, is introduced here for elastodynamics. It is shown that the stabilization term does not impose additional restrictions for the algorithmic stress to respect the invariants, leading to an energy-decaying and momentum-conserving fully discrete scheme. A set of numerical examples is provided to justify the claimed properties. The grad-div stabilization is found to enhance the discrete mass conservation effectively. Furthermore, in contrast to conventional algorithms based on Cardano's formula and perturbation techniques, the spectral decomposition algorithm developed by Scherzinger and Dohrmann is robust and accurate to ensure the discrete conservation laws and is thus recommended for stretch-based material modeling.
\vspace{5mm}

\noindent \textbf{Keywords:} Structure-preserving scheme, Elastodynamics, Isogeometric analysis, Stretch-based model, Grad-div stabilization, Discrete mass conservation

%
%
%

\section{Introduction}
\subsection{Motivation and literature survey}
\label{sec:motivation}
Preserving physical and geometrical properties in the fully discrete model has been an active research topic in computational science and engineering for decades. The foremost property is stability which guarantees numerical convergence. The pathological energy growth of the trapezoidal rule suggests stable schemes developed for linear problems are generally inapplicable to nonlinear analysis \cite{Hughes1983}, and it thus stimulates the development of energy-conserving schemes. In continuum mechanics, the total linear and angular momenta are often demanded to be preserved in the solutions; from the perspective of mathematical physics, preserving the symplectic character in the phase space can be salubrious. Designing algorithms that preserve invariants has led to a research area known as the \textit{structure preserving integrators}, and abundant numerical evidences have justified its advantage. Nevertheless, it is by no means a trivial task to preserve energy and symplecticity simultaneously \cite{Ge1988,Kane1999}. The superiority of either invariant over the other is controversial \cite{Simo1993,Marsden1998,Lew2004}. In this study, we restrict our discussion to schemes that preserve or dissipate energy and conserve momenta, and they are often referred to as the \textit{energy-momentum consistent} schemes. 

The development of energy-momentum consistent schemes heretofore is briefly summarized as follows. The initial attempt was made by enforcing the discrete conservation laws via the Lagrange multiplier method \cite{Hughes1978b}. Its major drawback is that adding constraints engenders difficulty in achieving convergence for the nonlinear solver \cite{Kuhl1996}. Inspired by the conserving algorithms developed for particle dynamics \cite{Greenspan1984}, an energy-momentum conserving method was introduced by modifying the mid-point rule with the stress evaluated in terms of a convex combination of the right Cauchy-Green tensors within a time interval \cite{Simo1992b}. The combination parameter needs to be determined iteratively through an algorithm for general materials based on the mean value theorem \cite{Laursen2001}. To alleviate the algorithmic complexity, the discrete gradient formula was proposed to provide an algorithmic stress definition, which conserves energy and momentum and conveniently works for general material models \cite{Gonzalez2000b}. Following that, several alternate algorithmic stress formulas have been proposed \cite{Romero2012}, and they have been successfully applied to shell dynamics \cite{ Miehe2001}, multibody systems \cite{Goicolea2000}, elastoplasticity \cite{Armero2006,Armero2007,Meng2002}, viscoelasticity \cite{Gross2010, Conde2014}, contact and impact \cite{Betsch2007, Hauret2006}, and incompressible elastodynamics \cite{Gonzalez2000b}. A Petrov-Galerkin finite element formulation was utilized in time for Hamiltonian systems \cite{ Betsch2001}. The aforementioned discrete gradient method can be realized within that framework by adopting a non-conventional quadrature formula. Recently, a framework based on polyconvexity has been developed \cite{Bonet2015,Bonet2015a}, and its combination with the discrete gradient method has further led to a multifield variational formulation \cite{Betsch2018}, with applications in thermo-elasticity \cite{Ortigosa2020}, viscoelasticity \cite{Kruger2016} and electro-mechanics\cite{Ortigosa2018,Franke2022}, to list a few. 

The energy-momentum schemes have been further generalized with dissipation effects taken into account. Dissipation arises due to numerical and physical mechanisms. The former is often purposely introduced to damp the error accumulation caused by poorly resolved modes from spatial discretization. This goal was first achieved in linear dynamics by the HHT-$\alpha$ \cite{Hilber1977} and generalized-$\alpha$ schemes \cite{Chung1993}. In nonlinear analysis, different approaches were made to introduce analogous damping effects on the high-frequency modes \cite{Armero1998,Armero2001a,Armero2001b,Kuhl1999}. On the other hand, using the energy-momentum method in non-reversible processes may also lead to the decay of the energy. Examples include elastoplasticity \cite{Armero2006,Armero2007,Meng2002}, viscoelasticity \cite{Gross2010, Conde2014}, and thermo-elasticity \cite{Franke2018, Ortigosa2020}. It is worth mentioning that the General Equations for Non-Equilibrium Reversible Irreversible Coupling (GENERIC) framework provides a systematic approach of extending the energy-momentum schemes to systems with physical dissipative mechanisms \cite{Romero2009,Kruger2016}. 

Many engineering and biomedical materials exhibit volume-preserving behavior under large deformation. Devising a structure-preserving scheme that works well for incompressible elastodynamics is propitious yet non-trivial. In \cite{Gonzalez2000b}, Gonzalez resorted to the \textit{quasi-incompressible} formulation based on the three-field variational principle \cite{Simo1985} and identified a modified Hamiltonian\footnote{In the definition of $\tilde{H}$, $W$ is the Helmholtz strain energy function, $\Theta$ is a kinematic field variable independently interpolated, $p$ and $\lambda$ are pressure-like Lagrange multipliers enforcing $J=\Theta$ and $\Theta = 1$, respectively.}
\begin{align*}
\tilde{H} := \int_{\Omega_{\bm X}} \frac12 \rho_0 \|\bm V\|^2 + W\left( \Theta^{2/3} \tilde{\bm C} \right) + p \left( J - \Theta \right) + \lambda \left( \Theta - 1 \right) d\Omega_{\bm X} 
\end{align*}
as an invariant for the three-field variational formulation. Invoking the discrete gradient formula, the author designed a time-stepping scheme that preserves the modified Hamiltonian $\tilde{H}$. This strategy was later extended to stretch-based elasticity problems \cite{Janz2019}. In \cite{Hauret2006}, the authors considered an alternate approach based on the two-field variational formulation \cite{Sussman1987}, in which a pressure-like variable is introduced to be interpolated independently. The Hamiltonian of the form\footnote{In the definition of $\hat{H}$, $p$ is a pressure-like variable acting as a Lagrange multiplier, and $\varepsilon$ is a non-negative number proportional to the inverse of the bulk modulus.}
\begin{align*}
\hat{H} := \int_{\Omega_{\bm X}} \frac12 \rho_0 \|\bm V\|^2 + W\left( \bm C \right) + \frac{\varepsilon}{2} p^2 d\Omega_{\bm X}
\end{align*}
is conserved after applying the discrete gradient formula to design an algorithmic stress. However, that study was still restricted to the \textit{quasi-incompressible} scenario.

\subsection{Main contribution of the proposed method}
In this work, we aim to develop a robust numerical framework for elastodynamics subjecting to the incompressibility constraint. In the incompressible limit, conventional approaches suffer from a singularity issue because the volume ratio $J$ remains $1$, and derivatives with respect to $J$ are undefined mathematically. To overcome this issue, it was suggested to perform a Legendre transformation on the volumetric part of the strain energy, resulting in a formulation based on the Gibbs free energy \cite{Liu2018}. It can be viewed as a generalization of the Herrmann variational principle \cite{Herrmann1965,Reissner1984} to the finite deformation setting, and it allows a well-behaved description for both compressible and fully-incompressible materials. As will be shown, the summation of the kinetic energy and the isochoric part of the potential energy, i.e.,
\begin{align*}
H := \int_{\Omega_{\bm X}} \frac12 \rho_0 \|\bm V\|^2 + G_{\mathrm{ich}}\left( \tilde{\bm C} \right) d\Omega_{\bm X}
\end{align*}
can be identified as a Hamilton for incompressible finite elastodynamics. We mention that the isochoric potential energy $G_{\mathrm{ich}}$ is, in fact, identical to that of the Helmholtz strain energy \cite{Liu2018}, meaning the Hamiltonian identified here is similar to but still different from the aforementioned Hamiltonians \cite{Gonzalez2000b,Hauret2006}. Based on this, one primary goal of this work is to design a scheme that preserves the invariants in the fully discrete formulation. Inspired by the discrete gradient approach \cite{Gonzalez2000b}, an algorithmic stress formula is applied for the isochoric part of the stress to achieve this goal. The scaled mid-point formula \cite{Chorin1978,Orden2021} is also considered as an alternate option for constructing the algorithmic stress. Although it is theoretically appealing \cite{Gonzalez1996,Orden2021}, its denominator is not a metric for the difference of the deformation tensors, and its numerical robustness becomes questionable. Our numerical evidence confirms this concern, and the scaled mid-point formula is not recommended for future investigation.

In this study, the material models are formulated in principal stretches to account for general isotropic hyperelastic materials. In particular, the stress and the elasticity tensor of the Ogden-type materials are carefully derived. A missing term in the documented formula for the elasticity tensor \cite[Eqn.~(6.197)]{Holzapfel2000} is identified, which has an impact on the convergence behavior for the Newton-Raphson iteration. We mention that, in practice, the quality of the discrete conservation laws is strongly influenced by the accuracy of the solution for the nonlinear systems \cite{Betsch2000}. Therefore, the missing term can be critical in certain applications. Moreover, for the eigenvalue-based models, the algorithm for the spectral decomposition impacts the solution quality substantially. To date, algorithms based on Cardano's formula are still widely adopted. However, it necessitates a perturbation technique when two or three eigenvalues coincide \cite{Miehe1993,Simo1992c}, which is known to have a detrimental impact on the eigenvector accuracy \cite{Hartmann2003,Scherzinger2008}. In the context of conserving integrators, the perturbation technique further causes the loss of the discrete conservation properties and demands a specially-developed strategy to mitigate that issue \cite{Mohr2008}. The algorithm developed by Scherzinger and Dohrmann is a robust alternative for spectral decomposition \cite{Scherzinger2008}. Its accuracy has been demonstrated to be comparable with that of LAPACK \cite{Harari2022}. In this study, we adopt this algorithm for the constitutive modeling of Ogden-type materials. Numerical tests indicate that it is as accurate as the invariant-based models in evaluating the algorithmic stresses, and there is thus no need to invoke the aforesaid strategy \cite{Mohr2008}.

For incompressible elastodynamics, the volume ratio $J$ is an invariant for the motion as well. Its preservation in the integrator is worth pursuing and is quite challenging \cite{Auricchio2005,Auricchio2010,Auricchio2013}. Within our framework, we propose a plan of two essential steps to achieve this goal. First, a spatial discretization strategy is needed to ensure the discrete satisfaction of the divergence-free condition. Second, with that velocity field, a time integration algorithm for the deformation state is demanded to ensure volume preservation. A time-stepping algorithm that satisfies this point is known as the volume-conserving algorithm \cite{Feng1995,McLachlan2002}. In this work, the Galerkin projection is made with a smooth generalization of the Taylor-Hood element using the spline technology \cite{Liu2019a,Liu2021}. This type of mixed element is convenient for implementation, inf-sup stable, robust for large-strain analysis, and related to the concept of isogeometric analysis \cite{Hughes2005}. Yet, a drawback is that its constraint ratio is large \cite[Sec.~4.3.7]{Hughes1987}, especially when the polynomial degree is low. To enhance the solution quality, the grad-div stabilization technique is introduced in our proposed formulation. This stabilization technique has been found to be rather effective in improving the discrete mass conservation in fluid mechanics and transport problems \cite{Olshanskii2002,Colomes2016,John2017} and can be interpreted as a sub-grid model \cite{Olshanskii2009}. Interestingly, it has been proved that, as the stabilization parameter approaches infinity, the solution of the Taylor-Hood element with the grad-div stabilization converges to that of the Scott-Vogelius element \cite{Scott1985a}, which is discretely divergence-free \cite{Case2011}. In this work, we introduce the grad-div stabilization to the study of elastodynamics as the first step in our proposed roadmap. We also show that the grad-div stabilization term respects the momentum conservation and dissipates the energy, thereby rendering an energy-decaying and momentum-conserving scheme. It is anticipated that the proposed numerical framework may offer a robust and accurate approach for elastodynamic analysis.

\subsection{Structure and content of the paper}
The body of this work is organized as follows. In Section \ref{sec:elastodynamics}, the strong-form problem and the constitutive relations are presented. After that, we present the major contribution of this work in Section \ref{sec:numerical_formulations}, including the energy-momentum consistent formulation, different algorithmic stress designs, grad-div stabilization, and a segregated predictor multi-corrector algorithm. In Section \ref{sec:numerical_examples}, the claimed numerical attributes are demonstrated by numerical examples. We draw conclusions in Section \ref{sec:conclusion}.

\section{Elastodynamics}
\label{sec:elastodynamics}
\subsection{Kinematics}
We start by summarizing the notations to be used in the formulation of the problem. Let the bounded open sets $\Omega_{\bm X}$ and $\Omega_{\bm x}^t \subset \mathbb R^3$ represent the initial and current configurations of a continuum body, respectively. Their boundaries are assumed to be Lipschitz and are denoted by $\Gamma_{\bm X}$ and $\Gamma_{\bm x}^t$, with unit outward normals $\bm{N}$ and $\bm{n}$, respectively. The boundary $\Gamma_{\bm X}$ can be decomposed into two non-overlapping parts, that is, $\Gamma_{\bm X} = \Gamma^{G}_{\bm X} \cup \Gamma^{H}_{\bm X}$, $\emptyset = \Gamma^{G}_{\bm X} \cap \Gamma^{H}_{\bm X}$. The subdivisions $\Gamma^{G}_{\bm X}$ and $\Gamma^{H}_{\bm X}$ will be associated with the Dirichlet and Neumann boundary conditions, respectively. As usual, we assume that there is a diffeomorphism between the two configurations,
\begin{align*}
\bm \varphi_t( \cdot ) := \bm \varphi(\cdot, t) : \Omega_{\bm X} &\rightarrow \Omega^t_{\bm x} := \bm \varphi(\Omega_{\bm X},t) = \bm \varphi_t(\Omega_{\bm X}), \quad \forall t \geq 0, \displaybreak[2] \\
\bm {X} &\mapsto \bm {x} = \bm \varphi(\bm {X},t) = \bm \varphi_t(\bm {X}),
\end{align*}
where $\bm {X} \in \Omega_{\bm X}$ labels a point in the initial configuration, and $\bm {x} \in \Omega^t_{\bm x}$ represents a point in the spatial configuration. The displacement and velocity are defined as
\begin{align*}
\bm U(\bm X,t) := \bm x - \bm X = \bm \varphi(\bm X, t) - \bm X, \quad \bm V(\bm X,t) :=  \left. \frac{\partial \bm U}{\partial t} \right|_{\bm X} = \frac{d}{dt} \bm U = \frac{d}{dt} \bm \varphi,
\end{align*}
in which $d(\cdot)/dt$ connotes the total time derivative. The deformation gradient, right Cauchy-Green deformation tensor, and Jacobian determinant are defined as
\begin{align*}
\bm {F} := \frac{\partial \bm \varphi_t}{\partial \bm{X}}, \qquad \bm {C} := \bm {F}^{\mathrm{T}} \bm {F}, \qquad J := \mathrm{det}(\bm {F}),
\end{align*}
where the superscript $\mathrm{T}$ represents the transpose of a tensor. The tensor $\bm C$ enjoys the spectral representation,
\begin{align}
\label{eq: spectral representation of C}
\bm C := \sum_{a=1}^3 \lambda_a^2 \bm{N}_a \otimes\bm{N}_a,
\end{align}
wherein $\lambda_a$ are the principal stretches, and $ \bm{N}_a$ are the principal referential directions, for $a=1,2,3$. The distortional part of $\bm C$ is given by
\begin{align*}
\tilde{\bm C} := J^{-\frac23} \bm C,
\end{align*}
which is also known as the modified right Cauchy-Green tensor and characterizes volume-preserving deformations. The principal stretches of $\tilde{\bm C}$ are given by $\tilde{\lambda}_a = J^{-1/3}\lambda_a$ and are referred to as the modified principal stretches. The derivative of $\tilde{\bm C}$ with respect to $\bm C$ can be represented in the following form,
\begin{align*}
\frac{\partial \tilde{\bm C}}{\partial \bm C} = J^{-\frac23} \mathbb P^{\mathrm{T}} \qquad \mbox{ and } \qquad \mathbb P:= \mathbb I - \frac13 \bm C^{-1} \otimes \bm C.
\end{align*}
In the above, $\mathbb I$ is the fourth-order identity tensor, and $\mathbb P$ is a projection tensor that renders a second-order tensor to be deviatoric under the Lagrangian setting. On the initial configuration, a scalar field $P$ needs to be introduced as the thermodynamic pressure, and its counterpart on the current configuration is denoted as $p := P \circ \bm \varphi_t^{-1}$. Additionally, the norm of a tensor $\bm A$ is denoted by $\|\bm A \|$ and is defined as $\|\bm A\| = (\bm A : \bm A)^{1/2}$.

\subsection{The initial-boundary value problem}
In this work, we restrict our discussion to the isothermal setting, leaving the energy equation decoupled from the mechanical system. The motion of the continuum body is governed by the following equations posed on the initial configuration,
\begin{align}
\label{eq:IBVP-kinematics}
&\frac{d \bm{U}}{dt} - \bm{V} = \bm 0, && \mbox{ in } \Omega_{\bm X} \times (0,T), \displaybreak[2] \\
\label{eq:IBVP-mass}
&J\beta \frac{dP}{dt} + \nabla_X \bm{V} : ( J \bm{F}^{-\mathrm{T}} ) = 0, && \mbox{ in } \Omega_{\bm X} \times (0,T), \displaybreak[2] \\
\label{eq:IBVP-momentum}
&\rho_0 \frac{d \bm{V}}{dt} - \nabla_X \cdot \bm P - \rho_0 \bm{B} = \bm 0,  && \mbox{ in } \Omega_{\bm X} \times (0,T), \displaybreak[2] \\
& \bm {U} = \bm{\bm{G}} \quad \mbox{ and }  \quad \bm{V} = \frac{d \bm{G}}{d t}, && \mbox{ on } \Gamma^{G}_{\bm X} \times (0,T), \displaybreak[2] \\
& \bm{P} \bm{N} = \bm{H}, && \mbox{ on } \Gamma^{H}_{\bm{X}} \times (0,T), \displaybreak[2] \\
\label{eq:IBVP-initial}
&\bm{U}(\bm{X}, 0) = \bm{U}_0(\bm{X}), \quad P(\bm{X}, 0) = P_0(\bm{X}), \quad \bm{V}(\bm{X}, 0) = \bm{V}_0(\bm{X}), && \mbox{ in } \Omega_{\bm X}.
\end{align}
In the above, $\beta$ is the isothermal compressibility factor, $\bm{P}$ is the first Piola-Kirchhoff stress, $\bm B$ is the body force per unit mass, $\bm G$ is the displacement prescribed on $\Gamma_{\bm X}^{G}$, $\bm H$ is the prescribed traction on $\Gamma^{H}_{\bm X}$, and $\bm U_0$, $P_0$, and $\bm V_0$ are the initial data. The system \eqref{eq:IBVP-kinematics}-\eqref{eq:IBVP-initial} constitutes the initial-boundary value problem for the continuum body, and the system gets closed once the material constitution is specified. Under the isothermal condition, the Gibbs free energy $G$ for hyperelastic materials enjoys the following structure, 
\begin{align}
\label{eq:IBVP-Gibbs-free-energy}
G(\tilde{\bm{C}}, P) = G_{\mathrm{ich}}(\tilde{\bm{C}} ) +  G_{\mathrm{vol}}(P).
\end{align}
In the above, the free energy $G$ is additively split into the isochoric part $G_{\mathrm{ich}}$ and the volumetric part  $G_{\mathrm{vol}}$. Notice that we choose to use `$\mathrm{ich}$', rather than `$\mathrm{iso}$' \cite{Holzapfel2000}, as the subscript for the isochoric part, because we feel the subscript `$\mathrm{iso}$' may misleadingly infer \textit{isotropic} quantities. This structure of the Gibbs free energy has been justified and results in the following constitutive relations for the density $\rho$, the isothermal compressibility factor $\beta$, and the Cauchy stress $\bm \sigma$ \cite{Liu2018,Liu2021},
\begin{align}
\label{eq:constitutive-relations}
\rho = \rho_0 \left( \frac{dG_{\mathrm{vol}}}{dP} \right)^{-1}, \quad \beta = \frac{1}{\rho}\frac{d \rho}{d P}, \quad \bm \sigma = \bm \sigma_{\mathrm{dev}} - p \bm I, \quad \bm \sigma_{\mathrm{dev}} := J^{-1} \tilde{\bm {F}} \left( \mathbb P : \tilde{\bm {S}} \right) \tilde{\bm{F}}^{\mathrm{T}}, \quad \tilde{\bm{S}} := 2 \frac{\partial G_{\mathrm{ich}}}{\partial \tilde{\bm{C}}} .
\end{align}
where the term $\tilde{\bm S}$ is known as the fictitious second Piola-Kirchhoff stress. Through the pull-back operation, the second Piola-Kirchhoff stress $\bm S$ can be obtained as follows,
\begin{align}
\label{eq:constitutive_S}
& \bm S := J \bm F^{-1} \bm \sigma \bm F^{-\mathrm{T}} = \bm S_{\mathrm{ich}} + \bm S_{\mathrm{vol}}, \\
\label{eq:consitutive_S_ich}
& \bm S_{\mathrm{ich}} := J \bm F^{-1} \bm \sigma_{\mathrm{dev}} \bm F^{-\mathrm{T}} = J^{-\frac23} \mathbb P : \left( 2 \frac{\partial G_{\mathrm{ich}}}{\partial \tilde{\bm C}} \right), \\
& \bm S_{\mathrm{vol}} := \frac13 \mathrm{tr}\left[ \bm \sigma \right] J \bm F^{-1} \bm F^{-\mathrm{T}} = -J P \bm C^{-1}.
\end{align}
Subsequently, the first Piola-Kirchhoff stress can be determined by $\bm P = \bm F \bm S$. 

\subsection{Constitutive relations in principal stretches}
\label{sec:hyperelasticity}
We complete the problem statement by defining the elastic material behavior through the free energy \eqref{eq:IBVP-Gibbs-free-energy}. The volumetric part $G_{\mathrm{vol}}(P)$ characterizes the dilational response. In this work, we focus on \textit{fully incompressible} materials, whose volumetric free energy is $G_{\mathrm{vol}}(P) = P$ \cite{Liu2018}. It leads to $\rho = \rho_0$ and $\beta = 0$ following the constitutive relations \eqref{eq:constitutive-relations}.

Since the Legendre transformation is only performed for the volumetric energy, the isochoric part of the energy adopts the same form in both Gibbs- and Helmholtz-type strain energies \cite{Herrmann1965,Liu2018,Shariff1997}. Here we focus on the Ogden model written in terms of the modified principal stretches. This model accurately describes the mechanical behavior of rubber-like materials \cite{ogden1972} and incorporates various well-known models as its particular instantiation \cite[Chapter 6.5]{Holzapfel2000}. Here, without abuse of notation, we use $G_{\mathrm{ich}}( \tilde{\lambda}_1, \tilde{\lambda}_2, \tilde{\lambda}_3 )$ to denote the energy function in terms of the stretches. Under the Valanis-Landel hypothesis, the Ogden model can be written in the form
\begin{align}
\label{eq:ogden-isochoric-energy}
G_{\mathrm{ich}}( \tilde{\lambda}_1, \tilde{\lambda}_2, \tilde{\lambda}_3 ) = \sum_{a=1}^{3} \varpi(\tilde{\lambda}_a), \quad \mbox{ with } \quad \varpi(\tilde{\lambda}_a) := \sum_{p=1}^{N} \frac{\mu_p}{\alpha_p}( \tilde{\lambda}_{a}^{\alpha_p} - 1 ).
\end{align}
Following \eqref{eq:consitutive_S_ich}, the isochoric second Piola-Kirchhoff stress can be explicitly represented as
\begin{align}
\label{eq:ogden-S-spectral-form}
\bm {S}_{\mathrm{ich}} = \sum_{a=1}^{3} S_{ \mathrm{ich} ~ a } \bm {N}_a \otimes \bm {N}_a, 
\end{align}
with
\begin{align*}
\quad S_{ \mathrm{ich} ~ a} = \frac{1}{\lambda_a^2} \left( \tilde{\lambda}_a \frac{\partial G_{\mathrm{ich}} }{\partial \tilde{\lambda}_a} -\frac{1}{3}\sum_{b=1}^{3} \tilde{\lambda}_b \frac{ \partial G_{\mathrm{ich}} }{ \partial \tilde{ \lambda}_b } \right) \quad \mbox{and} \quad \frac{\partial G_{\mathrm{ich}} }{\partial \tilde{\lambda}_a} = \frac{d\varpi}{d\tilde{\lambda}_a} = \sum_{p=1}^{N} \mu_p \tilde{\lambda}_a^{\alpha_p-1}.
\end{align*}
The above formula is known as the spectral form for the isochoric stress, and its derivation is documented in \cite[p.~246]{Holzapfel2000}. The elasticity tensor $\mathbb{C}_{\mathrm{ich}}$ can be obtained through the rate formulation \cite[Page 257]{Holzapfel2000} and is represented as
\begin{align}
\label{eq:CC-rate-form}
\mathbb{C}_{\mathrm{ich}} := 2\frac{\partial \bm S_{\mathrm{ich}}}{\partial \bm C} =& \sum_{a,b=1}^{3}\frac{1}{\lambda_b}\frac{\partial S_{\mathrm{ich~a}}}{\partial \lambda_b} \bm {N}_a \otimes \bm {N}_a \otimes \bm {N}_b \otimes \bm {N}_b \nonumber \\
&+ \sum_{\substack{a,b=1 \\ a\neq b}}^{3} \frac{ S_{\mathrm{ich~b}} - S_{\mathrm{ich~a}} }{\lambda_b^2 - \lambda_a^2}( \bm {N}_a \otimes \bm {N}_b \otimes \bm {N}_a \otimes \bm {N}_b + \bm {N}_a \otimes \bm {N}_b \otimes \bm {N}_b \otimes \bm {N}_a ).
\end{align}
Restricting to the Ogden model \eqref{eq:ogden-isochoric-energy}, one may obtain the following expression,
\begin{align}
\label{eq:CC-spectral-form}
\frac{1}{\lambda_b}\frac{\partial S_{\mathrm{ich~a}}}{\partial \lambda_b} =&
\begin{cases}
\lambda_a^{-4}\sum\limits_{p=1}^{N}\mu_p \alpha_p \left( (\frac{1}{3}-\frac{2}{\alpha_p}) \tilde{\lambda}_a^{\alpha_p} + (\frac{1}{9}+\frac{2}{3\alpha_p}) \sum\limits_{c=1}^3 \tilde{\lambda}_c^{\alpha_p}  \right) &  a=b, \\[2.0em]
\lambda_a^{-2}\lambda_b^{-2}\sum\limits_{p=1}^N \mu_p \alpha_p\left(-\frac{1}{3}
\tilde{\lambda}_b^{\alpha_p}-\frac{1}{3}\tilde{\lambda}_a^{\alpha_p} + \frac{1}{9}\sum\limits_{c=1}^3 \tilde{\lambda}_c^{\alpha_p}\right) & a\neq b.
\end{cases}
\end{align}
The above formula for the tensor components is different from the one given in \cite[p.~264]{Holzapfel2000} or \cite[p.~285]{Simo1991}, in which two missing terms can be identified. Our experience is that the missing terms may cause a loss of the convergence rate in the Newton-Raphson iterations, at least in certain cases, which further leads to the degradation of the discrete conservation properties (see Section \ref{sec:twisting-column}). A detailed derivation of \eqref{eq:CC-spectral-form} is given in Appendix \ref{sec:appendix-A}. 

\begin{remark}
When there are identical principal stretches, the quotient formula $(S_b - S_a)/(\lambda_b^2 - \lambda_a^2)$ in \eqref{eq:CC-rate-form} needs to be replaced by its limit
\begin{align*}
\lim_{\lambda_b \rightarrow \lambda_a} \frac{ (S_b - S_a)}{ \lambda_b^2 - \lambda_a^2 } = \frac{\partial S_b}{\partial \lambda_b^2} - \frac{\partial S_a}{\partial \lambda_b^2} = \frac{1}{2} ( \frac{\partial S_b}{\partial \lambda_b} - \frac{\partial S_a}{\partial \lambda_b} ),
\end{align*}
with the aid of the L'H\^{o}spital's rule.
\end{remark}

\subsection{Invariants of the motion}
We introduce a Hamiltonian of the system as the sum of the kinetic energy and the isochoric part of the Gibbs free energy, that is,
\begin{align}
\label{eq:continuum-H-def}
H := K + V, \quad K:= \int_{\Omega_{\bm X}} \frac12 \rho_0 \| \bm V \| ^2 d\Omega_{\bm X}  d\Omega_{\bm X}, \quad V:= \int_{\Omega_{\bm X}} G_{\mathrm{ich}} \left( \bm{\tilde{C}} \right) d\Omega_{\bm X}.
\end{align}
In the above, $K$ is the kinetic energy, and $V$ is the total isochoric potential energy. Different from the compressible or quasi-incompressible elastodynamics (e.g. the two Hamiltonians presented in Section \ref{sec:motivation}), the volumetric part of the energy is not involved in the definition of $H$. The total linear and angular momenta, two additional invariants, are defined as
\begin{align}
\label{eq:continuum-L-J-def}
\bm{L} := \int_{\Omega_{\bm X}} \rho_0 \bm{V} d\Omega_{\bm X}, \quad \mbox{ and } \quad \bm{J} := \int_{\Omega_{\bm X}} \rho_0 \bm \varphi_t \times \bm{V} d\Omega_{\bm X},
\end{align}
respectively. One primary goal of this work is to design a fully-discrete scheme that preserves the invariants \eqref{eq:continuum-H-def}-\eqref{eq:continuum-L-J-def} in the discrete solutions.

In addition to the above three invariants, the equation \eqref{eq:IBVP-mass} essentially state that the spatial velocity field is divergence-free, and thereby the Jacobian determinant $J$ is also preserved through the motion due to the fact that $dJ/dt = J \nabla_{\bm x} \cdot \bm v$. It has been known that preserving $J$ can be critical in detecting the stability range of the continuum problem \cite{Auricchio2010,Auricchio2013}. In the conventional formulations, different forms of the penalty function \cite{Schroder2017} and a variety of element pairs \cite{Auricchio2010,Auricchio2005} have been investigated. Nevertheless, the conventional approaches behave unsatisfactorily in dealing with these problems \cite{Auricchio2013}. Therefore, preserving $J$ is highly desirable and by no means trivial. Based on the new mixed formulation, we may address this issue in two steps, which involve purposely designed spatial and temporal discretization techniques. In this study, we focus on the first step of the strategic plan, that is, we strive to obtain a discrete divergence-free velocity field.

\begin{remark}
In the regime of compressible materials, one can get $\beta = 1/\kappa$ with a properly chosen $G_{\mathrm{vol}}$ \cite{Liu2019a}, in which $\kappa$ is the bulk modulus. With this, the Hamiltonian is appended by an integral of $p^2 / 2\kappa$ over the current configuration, which essentially recovers the Hamiltonian $\hat{H}$ \cite{Hauret2006} mentioned in Section \ref{sec:motivation}.
\end{remark}

\section{Numerical formulations}
\label{sec:numerical_formulations}
\subsection{Semi-discrete formulation}
\label{sec:semidiscrete-form}
In this section, we present the spatially discrete formulation for the initial-boundary value problem. The spatial projection is made by adopting a smooth generalization of the Taylor-Hood element based on the spline technology \cite{Buffa2011} for the following reasons. It is convenient as the displacement and velocity are interpolated via isoparametric elements. The spline technology enables a straightforward way of integrating with the CAD system within the paradigm of isogeometric analysis \cite{Hughes2005}. In the meantime, the robustness of the spline technology makes it an appealing higher-order candidate for large-strain analysis \cite{Lipton2010}. Also, recent analysis has also established its inf-sup stability property, making the choice mathematically sound \cite{Liu2019a,Rueberg2012}.

To construct the semi-discrete formulation, we need to define the spline spaces on the parametric domain $\hat{\Omega} := (0,1)^3$. A Cartesian mesh can be constructed for $\hat{\Omega}$ by open knot vectors $\Xi_d$, $d=1,2,3$. Given a set of weights, the NURBS basis functions can be constructed, and the space of the multivariate NURBS functions is denoted by $\mathcal R^{\mathsf p_1, \mathsf p_2, \mathsf p_3}_{\bm \alpha_1, \bm \alpha_2, \bm \alpha_3}$, in which $\mathsf p_d$ and $\bm \alpha_d$ represent the degree and interelement regularity in the $d$-th parametric direction, respectively. We may then define two discrete function spaces on $\hat{\Omega}$ as
\begin{align*}
\hat{\mathcal S}_h :=
\mathcal R^{\mathsf p+\mathsf a, \mathsf p+\mathsf a, \mathsf p+\mathsf a}_{\bm r_1+\mathsf b, \bm r_2+\mathsf b, \bm r_3+\mathsf b} \times \mathcal R^{\mathsf p+\mathsf a, \mathsf p+\mathsf a, \mathsf p+\mathsf a}_{\bm r_1+\mathsf b, \bm r_2+\mathsf b, \bm r_3+\mathsf b} \times \mathcal R^{\mathsf p+\mathsf a, \mathsf p+\mathsf a, \mathsf p+\mathsf a}_{\bm r_1+\mathsf b, \bm r_2+\mathsf b, \bm r_3+\mathsf b} \quad \mbox{ and } \quad
\hat{\mathcal P}_h := \mathcal R^{\mathsf p, \mathsf p, \mathsf p}_{\bm r_1, \bm r_2, \bm r_3},
\end{align*}
with integer parameters $1 \leq \mathsf a $ and $0 \leq \mathsf b < \mathsf a$. Here, for simplicity, we assume that the reference configuration of the body can be parametrized by a smooth, invertible geometrical mapping $\bm \psi : \hat{\Omega} \rightarrow \Omega_{\bm X}$. The discrete function spaces on $\Omega_{\bm X}$ can be defined through the pull-back operation,
\begin{align*}
\mathcal{S}_h : = \left \{ \bm{W}: \bm{W} \circ \bm \psi \in \hat{\mathcal{S}}_h \right \}, \qquad \mathcal{P}_h : = \left \{ Q: Q \circ \bm \psi \in \hat{\mathcal{P}}_h \right \}.
\end{align*}
With the above discrete spaces, the trial solution spaces for the displacement, pressure, and velocity on the referential configuration are defined as 
\begin{align*}
\mathcal S_{\bm{U}_h} &= \Big\lbrace \bm{U}_h : \bm{U}_h(\cdot, t) \in \mathcal S_h, \quad \bm{U}_h(\cdot,t) = \bm{G} \mbox{ on } \Gamma_{\bm X}^{G}, \quad t \in (0,T)  \Big\rbrace , \displaybreak[2] \\
\mathcal S_{P_h} &= \Big\lbrace P_h : P_h(\cdot, t) \in \mathcal P_h, \quad t \in (0,T) \Big\rbrace , \displaybreak[2] \\
 \mathcal S_{\bm{V}_h} &= \left\lbrace \bm{V}_h : \bm{V}_h(\cdot, t) \in \mathcal S_h, \quad \bm{V}_h(\cdot,t) = \frac{d\bm{G}}{dt} \mbox{ on } \Gamma_{\bm X}^{G}, \quad t \in (0,T) \right\rbrace.
\end{align*}
The corresponding test function spaces are
\begin{align}
\mathcal V_{\bm{U}_h} = \mathcal V_{\bm{V}_h} = \left\lbrace \bm{W}_h : \bm{W}_h(\cdot, t) \in \mathcal S_h, \quad \bm{W}_h(\cdot,t) = \bm 0 \mbox{ on } \Gamma_{\bm X}^{G}, \quad t \in (0,T) \right\rbrace \quad \mbox{and} \quad \mathcal V_{P_h} &=\mathcal S_{P_h}.
\end{align}
Notice that the discrete function spaces above are built based on a single-patch geometry. For many problems of practical interest, the geometries need to be represented by multiple patches. The discrete functions can be defined analogously, and the continuity is enforced across patch interfaces. We refer interested readers to \cite{Bucelli2021,Dittmann2019,Hughes2005} for a thorough discussion on this point. The semi-discrete formulation can then be stated as follows. Find $\bm Y_h(t) := \{\bm U_h(t), P_h(t), \bm V_h(t)\}^{\mathrm{T}} \in \mathcal{S}_{\bm U_h} \times \mathcal{S}_{P_h} \times \mathcal{S}_{\bm V_h}$ such that, for $t \in (0, T)$ 
\begin{align}
\label{eq:weakform-kinematics}
\bm 0 &= \frac{d \bm U_h}{dt} - \bm V_h, \\
\label{eq:weakform-mass}
0 &= \int_{\Omega_{\bm X}} Q_h J_h \nabla_{\bm X} \bm{V}_h : \bm{F}_h^{\mathrm{-T}} d\Omega_{\bm X}, \\
\label{eq:weakform-momentum}
0 &= \int_{\Omega_{\bm X}} \bm{W}_h \cdot \rho_0 \dot{\bm{V}}_h + \left( \bm{F}_h^{\mathrm{T}}\nabla_{\bm X} \bm{W}_h \right) : \bm{S}_{\mathrm{ich}~h} - J_h P_h \nabla_{\bm X} \bm{W}_h : \bm{F}_h^{\mathrm{-T}} - \bm{W}_h \cdot \rho_0 \bm B  d\Omega_{\bm X} - \int_{\Gamma_{\bm X}^H} \bm{W}_h \cdot \bm{H} d\Gamma_{\bm X} \nonumber \\
& \hspace{0.3cm} + \int_{\Omega_{\bm X}} \gamma J_h \left( \nabla_{\bm X} \bm{W}_h : \bm{F}_h^{\mathrm{-T}} \right) \left( \nabla_{\bm X} \bm{V}_h : \bm{F}_h^{\mathrm{-T}} \right)  d\Omega_{\bm X},
\end{align}
for $\forall \{ \bm W_h, Q_h \} \in \mathcal{V}_{\bm V_h} \times \mathcal{V}_{P_h}$. The initial solutions $\bm Y_h(0)=\{ \bm U_h(0), P_h(0), \bm V_h(0) \}$ are obtained through the $\mathcal L^2$-projection of the initial data. The last term in \eqref{eq:weakform-momentum} is the grad-div stabilization with a parameter $\gamma \geq 0$. This term can be introduced in a more general fashion by allowing the parameter $\gamma$ to vary in space and time  \cite{Olshanskii2009}. Without losing generality, we use a constant parameter $\gamma$ for the whole domain to simplify the subsequent discussion. In the following, we briefly discuss the energy and momentum conservation properties embedded in the semi-discrete formulation \eqref{eq:weakform-kinematics}-\eqref{eq:weakform-momentum}. First, if the boundary data $\bm G$ is time-independent, the spaces $\mathcal{S}_{V_h}$ and $\mathcal{V}_{V_h}$ become identical. Choosing $\bm W_h = \bm V_h$ in \eqref{eq:weakform-momentum} leads to the energy stability property,
\begin{align*}
\frac{d H_h}{dt} = P_{\mathrm{ext}~h} - \mathcal{D}_{h},
\end{align*}
wherein the spatially discrete Hamiltonian is
\begin{align*}
H_h := \int_{\Omega_{\bm X}} \frac12 \rho_0 \| \bm V_h \| ^2 d\Omega_{\bm X} + G_{\mathrm{ich}} \left( \bm{\tilde{C}}_h \right) d\Omega_{\bm X},
\end{align*}
and
\begin{align*}
\quad P_{\mathrm{ext}~h} := \int_{\Omega_{\bm X}} \bm{V}_h \cdot \rho_0 \bm B d\Omega_{\bm X} + \int_{\Gamma_{\bm X}^H} \bm{V}_h \cdot \bm{H}d\Gamma_{\bm X}  \quad \mbox{ and } \quad
\mathcal{D}_h := \int_{\Omega_{\bm X}} \gamma J_h \left(\nabla_{\bm X} \bm{V}_h : \bm F_h^{-\mathrm{T}} \right)^2 d\Omega_{\bm X}
\end{align*}
are the power of external loadings and the dissipation, respectively. We mention that the volumetric force, or the pressure, does not contribute to the variation of the Hamiltonian $H$. This fact is naturally built into the above formulation due to the fact that the constraint operator in \eqref{eq:weakform-mass} is conjugate to the gradient operator acting on pressure. In contrast, this may not hold if $J=1$ is imposed as the constraint equation in the conventional two-field variational formulation. 

Second, if $\Gamma^{G}_{\bm X} = \emptyset$, an arbitrary constant vector and its cross product with $\bm \varphi_{t~h} := \bm U_h(\bm X,t) + \bm X$ are both admissible test functions. This implies that the semi-discrete formulation is endowed with the following properties,
\begin{align*}
& \frac{d}{dt} \bm{L}_h := \frac{d}{dt} \int_{\Omega_{\bm X}} \rho_0 \bm V_h d\Omega_{\bm X} = \int_{\Omega_{\bm X}} \rho_0 \bm{B} d\Omega_{\bm X} + \int_{\Gamma_{\bm X}^{H}} \bm{H} d\Gamma_{\bm X}, \\
& \frac{d}{dt} \bm{J}_h := \frac{d}{dt} \int_{\Omega_{\bm X}} \rho_0 \bm \varphi_{t~h} \times \bm V_h d\Omega_{\bm X} = \int_{\Omega_{\bm X}} \rho_0 \bm \varphi_{t~h} \times \bm{B} d\Omega_{\bm X} + \int_{\Gamma_{\bm X}^{H}} \bm \varphi_{t~h} \times \bm{H} d\Gamma_{\bm X},
\end{align*}
representing the semi-discrete conservation of the linear and angular momenta.

Third, it is indeed desirable to have $\nabla_{\bm X}\bm V_h : \bm F_h^{-\mathrm{T}}=0 \in \mathcal V_{P_h}$ for an arbitrary admissible velocity solution $\bm V_h$, the equation \eqref{eq:weakform-mass} implies $\mathcal{D}_h = 0$, which further implies a pointwise satisfaction of the constraint equation. There exist discrete function spaces that satisfy this property \cite{Buffa2011,Evans2013,John2017,Scott1985a}, and this is a research direction worthy of further pursuing. For the discrete functions considered in this work, this property does not hold, and the term $\mathcal D_h$ is non-negative. It can be viewed as a penalization of the constraint and is dissipative in energy. This term was originally introduced in the context of stabilized finite element method \cite{Franca1988a} and can be interpreted as the sub-grid scale model for the pressure \cite{Olshanskii2009}. Its effect was not fully realized when used with equal-order interpolation. It recently gained popularity for inf-sup stable elements as it was recognized to have a significant impact on improving discrete mass conservation for incompressible flow and transport problems \cite{Colomes2016,John2010,Olshanskii2002,Olshanskii2009}. In particular, its solution asymptotically approaches the discrete divergence-free solution as the stabilization parameter gets larger \cite{Scott1985a}. Here, we introduce the grad-div stabilization mechanism to the above elastodynamics formulation with the goal of enhancing the discrete satisfaction of the divergence-free constraint.

\subsection{Temporal discretization}
\label{sec:temporal_discretization}
In this section, we perform the temporal discretization for the semi-discrete formulation \eqref{eq:weakform-kinematics}-\eqref{eq:weakform-momentum}. To simplify the notations, we will henceforth neglect the subscript $h$ used for indicating spatially discrete quantities. The time domain $(0, T)$ is discretized into $n_{\mathrm{ts}}$ sub-intervals $\mathcal I_n := (t_n, t_{n+1})$, with the time step size $\Delta t_n := t_{n+1}- t_n$. We denote the algorithmic approximation of a quantity $(\cdot)$ at time $t_n$ by  $(\cdot)_n$. Our fully-discrete formulation can be expressed as follows. Given $\bm Y_n := \{\bm U_n, P_n, \bm V_n\}^{\mathrm{T}}$, find $\bm Y_{n+1} := \{\bm U_{n+1}, P_{n+1}, \bm V_{n+1}\}^{\mathrm{T}} \in \mathcal{S}_{\bm U} \times \mathcal{S}_{P} \times \mathcal{S}_{\bm V}$ such that
\begin{align}
\label{eq:em-kinematic}
\bm 0 & = \frac{\bm{U}_{n+1} - \bm{U}_{n}}{\Delta t_n} - \bm{V}_{m}, \displaybreak[2] \\
\label{eq:em-mass}
0 &= \int_{\Omega_{\bm X}} Q J_{m} \nabla_{\bm X} \bm{V}_{m} : \bm{F}^{\mathrm{-T}}_{m} d\Omega_{\bm X}, \displaybreak[2] \\
\label{eq:em-momentum}
0 &= \int_{\Omega_{\bm X}} \bm{W} \cdot \rho_0 \frac{\bm{V}_{n+1} - \bm{V}_{n}}{\Delta t_n} + \left( \bm{F}^{\mathrm{T}}_{m}\nabla_{\bm X} \bm{W} \right) : \bm{S}_{\mathrm{ich}~\mathrm{alg}} - J_{m} P_{m} \nabla_{\bm X} \bm{W} : \bm{F}^{\mathrm{-T}}_{m} - \bm{W} \cdot \rho_0 \bm B_{m} d\Omega_{\bm X} \nonumber \displaybreak[2] \\
& \hspace{0.3cm} - \int_{\Gamma_{\bm X}^H} \bm{W} \cdot \bm{H}_{m} d\Gamma_{\bm X} + \int_{\Omega_{\bm X}} \gamma J_{m} \left( \nabla_{\bm X} \bm{W} : \bm{F}^{\mathrm{-T}}_{m} \right) \left( \nabla_{\bm X} \bm{V}_{m} : \bm{F}^{\mathrm{-T}}_{m} \right) ,
\end{align}
for $\forall \{ \bm W, Q \} \in \mathcal{V}_{\bm V} \times \mathcal{V}_P$, in which
\begin{gather}
 \left\lbrace  \bm{U}_{m}, P_{m}, \bm{V}_{m} \right\rbrace^{\mathrm{T}} :=  \frac12 \left\lbrace \bm{U}_{n} + \bm{U}_{n+1}, P_{n} + P_{n+1}, \bm{V}_{n} + \bm{V}_{n+1} \right\rbrace^{\mathrm{T}},\\
 \bm{F}_{m} := \nabla_{\bm X} \bm{U}_{m} + \bm I, \qquad
J_{m} := \mathrm{det}\left( \bm{F}_{m} \right), \qquad \bm{C}_{m} := \frac12 \left( \bm{C}_{n+1} + \bm{C}_{n} \right), \qquad \tilde{\bm{C}}_{m} := J^{-\frac23}_{m} \bm{C}_{m}, \\
 \label{eq:algorithmic Stress tensor}
 \bm{S}_{\mathrm{ich}~\mathrm{alg}} := \bm{S}_{\mathrm{ich}~m} + \bm{S}_{\mathrm{ich}~\mathrm{enh}}, \qquad 
 \bm S_{\mathrm{ich}~m} := \bm{S}_{\mathrm{ich}}\left( \bm{C}_{m} \right),\quad
\bm B_{m} := \bm B(t_{m}), \quad \bm H_{m} := \bm H(t_{m}).
\end{gather}
The discrete formulation \eqref{eq:em-kinematic}-\eqref{eq:algorithmic Stress tensor} is complete once the term $\bm S_{\mathrm{ich}~\mathrm{alg}}$, or equivalently the so-called ``stress enhancement" $\bm S_{\mathrm{ich}~\mathrm{enh}}$, is defined. There are a few design criteria for $\bm S_{\mathrm{ich}~\mathrm{alg}}$ to make the overall scheme \eqref{eq:em-kinematic}-\eqref{eq:algorithmic Stress tensor} energy-momentum consistent. First, it needs to satisfy the following property, i.e., 
\begin{align}
\label{eq:directionality-condition}
\bm S_{\mathrm{ich}~\mathrm{alg}} : \bm Z_n = G_{\mathrm{ich}}(\tilde{\bm{C}}_{n+1}) - G_{\mathrm{ich}}(\tilde{\bm{C}}_{n}),
\end{align}
with $\bm{Z}_{n} := \left( \bm{C}_{n+1} - \bm{C}_{n} \right)/2$. This ensures the relation $d G_{\mathrm{ich}}/dt = \bm S_{\mathrm{ich}} : d \bm C/2dt$ is inherited to the fully-discrete setting. We mention that the relation \eqref{eq:directionality-condition} is indeed the \textit{directionality} property restricted to the isochoric part of the energy and stress, which is slightly different from that in the compressible theory \cite{Romero2012}. With this property and proper boundary conditions, the discrete energy stability can be established.
\begin{proposition}
Assuming the boundary data $\bm{G}$ is time independent and the directionality property \eqref{eq:directionality-condition} holds, the scheme \eqref{eq:em-kinematic}-\eqref{eq:algorithmic Stress tensor} enjoys the following energy stability property
\begin{align}
\label{eq:em-prop-1}
\frac{1}{\Delta t_n} \left( H_{n+1} - H_n \right) = \int_{\Omega_{\bm X}} \rho_0 \bm{V}_{m} \cdot \bm B_{m} d\Omega_{\bm X} + \int_{\Gamma_{\bm X}^H} \bm{V}_{m} \cdot \bm H_{m} d\Gamma_{\bm X} -  \mathcal{D}_{m},
\end{align}
wherein
\begin{align}
\label{eq: discrete Hamiltonian and Dissipation}
H_n := \int_{\Omega_{\bm X}} \frac{\rho_0 \|\bm{V}_{n}\|^2}{2} + G_{\mathrm{ich}}\left( \tilde{\bm{C}}_{n} \right) d\Omega_{\bm X} \quad \mbox{ and } \quad \mathcal{D}_{m} := \int_{\Omega_{\bm X}} \gamma J_{m} \left( \nabla_{\bm X} \bm{V}_{m} : \bm{F}^{\mathrm{-T}}_{m} \right)^2 d\Omega_{\bm X} \geq 0.
\end{align}
\end{proposition}
\begin{proof}
The time independence of the boundary data $\bm{G}$ suggests the function spaces $\mathcal S_{\bm V}$ and $\mathcal V_{\bm V}$ are identical. We may thus choose $\bm{W} = \bm{V}_{m}$ in \eqref{eq:em-momentum}, leading to
\begin{align}
\label{eq:em-prop-2}
0 &= \int_{\Omega_{\bm X}} \bm{V}_{m} \cdot \rho_0 \frac{\bm{V}_{n+1} - \bm{V}_n}{\Delta t_n} + \left( \bm{F}^{\mathrm{T}}_{m}\nabla_{\bm X} \bm{V}_{m} \right) : \bm{S}_{\mathrm{ich}~\mathrm{alg}} -J_{m} P_{m} \nabla_{\bm X} \bm{V}_{m} : \bm{F}^{\mathrm{-T}}_{m} d\Omega_{\bm X} \nonumber \\
& \hspace{0.3cm} -\int_{\Omega_{\bm X}} \bm{V}_{m} \cdot \rho_0 \bm{B}_{m} d\Omega_{\bm X} -\int_{\Gamma_{\bm X}^H} \bm{V}_{m} \cdot \bm{H}_{m} d\Gamma_{\bm X} +\int_{\Omega_{\bm X}} \gamma J_{m} \left( \nabla_{\bm X} \bm{V}_{m} : \bm{F}^{\mathrm{-T}}_{m} \right)^2 d\Omega_{\bm X}.
\end{align}
The first term on the right-hand side of  \eqref{eq:em-prop-2} can be reorganized as
\begin{align*}
\int_{\Omega_{\bm X}} \bm{V}_{m} \cdot \rho_0 \frac{\bm{V}_{n+1} - \bm{V}_n}{\Delta t_n} d\Omega_{\bm X} = \int_{\Omega_{\bm X}} \rho_0 \frac{\|\bm{V}_{n+1}\|^2 - \|\bm{V}_{n}\|^2}{2 \Delta t_n} d\Omega_{\bm X}.
\end{align*}
Regarding the second term on the right-hand side of \eqref{eq:em-prop-2}, we make use of the kinematic relation \eqref{eq:em-kinematic} and get
\begin{align*}
& \int_{\Omega_{\bm X}} \left( \bm{F}^{\mathrm{T}}_{m} \nabla_{\bm X} \bm{V}_{m} \right) : \bm{S}_{\mathrm{ich}~\mathrm{alg}} d\Omega_{\bm X} = \frac{1}{\Delta t_n} \int_{\Omega_{\bm X}} \left( \bm{F}^{\mathrm{T}}_{m} \left( \bm{F}_{n+1} - \bm{F}_{n} \right) \right) : \bm{S}_{\mathrm{ich}~\mathrm{alg}} d\Omega_{\bm X} \displaybreak[2] \\
=& \frac{1}{\Delta t_n} \int_{\Omega_{\bm X}} \frac12 \left( \bm{F}^{\mathrm{T}}_{n+1} \bm{F}_{n+1} - \bm{F}^{\mathrm{T}}_{n} \bm{F}_{n} + \bm{F}^{\mathrm{T}}_{n} \bm{F}_{n+1} - \bm{F}^{\mathrm{T}}_{n+1} \bm{F}_{n} \right) : \bm{S}_{\mathrm{ich}~\mathrm{alg}} d\Omega_{\bm X} = \frac{1}{\Delta t_n} \int_{\Omega_{\bm X}} \bm{Z}_{n} : \bm{S}_{\mathrm{ich}~\mathrm{alg}} d\Omega_{\bm X} \\
=& \frac{1}{\Delta t_n} \int_{\Omega_{\bm X}} \left( G_{\mathrm{ich}}(\tilde{\bm{C}}_{n+1} ) - G_{\mathrm{ich}}(\tilde{\bm{C}}_{n}) \right) d\Omega_{\bm X}.
\end{align*}
Noticing that $\bm{S}_{\mathrm{ich}~\mathrm{alg}}$ is symmetric, and $\bm{F}^{\mathrm{T}}_{n} \bm{F}_{n+1} - \bm{F}^{\mathrm{T}}_{n+1} \bm{F}_{n}$ is skew-symmetric, the contraction of the two is zero in the second equality of the above derivation. The last equality of the above is due to the directionality property \eqref{eq:directionality-condition}. We may choose $Q = P_{m}$ in \eqref{eq:em-mass} by virtue of $\mathcal{S}_{P} = \mathcal{V}_{P}$, and the third term in \eqref{eq:em-prop-2} vanishes, indicating that the work done by the pressure does not contribute to the energy variation for fully incompressible materials. Based on the above discussion, the relation \eqref{eq:em-prop-1} is established.
\end{proof}
\begin{remark}
\label{remark:grad-div-dissipation-vanishment}
According to \eqref{eq: discrete Hamiltonian and Dissipation}, the dissipation $\mathcal{D}_{m}$ is governed by the value of $\gamma$ and $\|\nabla_{\bm x} \cdot \bm v_m\|_{\mathcal L_2}$. It has been demonstrated that the $\mathcal L_2$-norm of the discrete velocity divergence is bounded by $C / \gamma$ in the approximation of the Navier-Stokes equations using the Taylor-Hood element \cite{Case2011}. It is therefore reasonable to expect the dissipation approaches zero asymptotically when $\gamma$ goes to infinity. In practice, as the parameter $\gamma \rightarrow \infty$, the iterative solver will experience difficulty in achieving convergence \cite{Bowers2014}.
\end{remark}

Second, the algorithmic stress $\bm S_{\mathrm{ich}~\mathrm{alg}}$ needs to be symmetric to ensure the discrete angular momentum conservation. This is a well-known fact in the conventional formulation \cite{Gonzalez1996b,Simo1992}. Here we demonstrate that the newly-introduced grad-div stabilization term does not impose additional restrictions for the angular momentum conservation through the following analysis.
\begin{proposition}
Assuming $\Gamma^{G}_{\bm X} = \emptyset$, the discrete linear momentum is conserved in the following sense,
\begin{align}
\label{em:prop-L}
& \frac{1}{\Delta t_n} \left( \bm{L}_{n+1} - \bm{L}_{n} \right) = \int_{\Omega_{\bm X}} \rho_0 \bm{B}_{m} d\Omega_{\bm X} + \int_{\Gamma_{\bm X}^{H}} \bm{H}_{m} d\Gamma_{\bm X},
\end{align}
if further the algorithmic stress $\bm S_{\mathrm{ich}~\mathrm{alg}}$ is symmetric, the discrete angular momentum is conserved, i.e.,
\begin{align}
\label{em:prop-J}
& \frac{1}{\Delta t_n} \left( \bm{J}_{n+1} - \bm{J}_{n} \right) = \int_{\Omega_{\bm X}} \rho_0 \bm \varphi_{m} \times \bm{B}_{m} d\Omega_{\bm X} + \int_{\Gamma_{\bm X}^{H}} \bm \varphi_{m} \times \bm{H}_{m} d\Gamma_{\bm X}.
\end{align}
In the above,
\begin{align}
\label{eq: discrete momenta}
\bm{L}_{n} := \int_{\Omega_{\bm X}} \rho_0 \bm{V}_{n} d\Omega_{\bm X} \quad \mbox{ and } \quad \bm{J}_{n} := \int_{\Omega_{\bm X}} \rho_0 \bm \varphi_{n} \times \bm{V}_{n} d\Omega_{\bm X}
\end{align}
are the discrete approximations of $\bm{L}$ and $\bm{J}$ at time $t_n$.
\end{proposition}
\begin{proof}
Choosing $\bm{W = \bm \xi } \in \mathcal V_{\bm V}$ for an arbitrary constant vector $\bm \xi \in \mathbb R^3$ in \eqref{eq:em-momentum}, one has
\begin{align*}
\bm \xi \cdot \left( \frac{1}{\Delta t_n} \left( \bm{L}_{n+1} - \bm{L}_{n} \right) - \int_{\Omega_{\bm X}} \rho_0 \bm{B}_{m} d\Omega_{\bm X} - \int_{\Gamma_{\bm X}^{H}} \bm{H}_{m} d\Gamma_{\bm X} \right) = 0.
\end{align*}
Due to the arbitrariness of $\bm \xi$, one may conclude that \eqref{em:prop-L} holds. It can be also shown that $\bm \xi \times \bm \varphi_{m}$ is an admissible test function for an arbitrary constant vector $\bm \xi$. Choosing $\bm{W} = \bm \xi \times \bm \varphi_{m}$ in \eqref{eq:em-momentum} leads to the following,
\begin{align}
\label{eq:em-prop-3}
0 =& \int_{\Omega_{\bm X}} \left( \bm \xi \times \bm \varphi_{m} \right) \cdot \rho_0 \frac{\bm{V}_{n+1} - \bm{V}_{n}}{\Delta t_n} + \left( \bm{F}^{\mathrm{T}}_{m}\nabla_{\bm X} \left( \bm \xi \times \bm \varphi_{m} \right) \right) : \bm{S}_{\mathrm{ich}~\mathrm{alg}} d\Omega_{\bm X} \nonumber \\
& - \int_{\Omega_{\bm X}} J_{m} P_{m} \nabla_{\bm X} \left( \bm \xi \times \bm \varphi_{m} \right) : \bm{F}^{\mathrm{-T}}_{m} 
- \left( \bm \xi \times \bm \varphi_{m} \right) \cdot \rho_0 \bm{B}_{m} d\Omega_{\bm X} - \int_{\Gamma_{\bm X}^H} \left( \bm \xi \times \bm \varphi_{m} \right) \cdot \bm{H}_{m} d\Gamma_{\bm X} \nonumber \\
& + \int_{\Omega_{\bm X}} \gamma J_{m}\left( \nabla_{\bm X} \left( \bm \xi \times \bm \varphi_{m} \right) : \bm{F}^{\mathrm{-T}}_{m} \right) \left( \nabla_{\bm X} \bm{V}_{m} : \bm{F}^{\mathrm{-T}}_{m} \right) d\Omega_{\bm X}.
\end{align}
It is known that there exists a skew tensor $\bm{W}_{\bm \xi}$ such that $\bm \xi \times \bm a = \bm{W}_{\bm \xi} \bm a$ for any vector $\bm a$. With $\bm W_{\bm \xi}$, the second term in the above can be reorganized as
\begin{align*}
\left( \bm{F}^{\mathrm{T}}_{m}\nabla_{\bm X} \left( \bm \xi \times \bm \varphi_{m} \right) \right) : \bm{S}_{\mathrm{ich}~\mathrm{alg}} &= \left( \nabla_{\bm X}\left( \bm{W}_{\bm \xi} \bm \varphi_{m} \right) \right) : \left( \bm{F}_{m} \bm{S}_{\mathrm{ich}~\mathrm{alg}} \right) = \bm{W}_{\bm \xi} : \left( \bm{F}_{m} \bm{S}_{\mathrm{ich}~\mathrm{alg}} \bm{F}^{\mathrm{T}}_{m} \right) = 0.
\end{align*}
The last equality of the above is due to the symmetry of $\bm S_{\mathrm{ich}~\mathrm{alg}}$. Invoking an analogous argument, we have
\begin{align*}
\nabla_{\bm X} \left( \bm \xi \times \bm \varphi_{m} \right) : \bm{F}^{\mathrm{-T}}_{m} &= \bm{W}_{\bm \xi} : \left( \bm{F}^{\mathrm{-T}}_{m}  \bm{F}_{m}^{\mathrm{T}} \right) = \bm{W}_{\bm \xi} : \bm{I} = 0.
\end{align*}
Therefore, the third and last terms in \eqref{eq:em-prop-3} vanish. Consequently, the relation \eqref{eq:em-prop-3} can be simplified as
\begin{align}
\label{eq:em-prop-4}
\frac{1}{\Delta t_n} \int_{\Omega_{\bm X}}  \rho_0 \bm \varphi_{m} \times \left( \bm V_{n+1} - \bm{V}_{n} \right) d\Omega_{\bm X} = \int_{\Omega_{\bm X}} \rho_0 \bm \varphi_{m} \times \bm{B}_{m} d\Omega_{\bm X} + \int_{\Gamma_{\bm X}^{H}} \bm \varphi_{m} \times \bm{H}_{m} d\Gamma_{\bm X}.
\end{align}
Performing an inner product of the equation \eqref{eq:em-kinematic} with $\bm \xi \times \rho_0 \bm{V}_{m}$ leads to
\begin{align}
\label{eq:em-prop-5}
\frac{1}{\Delta t_n} \int_{\Omega_{\bm X}} \rho_0 \bm{V}_{m} \times \left( \bm \varphi_{n+1} - \bm \varphi_n \right)d\Omega_{\bm X} = 0.
\end{align}
Adding \eqref{eq:em-prop-5} with \eqref{eq:em-prop-4} gives \eqref{em:prop-J}, which completes the proof.
\end{proof}

Third, the term $\bm S_{\mathrm{ich}~\mathrm{alg}}$ needs to be a second-order approximation of  $\bm S_{\mathrm{ich}}(\bm F^{\mathrm{T}}_{m} \bm F_{m} )$, which is needed to maintain the second-order temporal accuracy of the scheme \eqref{eq:em-kinematic}-\eqref{eq:algorithmic Stress tensor}. It is worth pointing out that the stress $\bm S_{\mathrm{ich}~m}$ is purposely evaluated at $\bm C_{m}$ rather than $\bm F^{\mathrm{T}}_{m} \bm F_{m}$ to preserve the relative equilibria of the continuum system \cite{Armero2001a}.

Algorithmic stresses that satisfy the three aforementioned attributes are candidates for the construction of an energy-momentum consistent scheme. As was analyzed in \cite{Romero2012}, the algorithmic stress can be designed through a constrained optimization problem, minimizing the difference between the algorithmic stress and $\bm S_{\mathrm{ich}~m}$ for symmetric tensors  that satisfy the directionality property. The result can be represented as
\begin{align}
\label{eq:general-formula-Romero}
\bm S_{\mathrm{ich}~\mathrm{alg}} = \bm S_{\mathrm{ich}~m} + \frac{ \Delta  G_{\mathrm{ich}} - \bm{S}_{\mathrm{ich}~m} : \bm{Z}_{n} }{\bm{Z}_{n}: \left(\mathbb M^{-1} \bm Z_n \right)} \mathbb M^{-1} \bm Z_n,
\end{align}
with $\mathbb M$ being a fourth-order positive-definite tensor \cite[Eqn.~(29)]{Romero2012} and $\Delta G_{\mathrm{ich}}:= G_{\mathrm{ich}}(\tilde{\bm{C}}_{n+1}) - G_{\mathrm{ich}}(\tilde{\bm{C}}_{n})$.

\paragraph{The discrete gradient}
We first present a choice for the algorithmic stress as
\begin{align}
\label{eq:The Gonzalez discrete gradient}
\bm S^{\mathrm{gon}}_{\mathrm{ich}~\mathrm{alg}} = \bm S_{\mathrm{ich}~m}+\bm S_{\mathrm{ich}~\mathrm{enh}}^{\mathrm{gon}}, \quad \mbox{ with } \quad \bm S_{\mathrm{ich}~\mathrm{enh}}^{\mathrm{gon}} := \frac{\Delta G_{\mathrm{ich}} - \bm{S}_{\mathrm{ich}~m} : \bm{Z}_{n} }{\|\bm{Z}_{n}\|} \frac{\bm{Z}_{n}}{\|\bm{Z}_{n}\|}.
\end{align}
This formula was originally proposed by Gonzalez in 1996 \cite{Gonzalez1996}, and it can also be obtained by taking $\mathbb M$ as the fourth-order identity tensor in \eqref{eq:general-formula-Romero}. 


\paragraph{The scaled mid-point formula}
An alternative algorithmic stress was originally suggested in \cite[p.~252]{Chorin1978} and can be stated as
\begin{align}
\label{eq:scaled-mid-point-formula}
\bm S^{\mathrm{sca}}_{\mathrm{ich}~\mathrm{alg}} = \bm S_{\mathrm{ich}~m}+\bm S_{\mathrm{ich}~\mathrm{enh}}^{\mathrm{sca}} = \frac{\Delta G_{\mathrm{ich}}}{\bm S_{\mathrm{ich~m}}:\bm Z_n} \bm S_{\mathrm{ich}~m}, \quad \bm S_{\mathrm{ich}~\mathrm{enh}}^{\mathrm{sca}} := \frac{ \Delta G_{\mathrm{ich}} - \bm{S}_{\mathrm{ich}~m} : \bm{Z}_{n} }{\bm{S}_{\mathrm{ich~m}}:\bm{Z}_n} \bm{S}_{\mathrm{ich~m}}.
\end{align}
It was subsequently used in many-body dynamics \cite{Gonzalez1996} and elastodynamics \cite{Bui2007} as a conserving integrator. However, this candidate is rather controversial. On one side, it has been under criticism because its multiplicative modification affects the whole stress homogeneously \cite{Reich1996}. Romero also demonstrated that its corresponding tensor norm in the optimization problem cannot be equivalent to the Frobenius norm \cite{Romero2012}. On the other side, a recent study on many-body dynamics \cite{Orden2021}  suggests that the scaled mid-point formula \eqref{eq:scaled-mid-point-formula} is the unique one that enjoys the G-invariant property \cite{Gonzalez1996b} and is demonstrated to perform rather well. For isotropic materials, the stress $\bm S^{\mathrm{sca}}_{\mathrm{ich}~\mathrm{alg}}$ is coaxial with the deformation tensor $\bm C_m$. Coaxial stresses are featured by several interesting and beneficial properties \cite[Appendix~A]{SouzaNeto2008}. In this regard, a different algorithmic stress that maintains the coaxial property can be obtained by choosing $\mathbb M$ in \eqref{eq:general-formula-Romero} such that $\mathbb M^{-1} \bm Z_n = \bm C_{m}$, i.e.,
\begin{align}
\label{eq:a new coaxial gradient}
\bm S^{\mathrm{coa}}_{\mathrm{ich}~\mathrm{alg}} = \bm S_{\mathrm{ich}~m}+\bm S_{\mathrm{ich}~\mathrm{enh}}^{\mathrm{coa}}, \quad  \bm S_{\mathrm{ich}~\mathrm{enh}}^{\mathrm{coa}} := \frac{ \Delta G_{\mathrm{ich}} - \bm{S}_{\mathrm{ich}~m} : \bm{Z}_{n} }{\bm{C}_m:\bm{Z}_n} \bm{C}_{m}.
\end{align}

We mention that the above formulas for the algorithmic stress all adopt a quotient form. In practice, the denominator needs to be monitored. If it gets too small, one needs to turn off the stress enhancement to avoid numerical singularities. In our implementation, the switching criterion is given by a prescribed tolerance $\mathrm{tol}_{\mathrm{B}}$, and a reasonable choice of the value of $\mathrm{tol}_{\mathrm{B}}$ is $10^{-10}$ for the Gonzalez discrete gradient formula according to numerical tests. For the scaled mid-point formula \eqref{eq:scaled-mid-point-formula} and its extension \eqref{eq:a new coaxial gradient}, the denominators are more likely to get below the tolerance. This can be better explained using the formula \eqref{eq:a new coaxial gradient}, whose denominator is $\bm{C}_m:\bm{Z}_n = (\|\bm C_{n+1}\|^2 - \|\bm C_n\|^2) / 4$. In comparison with $\|\bm Z_n\|^2$ in the discrete gradient \eqref{eq:The Gonzalez discrete gradient}, the expression $\bm C_m : \bm Z_n$ is not a metric for $\bm Z_n$, meaning $\bm C_m : \bm Z_n$ can be zero while $\bm C_{n+1} \neq \bm C_n$. Detailed numerical evidence is provided in Section \ref{sec:numerical-robustness}. The lack of the metric nature renders the latter two formulas non-robust and are thus not recommended.

\subsection{A segregated predictor multi-corrector algorithm}
The Newton-Raphson method is invoked to deal with the nonlinear algebraic equations in each time step. At the time step $t_{n+1}$, the solution vector $\bm Y_{n+1}$ is solved by means of a predictor multi-corrector algorithm. We denote  $\bm{Y}_{n+1, (l)} := \{ \bm U_{n+1, (l)}, P_{n+1, (l)}, \bm V_{n+1, (l)} \}^{\mathrm{T}}$ as the solution vector at the Newton-Raphson iteration step $l = 0, ..., l_{\mathrm{max}}$. At the $l$-th iteration step, the residual vectors corresponding to \eqref{eq:em-kinematic}-\eqref{eq:em-momentum} evaluated with $\bm Y_{n+1, (l)}$ are denoted as
\begin{align*}
\bm{ \mathrm{R} }_{(l)} := \left \{ \bm{ \mathrm{R} }_{(l)}^{\mathrm{k}}, \bm{ \mathrm{R} }_{(l)}^{\mathrm{p}}, \bm{ \mathrm{R} }_{(l)}^{\mathrm{m}} \right \}^{\mathrm{T}}.
\end{align*}
The superscripts $\mathrm{k}$, $\mathrm{p}$, and $\mathrm{m}$ denote the kinematic, mass, and linear momentum discrete equations, respectively. The consistent tangent matrix can be represented as
\begin{align*}
\bm{ \mathrm{K} }_{(l)} = 
\begin{bmatrix}
\bm{ \mathrm{K} }^{\mathrm{k}}_{(l),\bm U} & \bm{\mathrm{O}} & \bm{ \mathrm{K} }^{\mathrm{k}}_{(l),\bm V} \\
\bm{ \mathrm{K} }^{\mathrm{p}}_{(l),\bm U} & \bm{\mathrm{O}} & \bm{ \mathrm{K} }^{\mathrm{p}}_{(l),\bm V} \\
\bm{ \mathrm{K} }^{\mathrm{m}}_{(l),\bm U} & \bm{ \mathrm{K} }^{\mathrm{m}}_{(l),\bm P} & \bm{ \mathrm{K} }^{\mathrm{m}}_{(l),\bm V}  
\end{bmatrix},
\end{align*}
wherein
\begin{align*}
\bm{ \mathrm{K} }^{\mathrm{k}}_{(l),\bm U} := \frac{ \partial \bm{\mathrm{R}}_{(l)}^{\mathrm{k}} }{ \partial \bm U_{n+1} } =  \frac{1}{\Delta t_n} \mathbf{I}, \qquad
\bm{ \mathrm{K} }^{\mathrm{k}}_{(l),\bm V} := \frac{ \partial \bm{\mathrm{R}}_{(l)}^{\mathrm{k}} }{ \partial \bm V_{n+1} } = -\frac{1}{2} \mathbf{I},
\end{align*}
$\mathbf{I}$ is the identity matrix, and $\bm{\mathrm{O}}$ is the zero matrix. We mention that the above consistent tangent matrix has a $3\times 3$ block structure. Due to the particular forms of $\bm{ \mathrm{K} }^{\mathrm{k}}_{(l),\bm U}$ and $\bm{ \mathrm{K} }^{\mathrm{k}}_{(l),\bm V}$, the Newton-Raphson solution procedure can be consistently  reduced to a two-stage algorithm \cite{Liu2018}. In the first stage, the increments, $ \Delta \bm V_{n+1,(l)}$ and $ \Delta P_{n+1,(l)}$, are obtained by solving the linear system,
\begin{align}
\label{eq:newton-raphson-2x2-matrix-problem}
\begin{bmatrix}
\bm{ \mathrm{K} }^{\mathrm{m}}_{(l),\bm V} + \frac{\Delta t_n}{2} \bm{ \mathrm{K} }^{\mathrm{m}}_{(l),\bm U} & \bm{ \mathrm{K} }^{\mathrm{m}}_{(l),\bm P} \\
\bm{ \mathrm{K} }^{\mathrm{p}}_{(l),\bm V} + \frac{\Delta t_n}{2} \bm{ \mathrm{K} }^{\mathrm{p}}_{(l),\bm U} & \bm{\mathrm{O}}
\end{bmatrix}
\begin{bmatrix}
\Delta \bm V_{n+1,(l)} \\
\Delta P_{n+1,(l)}
\end{bmatrix}
= -
\begin{bmatrix}
\bm{ \mathrm{R} }_{(l)}^{\mathrm{m}} - \Delta t_n \bm{ \mathrm{K} }^{\mathrm{m}}_{(l),\bm U} \bm{\mathrm{R}}_{(l)}^{\mathrm{k}} \\
\bm{ \mathrm{R} }_{(l)}^{\mathrm{p}} - \Delta t_n \bm{ \mathrm{K} }^{\mathrm{p}}_{(l),\bm U} \bm{\mathrm{R}}_{(l)}^{\mathrm{k}} 
\end{bmatrix}.
\end{align}
In the second stage, the increment $ \Delta \bm U_{n+1,(l)}$ is obtained by the updating formula
\begin{align}
\label{eq:Delta U}
\Delta \bm U_{n+1,(l)} = \Delta t_n \left( \frac{1}{2} \Delta \bm V_{n+1,(l)} - \bm{\mathrm{R}}_{(l)}^{\mathrm{k}} \right). 
\end{align}
To simplify the subsequent discussion, we introduce the following notations
\begin{align}
\label{eq:tangent matrices}
\bm{\mathrm{A}}_{(l)} := \bm{ \mathrm{K} }^{\mathrm{m}}_{(l),\bm V} + \frac{\Delta t_n}{2} \bm{ \mathrm{K} }^{\mathrm{m}}_{(l),\bm U}, \qquad \bm{\mathrm{B}}_{(l)} := \bm{ \mathrm{K} }^{\mathrm{m}}_{(l),\bm P}, \qquad \bm{\mathrm{C}}_{(l)} := \bm{ \mathrm{K} }^{\mathrm{p}}_{(l),\bm V} + \frac{\Delta t_n}{2} \bm{ \mathrm{K} }^{\mathrm{p}}_{(l),\bm U}.
\end{align}
It has been shown that $\bm{\mathrm{R}}_{(l)}^{\mathrm{k}} = \bm 0$ for $l \geq 2$ \cite[Appendix~B]{Liu2018}, and we set $\bm{\mathrm{R}}_{(l)}^{\mathrm{k}} = \bm 0$ on the right-hand side of \eqref{eq:newton-raphson-2x2-matrix-problem} throughout the Newton-Raphson iteration to simplify the implementation. The procedure for solving the above nonlinear algebraic equations in a time sub-interval $\mathcal I_n$ can be summarized as the following predictor multi-corrector algorithm. Its stopping criteria involve the maximum number of iterations $l_{\mathrm{max}}$, the relative tolerance $\mathrm{tol}_{\mathrm{R}}$,  and the absolute tolerance $\mathrm{tol}_{\mathrm{A}}$.

\noindent \textbf{Predictor stage:}
\begin{enumerate}
\item Set $\bm Y_{n+1,(0)} = \bm Y_{n}$.
\end{enumerate}

\noindent \textbf{Multi-corrector stage:} Repeat the following steps for $l = 1, ..., l_{\mathrm{max}}$.
\begin{enumerate}
\item Assemble the residual vectors $\bm{\mathrm{R}}_{(l)}^{m}$ and $\bm{\mathrm{R}}_{(l)}^{\mathrm{p}}$ using $\bm Y_{n+1,(l)}$ and $\bm Y_{n, (l)}$.

\item Let $\Vert \bm{\mathrm{R}}_{(l)} \Vert _{\mathfrak l^2}$ denotes the $\mathfrak l^2 \operatorname{-norm}$ of the residual vector. If one of the criteria,
\begin{align*}
\frac{\Vert \bm{\mathrm{R}}_{(l)} \Vert _{\mathfrak l^2} }{ \Vert \bm{\mathrm{R}}_{(0)} \Vert _{\mathfrak l^2} } \leq \mathrm{tol}_{\mathrm{R}} \quad \mbox{ and } \quad \Vert \bm{\mathrm{R}}_{(l)} \Vert _{\mathfrak l^2} \leq \mathrm{tol}_{\mathrm{A}},
\end{align*}
is satisfied, set the solution vector at the time step $t_{n+1}$ as $\bm Y_{n+1} = \bm Y_{n+1,(l-1)}$, and exit the multi-corrector stage; otherwise continue to step 3.

\item Assemble the tangent matrices $\bm{\mathrm{A}}_{(l)}$, $\bm{\mathrm{B}}_{(l)}$, and $\bm{\mathrm{C}}_{(l)}$ according to \eqref{eq:tangent matrices}.

\item Solve the linear system
\begin{align}
\label{eq:predictor-corrector-linear-system}
\begin{bmatrix}
\bm{\mathrm{A}}_{(l)} & \bm{\mathrm{B}}_{(l)} \\
\bm{\mathrm{C}}_{(l)} & \bm{\mathrm{O}}
\end{bmatrix}
\begin{bmatrix}
\Delta \bm V_{n+1,(l)} \\
\Delta P_{n+1, (l)}
\end{bmatrix}
= -
\begin{bmatrix}
\bm{\mathrm{R}}_{(l)}^{\mathrm{m}} \\
\bm{\mathrm{R}}_{(l)}^{\mathrm{p}}
\end{bmatrix}
\end{align}
for $\Delta \bm V_{n+1,(l)}$ and $\Delta P_{n+1,(l)}$.

\item Obtain $\Delta \bm U_{n+1,(l)}$ according to the relation \eqref{eq:Delta U}.

\item Update the solution vector as $\bm Y_{n+1,(l)} = \bm Y_{n+1,(l)} + \Delta \bm Y_{n+1,(l)}$.
\end{enumerate}

\begin{remark}
In the scenario of structure-preserving time integration, it is critical to solve the nonlinear problem accurately in each time step. Correspondingly, it is necessary to solve the linear problem \eqref{eq:predictor-corrector-linear-system} accurately. Otherwise the Newton-Raphson iteration may stagnate due to inaccurate solutions of the incremental. In this work, we invoke the iterative solver preconditioned by the nested block preconditioner \cite{Liu2019} as well as sparse direct solver implemented in \texttt{MUMPS} \cite{Amestoy2001} to solve the linear system, both of which can effectively solve the linear system.
\end{remark}

\section{Numerical examples}
\label{sec:numerical_examples}
In our numerical investigation, unless otherwise specified, the discrete pressure function space is generated by $k$-refinement to achieve the highest possible continuity, and the degree elevation is adopted with $\mathsf{a} = 1$ and $\mathsf{b}=0$ to construct the velocity function space; the Gonzalez discrete gradient is used to construct the algorithmic stress; we use $\mathsf p + \mathsf a + 2$ Gaussian quadrature points in each direction; we use $\mathrm{tol}_\mathrm{R} = 10^{-10}$, $\mathrm{tol}_\mathrm{A} = 10^{-10}$, $l_{\mathrm{max}} = 10$ as the stopping criteria in the predictor multi-corrector algorithm; the meter-kilogram-second system of units is used.

\subsection{Numerical robustness}
\label{sec:numerical-robustness}
In this first example, we assess the robustness of the algorithmic stress. The assessing procedure can be summarized briefly. A fixed initial deformation state characterized by $\bm F_1$ is introduced, and the deformed state is given by $\bm F_2 = \bm F_1 + \xi \bm D$, with $\bm D$ being a two-point tensor representing a prescribed deformation state. Varying the parameter $\xi \in \mathbb R_{+}$, one may control the difference between the two deformation states. We are concerned with cases when $\xi$ approaches $0$, mimicking the behavior when the time step size gets small or when the body reaches a steady state. The stress enhancement behaves as a $0/0$-type indeterminate form in the limit, and their numerical evaluation will become numerically unstable eventually. In practice, if the denominator is smaller than $\mathrm{tol}_{\mathrm{B}}$, the stress enhancement will be turned off. One goal of this example is to determine the value of $\mathrm{tol}_{\mathrm{B}}$ by examining the quotient formula in different scenarios.

For the stretch-based material models, the spectral decomposition plays a role in the evaluation of all constitutive relations. The conventional approaches adopt Cardano's formula with a perturbation technique to handle the case of nearly identical eigenvalues \cite{Miehe1993,Simo1992c}. It was noticed that the perturbed principal stretches lead to the numerical blow-up when $\|\bm C_2 - \bm C_1\|$ is of the order $10^{-2}$ \cite{Mohr2008}. In order words, the perturbed technique leads to an early occurrence of numerical instability, which will inevitably destroy the conservation properties. A special technique was developed to alleviate this issue in the context of conserving integrators \cite{Mohr2008}. In this study, we adopt the spectral decomposition algorithm proposed in \cite{Scherzinger2008}. Its accuracy is comparable with that of LAPACK \cite{Harari2022,Scherzinger2008}, and numerical artifacts like the perturbation technique is unneeded in this algorithm.

\subsubsection{Robustness of the quotient formula and the spectral decomposition algorithm}
We consider the following three-dimension deformation states inspired by the two-dimensional benchmark used in \cite{Mohr2008}. The initial deformation gradient $\bm{F}_1$ is given by
\begin{align*}
\bm{F}_1 = \left[
\begin{array}{ccc}
1.5 & 0.0 & 0.0 \\
0.1 & 0.8 & 0.0 \\
0.0 & 0.0 & 1.0
\end{array}
\right],
\end{align*}
and we consider three deformation types as follows,
\begin{align*}
\bm{D}_{\mathrm{comp}} = \left[
\begin{array}{ccc}
0.0 & 0.0 & 0.0 \\
0.0 & -1.0 & 0.0 \\
0.0 & 0.0 & 1.0
\end{array}
\right], \quad
\bm{D}_{\mathrm{shear}} = \left[
\begin{array}{ccc}
0.0 & 1.0 & 0.0 \\
0.0 & 0.0 & 0.0 \\
0.0 & 0.0 & 1.0
\end{array}
\right], \quad
\bm{D}_{\mathrm{mix}} = \left[
\begin{array}{ccc}
0.0 & 1.0 & 0.0 \\
0.0 & -1.0 & 0.0 \\
0.0 & 0.0 & 1.0
\end{array}
\right],
\end{align*}
representing the compression, simple shear, and mixed deformation states. As a special case of the Ogden mode, the neo-Hookean model is utilized as it can be represented in terms of principal invariants as well. The shear modulus is taken to be $5000$. In Figure \ref{fig:numerical_limit_behavior}, the norm of $\bm S_{\mathrm{ich~enh}}^{\mathrm{gon}}$ is plotted against $\|\bm C_2 - \bm C_1\|$ using both the invariant- and stretch-based formulation. The numerical behaviors for the three imposed deformation states are similar. We know that the analytic value of $\|\bm S_{\mathrm{ich~enh}}^{\mathrm{gon}}\|$ goes to zero as $\|\bm C_2 - \bm C_1\|$ approaches zero \cite{Mohr2008}. The numerical value of $\|\bm S_{\mathrm{ich~enh}}^{\mathrm{gon}}\|$ monotonically decreases when $10^{-4}<\|\bm C_2 - \bm C_1\|$; for $10^{-5}<\|\bm C_2 - \bm C_1\|<10^{-4}$, the floating-point calculations start to result in oscillations; when the value of $\|\bm C_2 - \bm C_1\|$ further decreases, the magnitude of the oscillation amplifies, and the value of $\|\bm S_{\mathrm{ich~enh}}^{\mathrm{gon}}\|$ ceases to approach its analytic limit. The oscillation is due to the error in floating-point calculations and indicates the value of the stress enhancement is getting non-reliable. Therefore, we set $\mathrm{tol}_{\mathrm B}$ to be $10^{-10}$ in our calculations. This means that the algorithmic stress $\bm{S}_{\mathrm{ich~alg}}$ will be switched to $\bm{S}_m$ once the condition $\bm{Z}_n :\bm{Z}_n \leq 10^{-10}$ is detected.

We may observe from Figure \ref{fig:numerical_limit_behavior} that the invariant-based model behaves similarly to the stretch-based model. This means that the spectral decomposition algorithm \cite{Scherzinger2008} is accurate enough that it does not incur an early occurrence of numerical instabilities. In this regard, the spectral decomposition algorithm \cite{Scherzinger2008} can be safely used without considering additional strategies \cite{Mohr2008}.

\begin{figure}
\begin{center}
\begin{tabular}{ccc}
\includegraphics[angle=0, trim=60 90 160 120, clip=true, scale = 0.14]{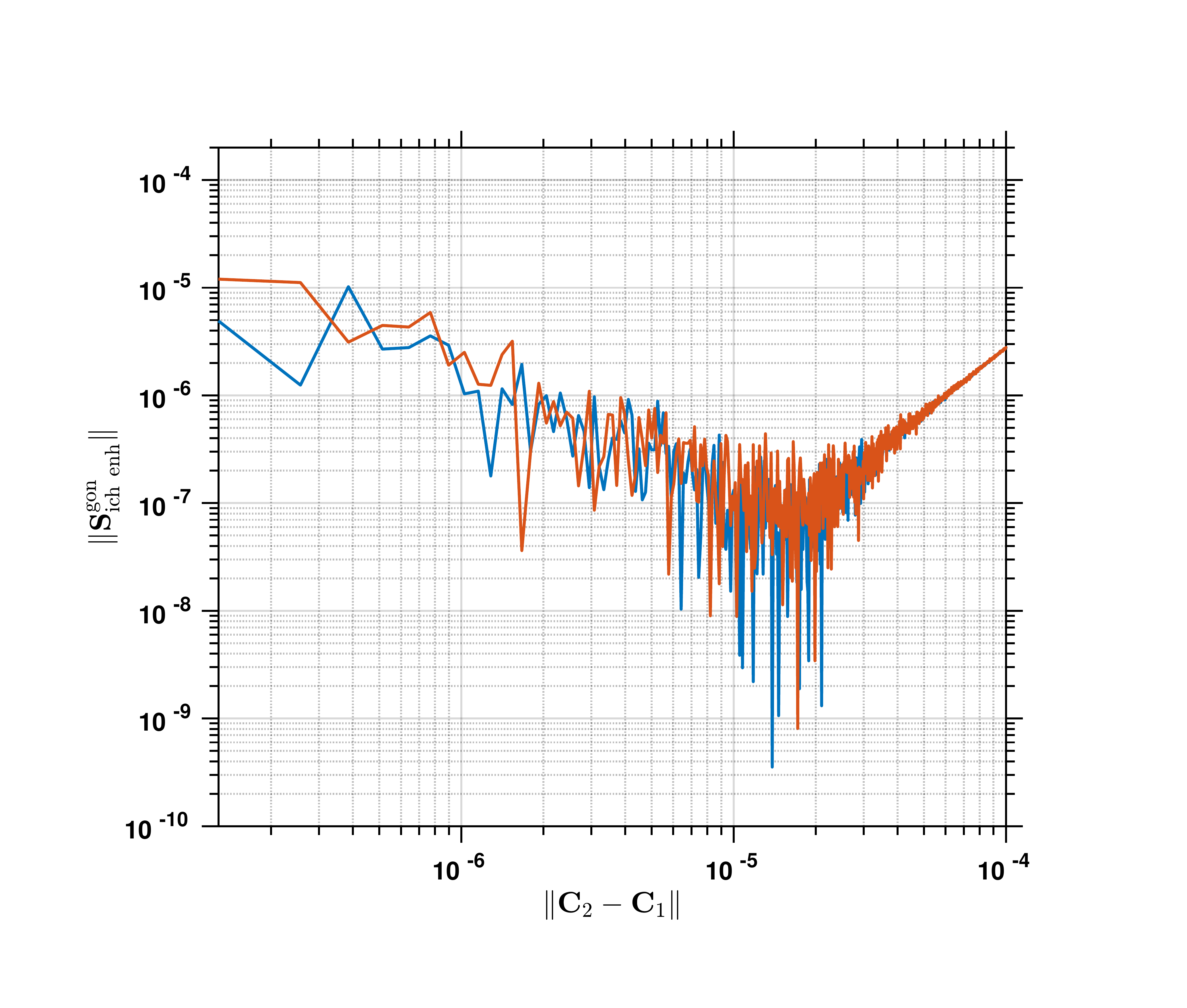} &
\includegraphics[angle=0, trim=60 90 160 120, clip=true, scale = 0.14]{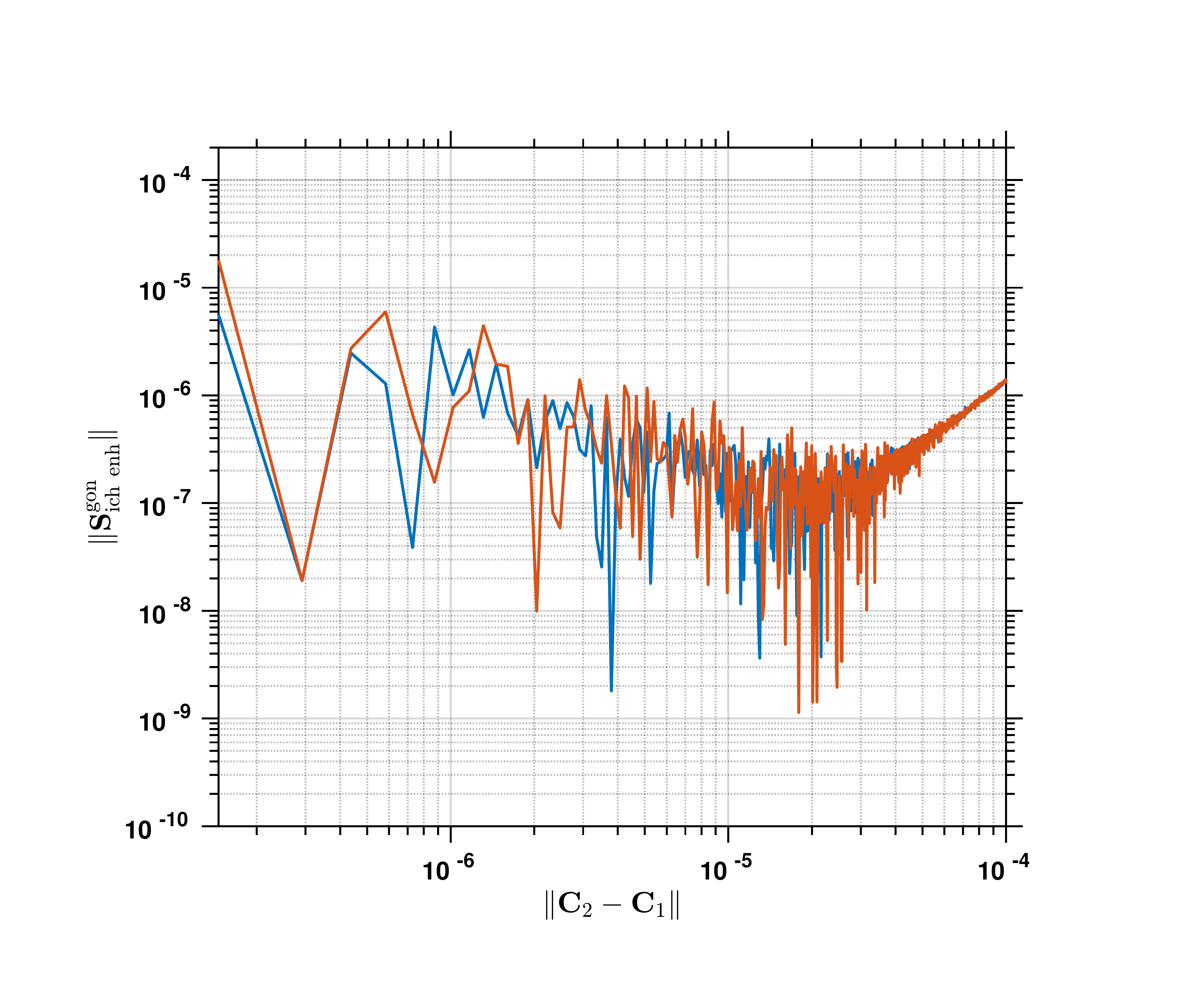} & 
\includegraphics[angle=0, trim=60 90 160 120, clip=true, scale = 0.14]{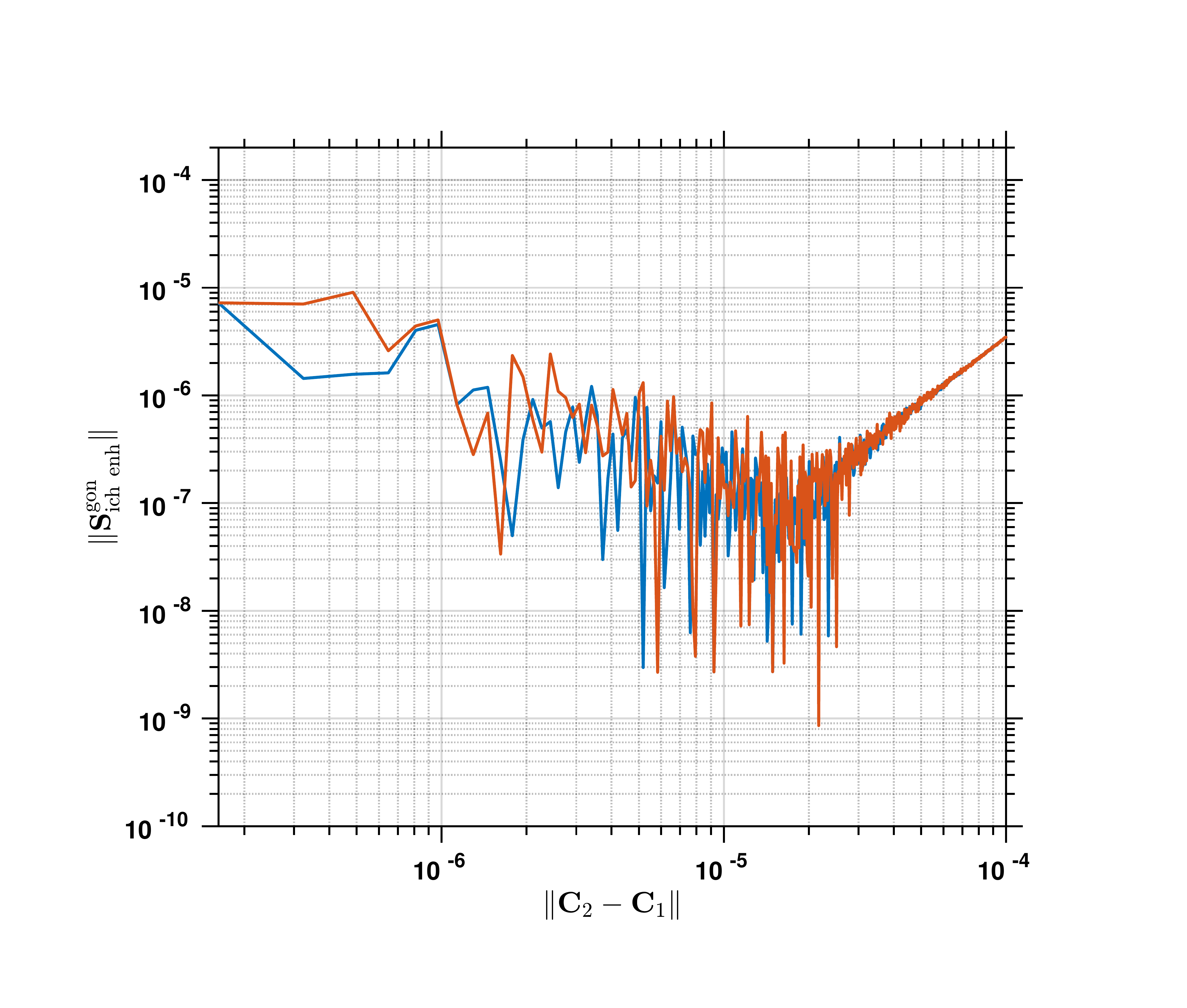} \\
(a) & (b) & (c) \\
\multicolumn{3}{c}{ \includegraphics[angle=0, trim=160 190 160 910, clip=true, scale = 0.35]{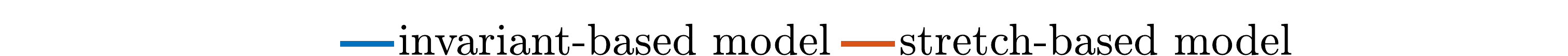} }
\end{tabular}
\caption{Numerical limit behavior of the norm of the stress enhancement $\bm S_{\mathrm{ich~enh}}^{\mathrm{gon}}$ with respect to $\| \bm{C}_2-\bm{C}_1 \|$ for (a) compression, (b) simple shear, and (c) mixed deformation.}
\label{fig:numerical_limit_behavior}
\end{center}
\end{figure}

\begin{figure}
\begin{center}
\begin{tabular}{cc}
\includegraphics[angle=0, trim=50 200 50 100, clip=true, scale = 0.20]{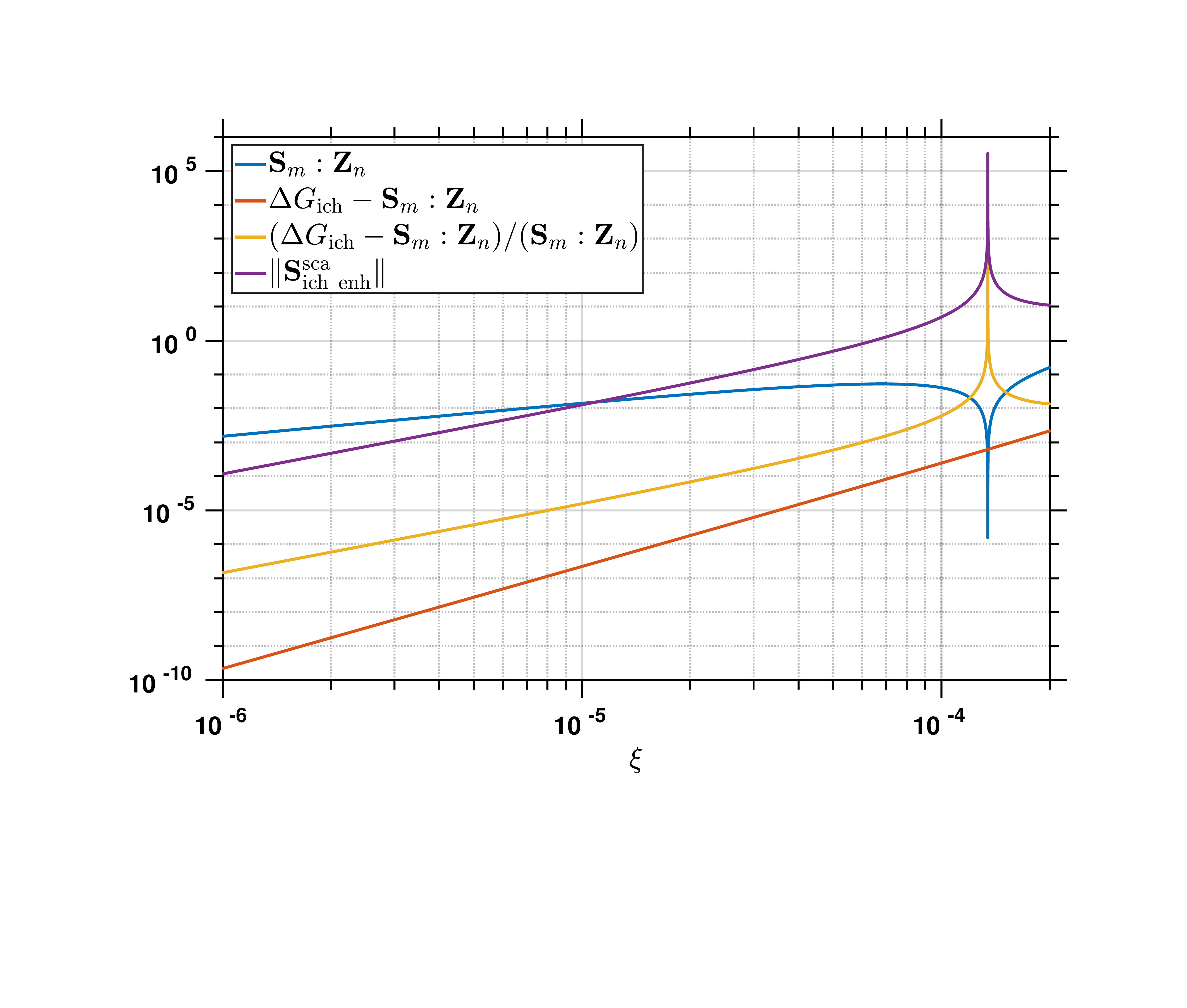} &
\includegraphics[angle=0, trim=50 200 50 100, clip=true, scale = 0.20]{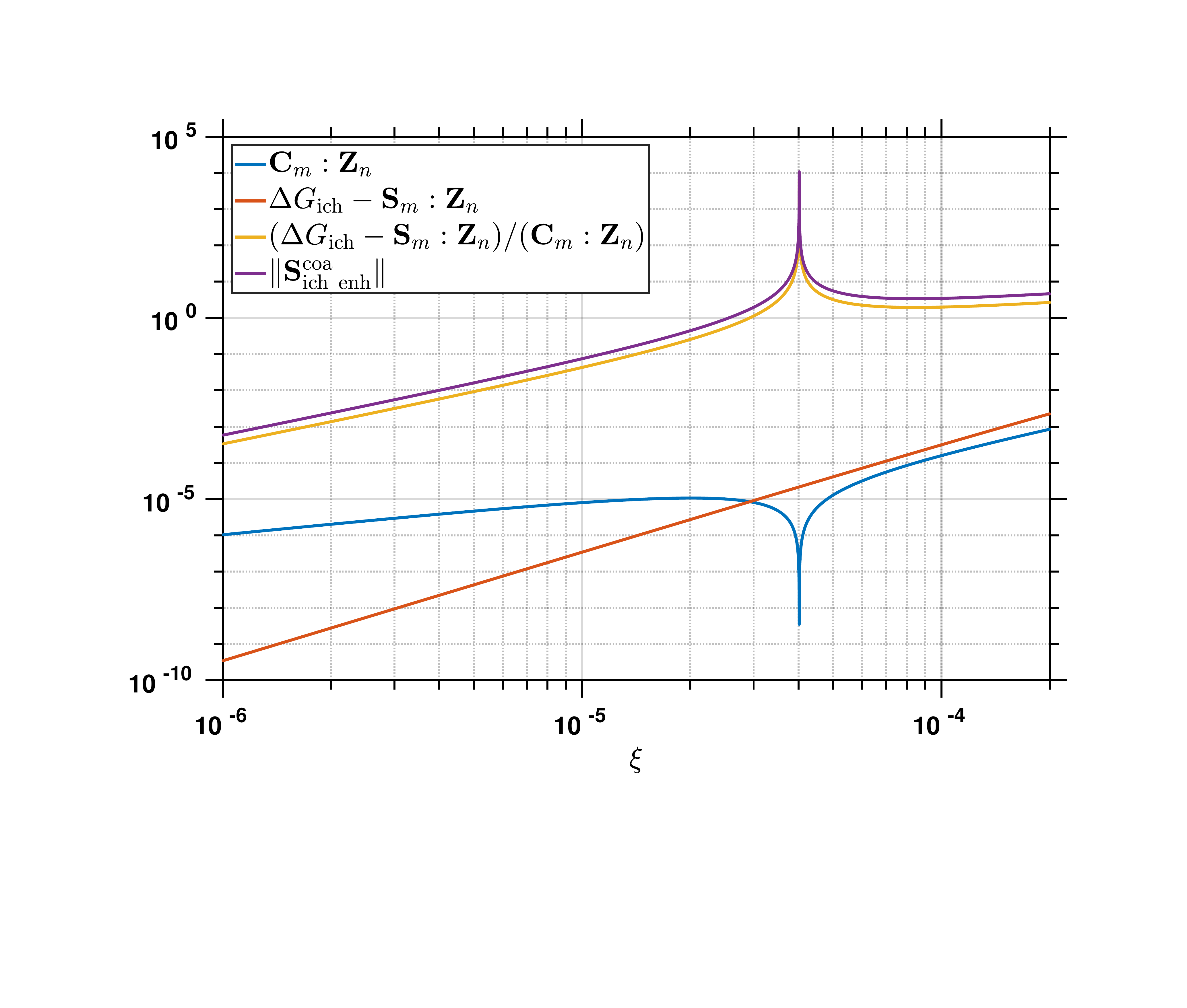} \\
(a) & (b)
\end{tabular}
\caption{Numerical behavior of the enhanced stress for (a) the scaled mid-point gradient \eqref{eq:scaled-mid-point-formula} and (b) the discrete gradient \eqref{eq:a new coaxial gradient}.}
\label{fig:potential_problems_sca_coa}
\end{center}
\end{figure}

\subsubsection{Robustness of the scaled mid-point formula}
The scaled mid-point formula \eqref{eq:scaled-mid-point-formula} and its extension \eqref{eq:a new coaxial gradient} were analyzed previously in Section \ref{sec:temporal_discretization}. A major concern is that the denominators in their quotient formulas are not of metric nature, which tends to incur numerical instability. Here we provide two cases in which the two discrete gradients behave poorly. The two cases are obtained from our practical calculations and are thus non-trivial compared with the previous three cases. The initial deformation state $\bm{F}_{1}$ and the pre-defined deformation state $\bm{D}$ are set as
\begin{align*}
\bm{F}_1 = \left[
\begin{array}{ccc}
0.996 & 0.001 & 0.185 \\
0.0 & 1.0 & 0.002 \\
-0.069 & 0.0 & 1.008
\end{array}
\right], \quad
\bm{D} = \left[
\begin{array}{ccc}
-20.0 & 0.0 & 170.0 \\
-10.0 & -20.0 & 10.0 \\
-180.0 & 0.0 & 20.0
\end{array}
\right]
\end{align*}
for the scaled mid-point formula \eqref{eq:scaled-mid-point-formula}, and
\begin{align*}
\bm{F}_1 = \left[
\begin{array}{ccc}
0.985 & 0.0 & 0.15 \\
0.0 & 1.0 & 0.0 \\
-0.032 & 0.0 & 1.003
\end{array}
\right], \qquad
\bm{D} = \left[
\begin{array}{ccc}
60.0 & 0.0 & 170.0 \\
0.0 & 10.0 & 0.0 \\
-100.0 & 0.0 & 10.0
\end{array}
\right]
\end{align*}
for the formula \eqref{eq:a new coaxial gradient}. Again, the neo-Hookean model formulated in principal invariants is used with identical material parameters. Different components of the formulas are plotted in Figure \ref{fig:potential_problems_sca_coa} with $\xi$. It can be observed that the evaluation of the term $\Delta G_{\mathrm{ich}}-\bm{S}_m:\bm{Z}_n$ monotonically decreases with $\xi$. In contrast, the denominators are numerically non-robust. There are values of $\xi$ where the denominators $\bm S_m : \bm Z_n$ and $\bm C_m : \bm Z_n$ drop by several orders of magnitude. This consequently leads to the blow-up of the scaling factors and the stress enhancements. This pathological behavior renders the two algorithmic stresses non-robust as one needs to frequently encounter cases when the denominator gets below the tolerance $\mathrm{tol}_{\mathrm B}$. We, therefore, do not use those formulas in the following examples.

\subsection{Rotating propeller}
Here we assess the structure-preserving integrator combined with the grad-div stabilization. The problem setting is summarized in Table \ref{table:propeller}. The geometry is inspired by the benchmark problem studied in \cite{Armero2001b}, which consists of a central ring and three equally-spaced blades. Different from the original benchmark problem, we invoke NURBS for the geometrical description, which provides an exact circular ring. The complete geometrical information for the four patches, including the NURBS knot vectors and control points, can be found in the repository \cite{propeller_geo2023}. The material properties of the central ring and blades are characterized by the Ogden model, and the central ring is eight times stiffer than the blades. The motion is initiated by a body force imposed on the central ring, whose form is given by $\bm{B} \left(\bm{X}, t \right) = \tau(t) \left( \bm{E}_3 \times \bm{\varphi} \left( \bm{X}, t\right) \right)$, with
\begin{align*}
\tau(t) = 
\begin{cases}
2\tau_{\mathrm{max}} t / \bar{T}, & 0 \leq t \leq \bar{T}/2, \\
2\tau_{\mathrm{max}} \left(1-t/\bar{T} \right), &  \bar{T}/2 < t \leq \bar{T}, \\
0, & \bar{T} < t,
\end{cases}
\end{align*}
$\tau_{\mathrm{max}} = 5.6$, and $\bar{T} = 15$. Here $E_a$, for $a=1,2,3$, are fixed orthonormal bases in $\Omega_{\bm X}$.

\begin{table}[htbp]
  \centering
  \begin{tabular}{ m{.4\textwidth}   m{.45\textwidth} }
    \hline
    \begin{minipage}{.35\textwidth} 
      \includegraphics[width=1.0\linewidth, trim= 100 250 100 180, clip]{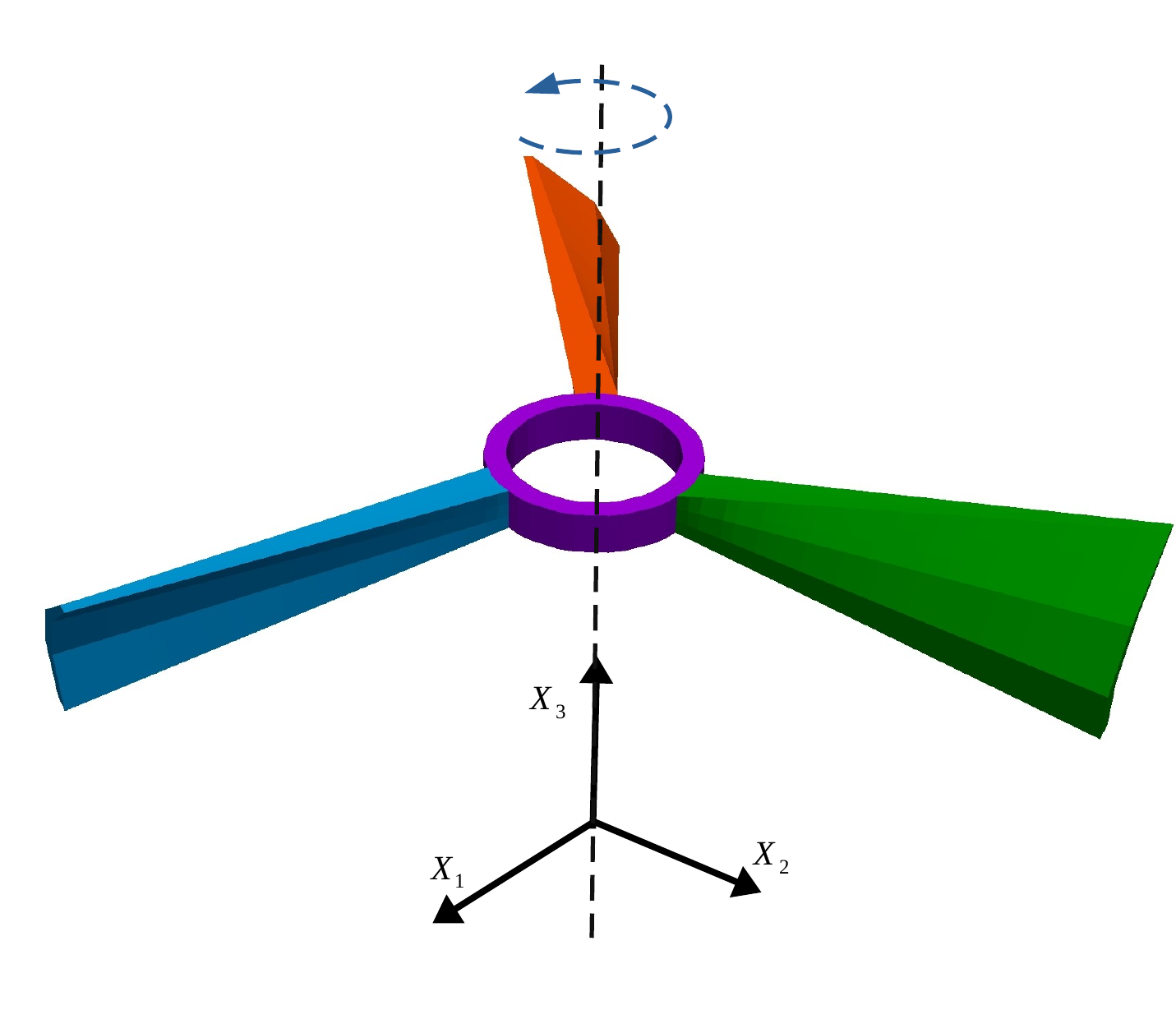}
    \end{minipage}
    &
    \begin{minipage}{.45\textwidth}
      \begin{itemize}
        \item[] Material properties:
        \item[] $\rho_0 = 8.93, \quad N =1$,
        \item[] $G_{\mathrm{ich}}( \tilde{\lambda}_1, \tilde{\lambda}_2, \tilde{\lambda}_3 ) = \sum\limits_{a=1}^{3} \sum\limits_{p=1}^{N} \frac{\mu_p}{\alpha_p}( \tilde{\lambda}_{a}^{\alpha_p} - 1 )$,
        \item[] central ring: $\alpha_1 = 2$ and  $\mu_1 = 307.68$,
        \item[] three blades: $\alpha_1 = 2$ and $\mu_1 = 38.46$.
      \end{itemize}      
    \end{minipage}   
    \\ 
    \hline
  \end{tabular}
  \caption{The three-dimensional rotating propeller: problem definition. The dashed line represents the axis of the ring corresponding to the $X_3$ axis.} 
\label{table:propeller}
\end{table}

\begin{figure}
\begin{center}
\begin{tabular}{cccc}
\includegraphics[angle=0, trim=250 160 250 160, clip=true, scale = 0.06]{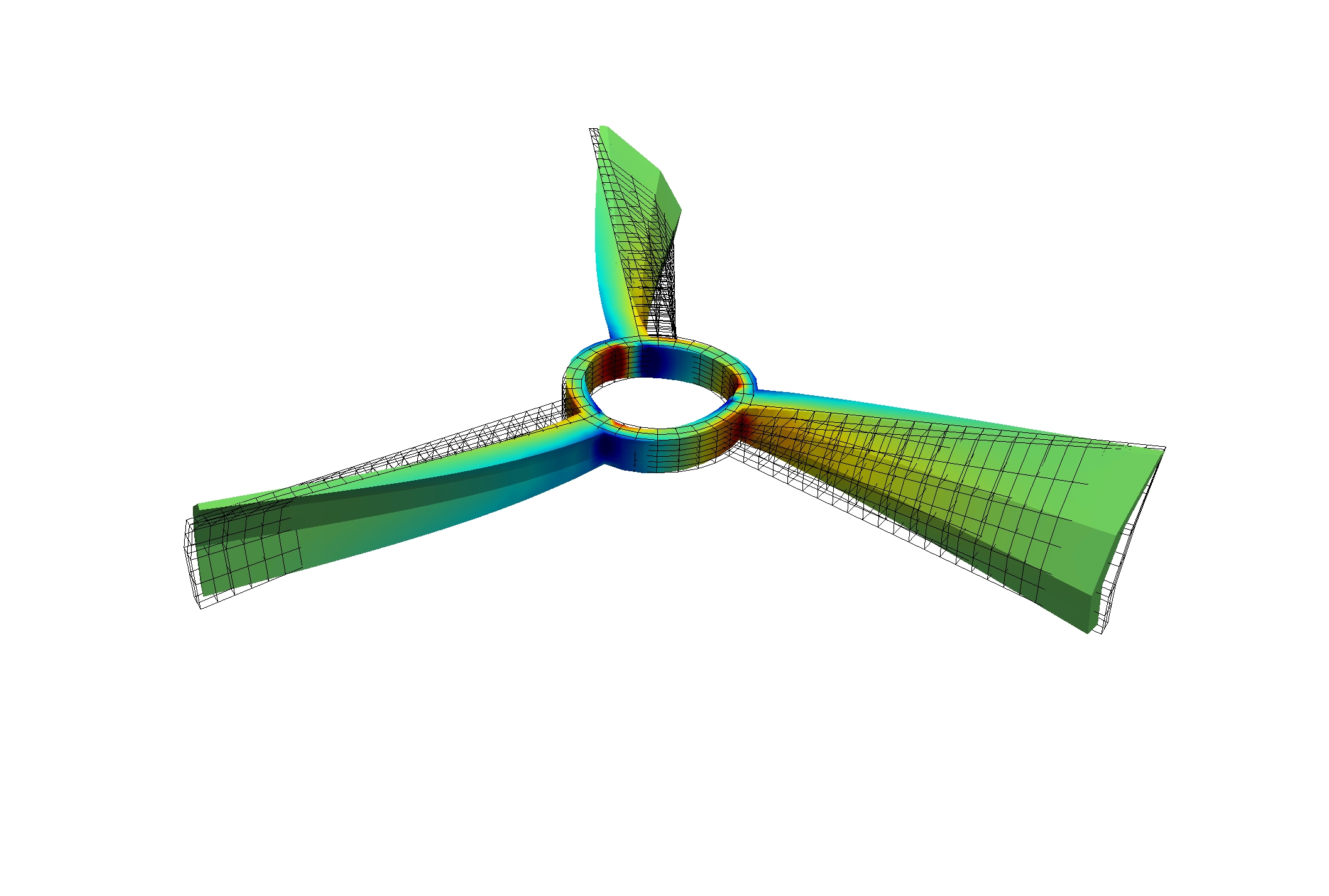} &
\includegraphics[angle=0, trim=250 160 250 160, clip=true, scale = 0.06]{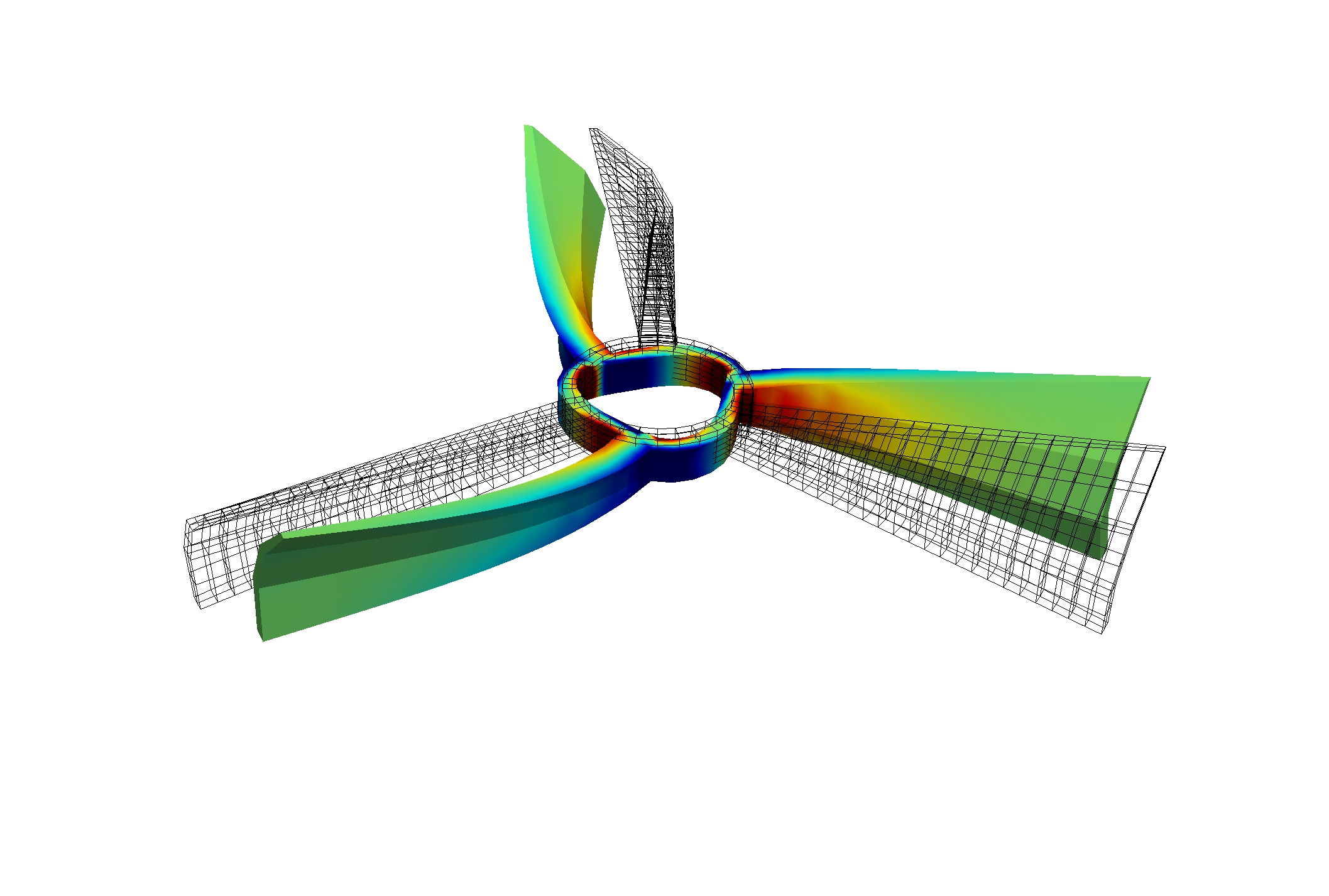} &
\includegraphics[angle=0, trim=250 160 250 160, clip=true, scale = 0.06]{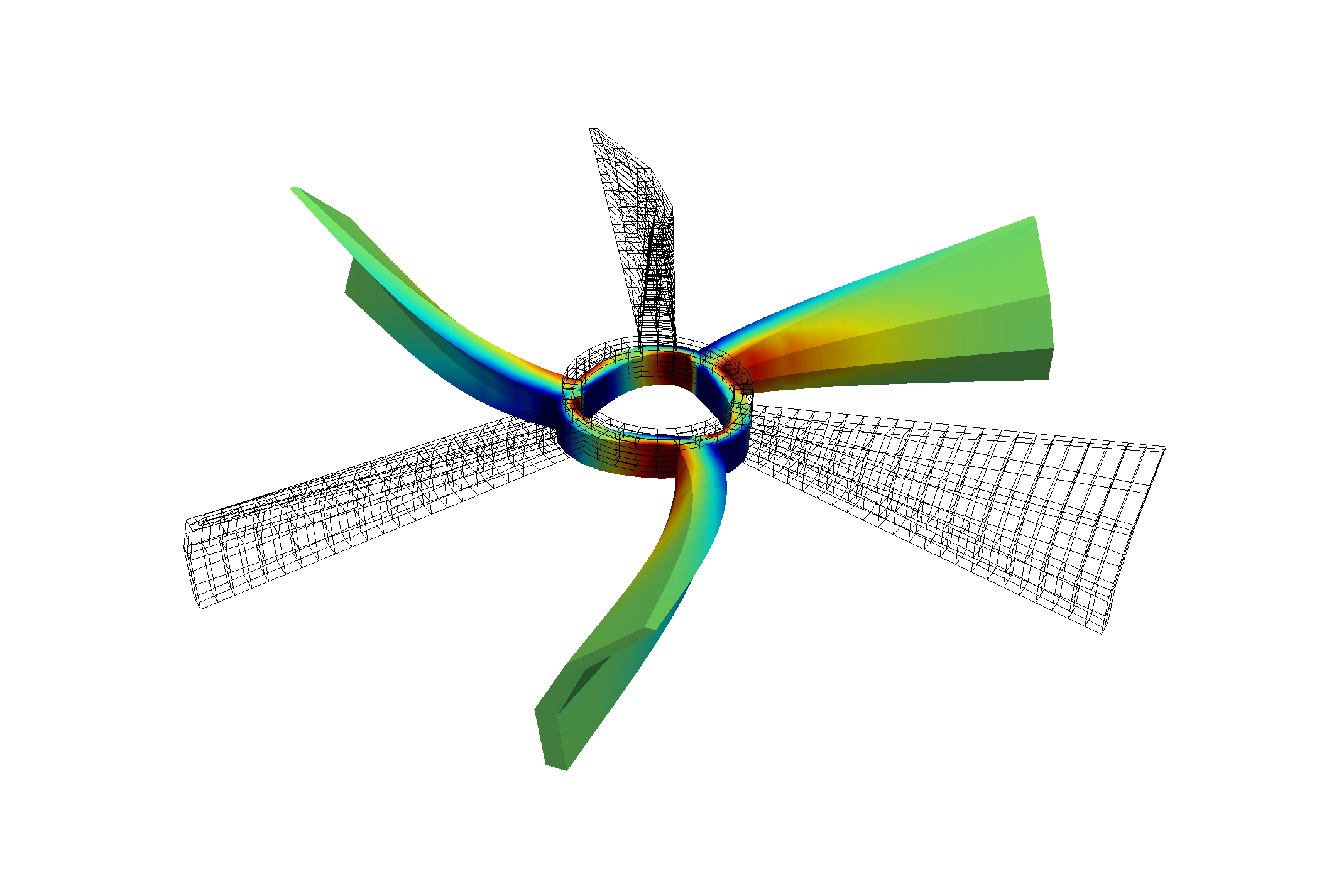} &
\includegraphics[angle=0, trim=250 160 250 160, clip=true, scale = 0.06]{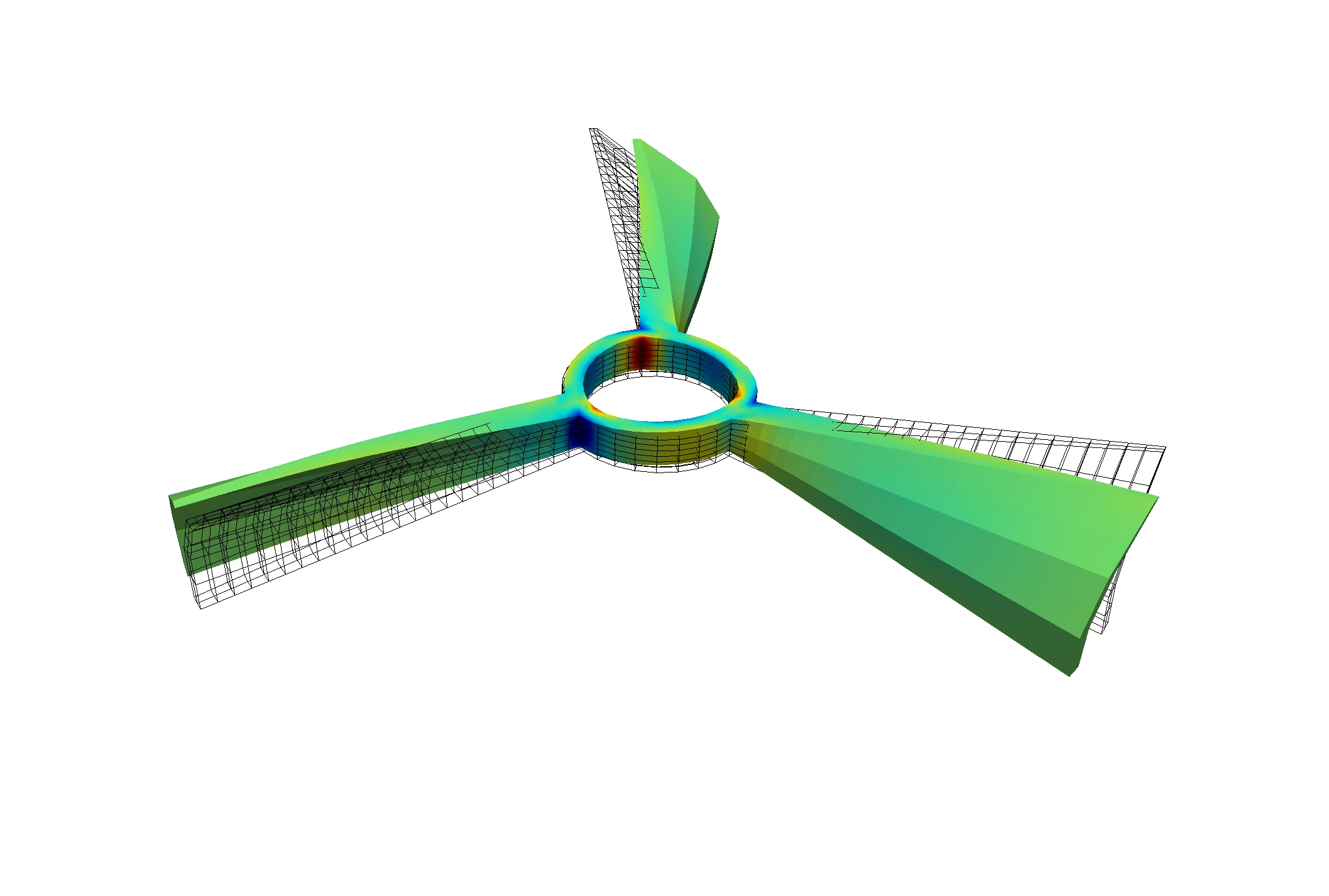} \\
$t = 3$ & $6$ & $9$ & $12$ \\
\includegraphics[angle=0, trim=250 160 250 160, clip=true, scale = 0.06]{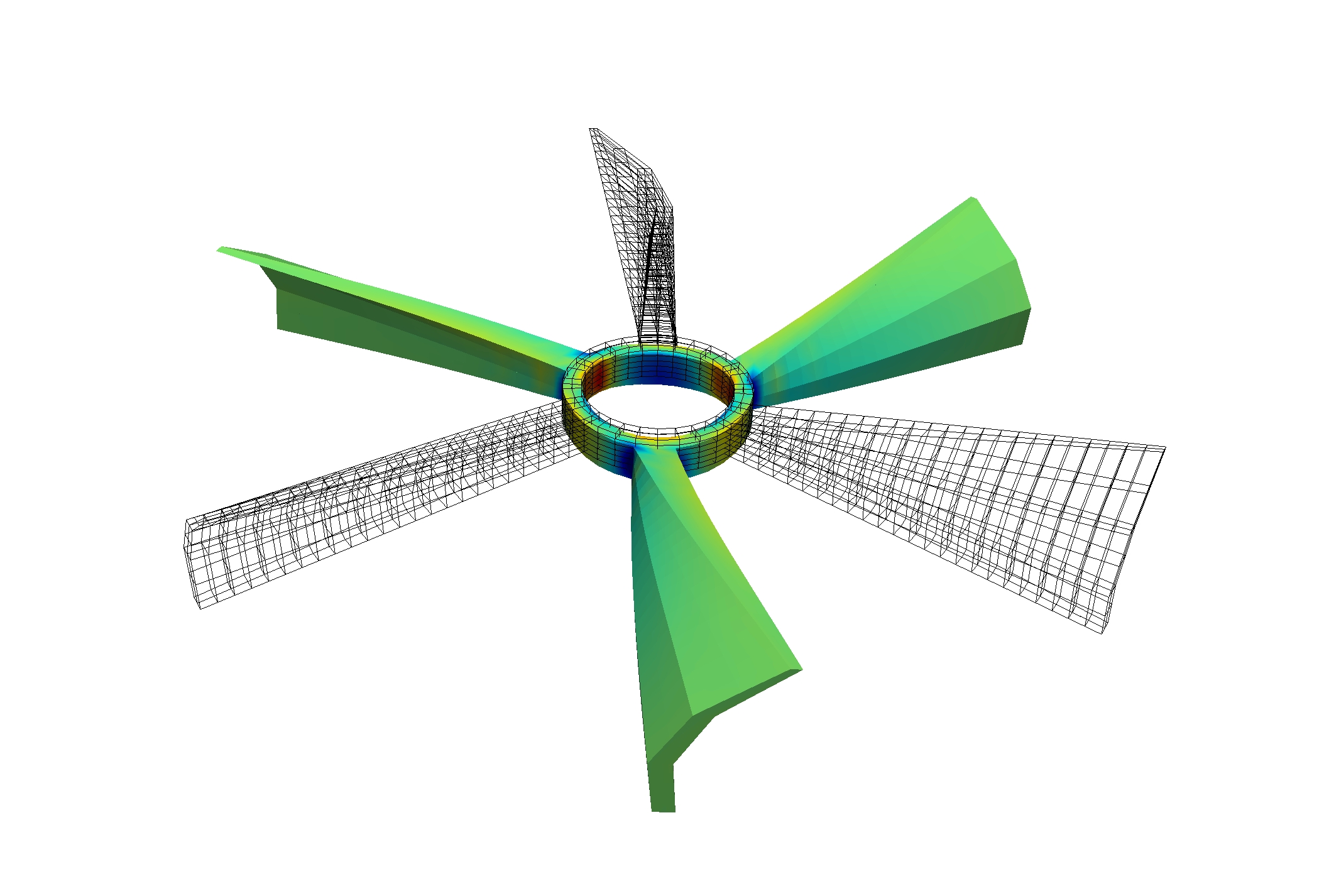} &
\includegraphics[angle=0, trim=250 160 250 160, clip=true, scale = 0.06]{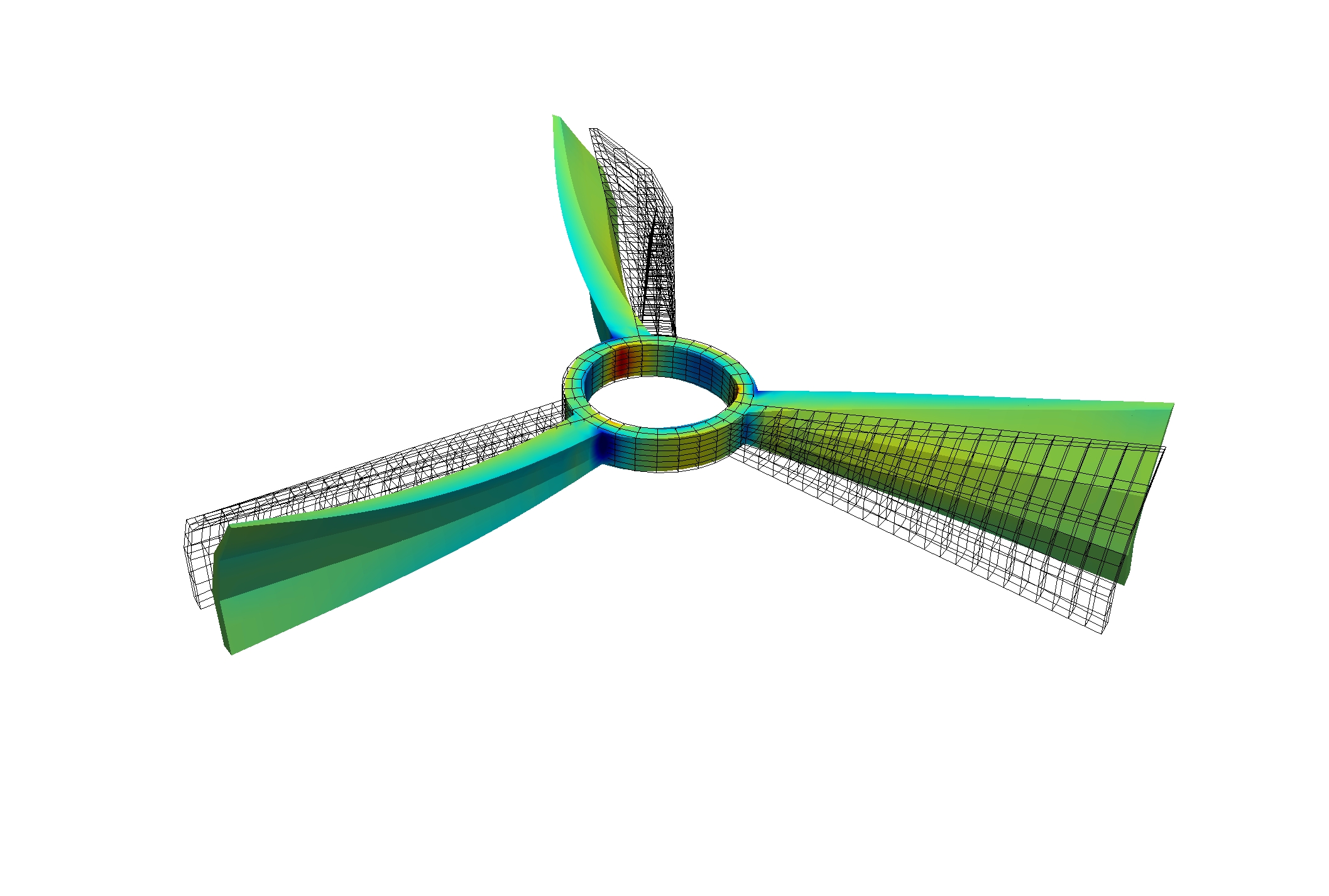} &
\includegraphics[angle=0, trim=250 160 250 160, clip=true, scale = 0.06]{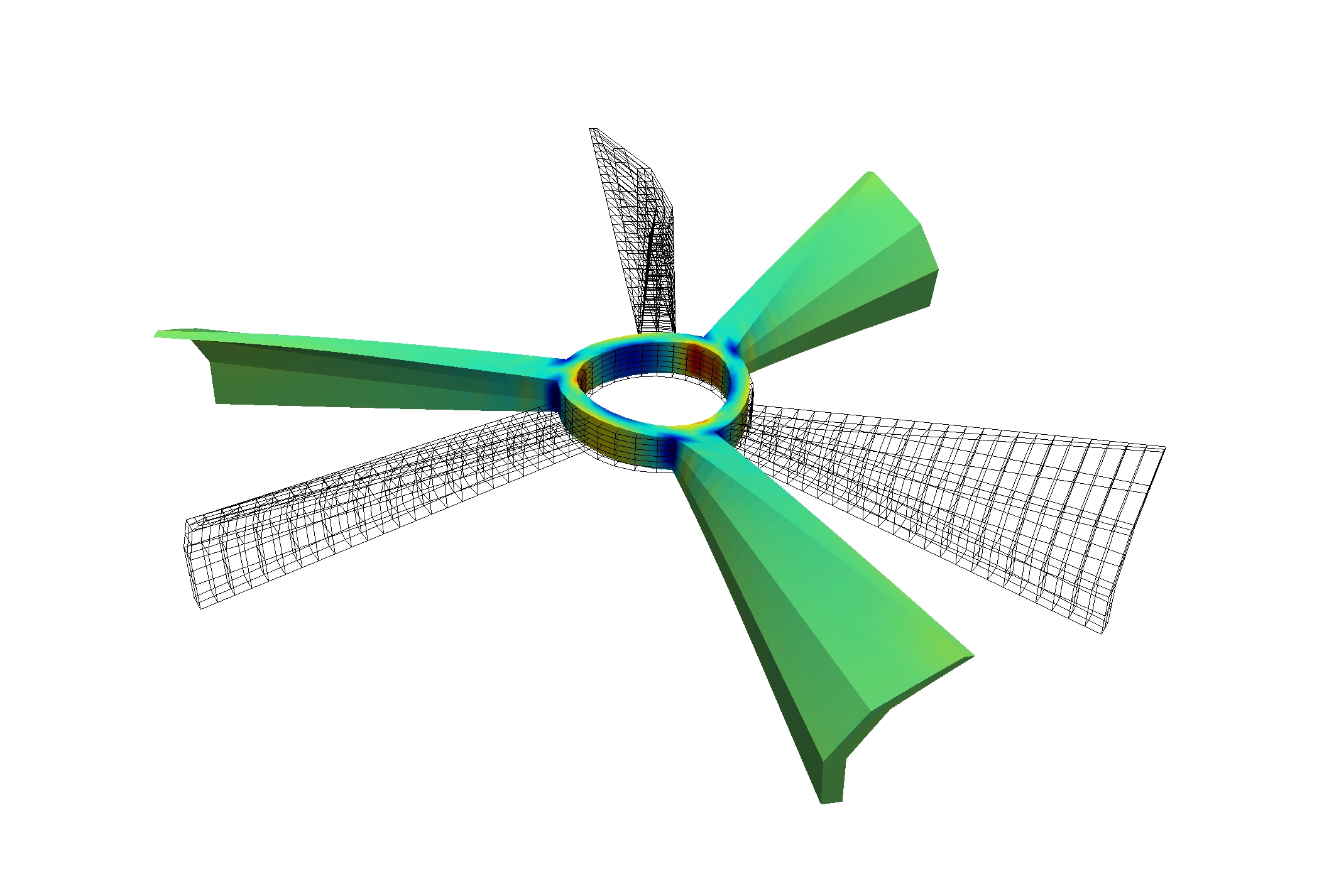} &
\includegraphics[angle=0, trim=250 160 250 160, clip=true, scale = 0.06]{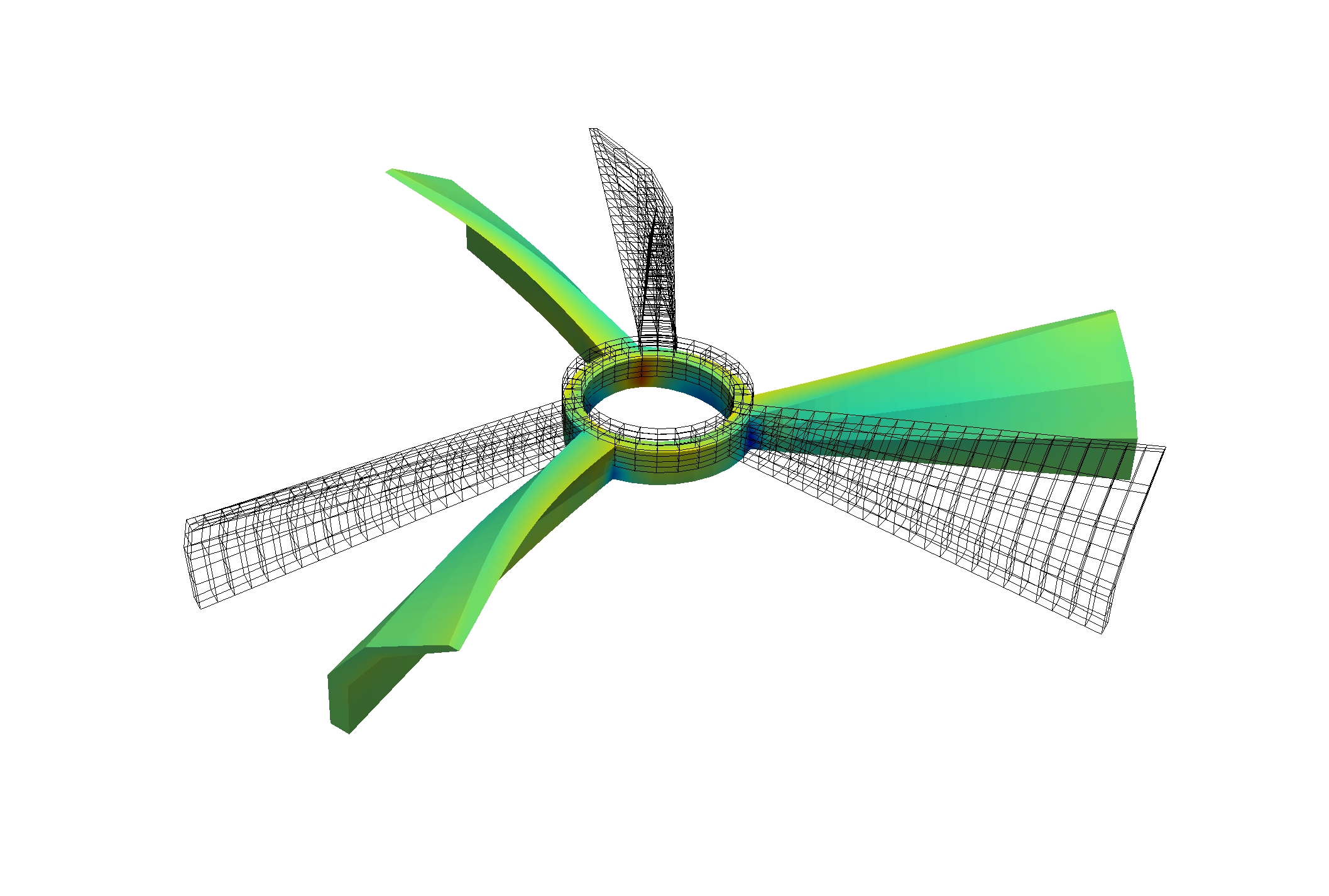} \\
$15$ & $18$ & $21$ & $24$ \\
\multicolumn{4}{c}{ \includegraphics[angle=0, trim=250 380 250 940, clip=true, scale = 0.25]{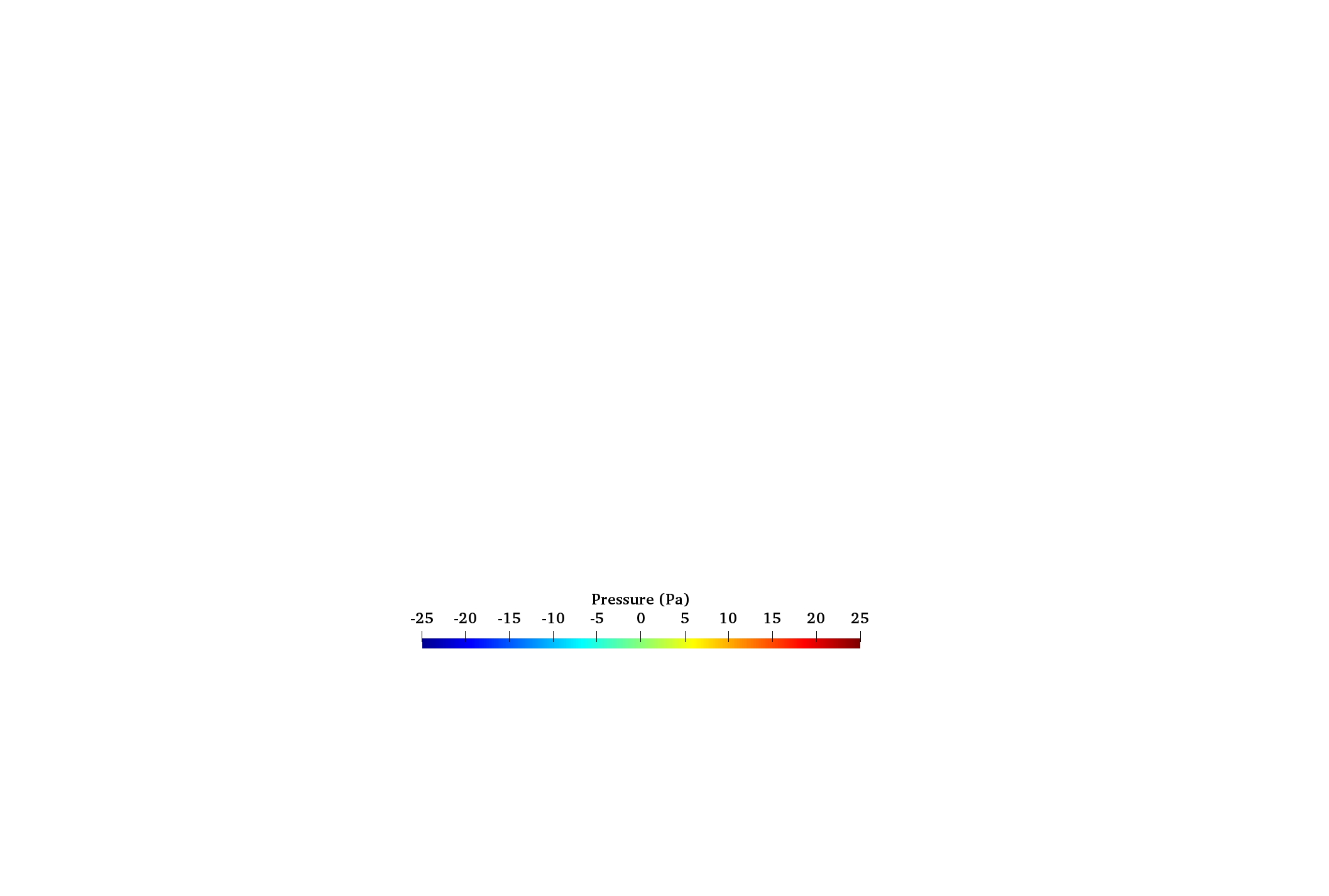} }
\end{tabular}
\caption{Snapshots of the pressure on the deformed configurations with $\Delta t = 0.05$. The black grid represents the initial undeformed configuration.}
\label{fig:propeller_deformation}
\end{center}
\end{figure}

\begin{figure}
\begin{center}
\begin{tabular}{cc}
\includegraphics[angle=0, trim=150 330 150 220, clip=true, scale = 0.1]{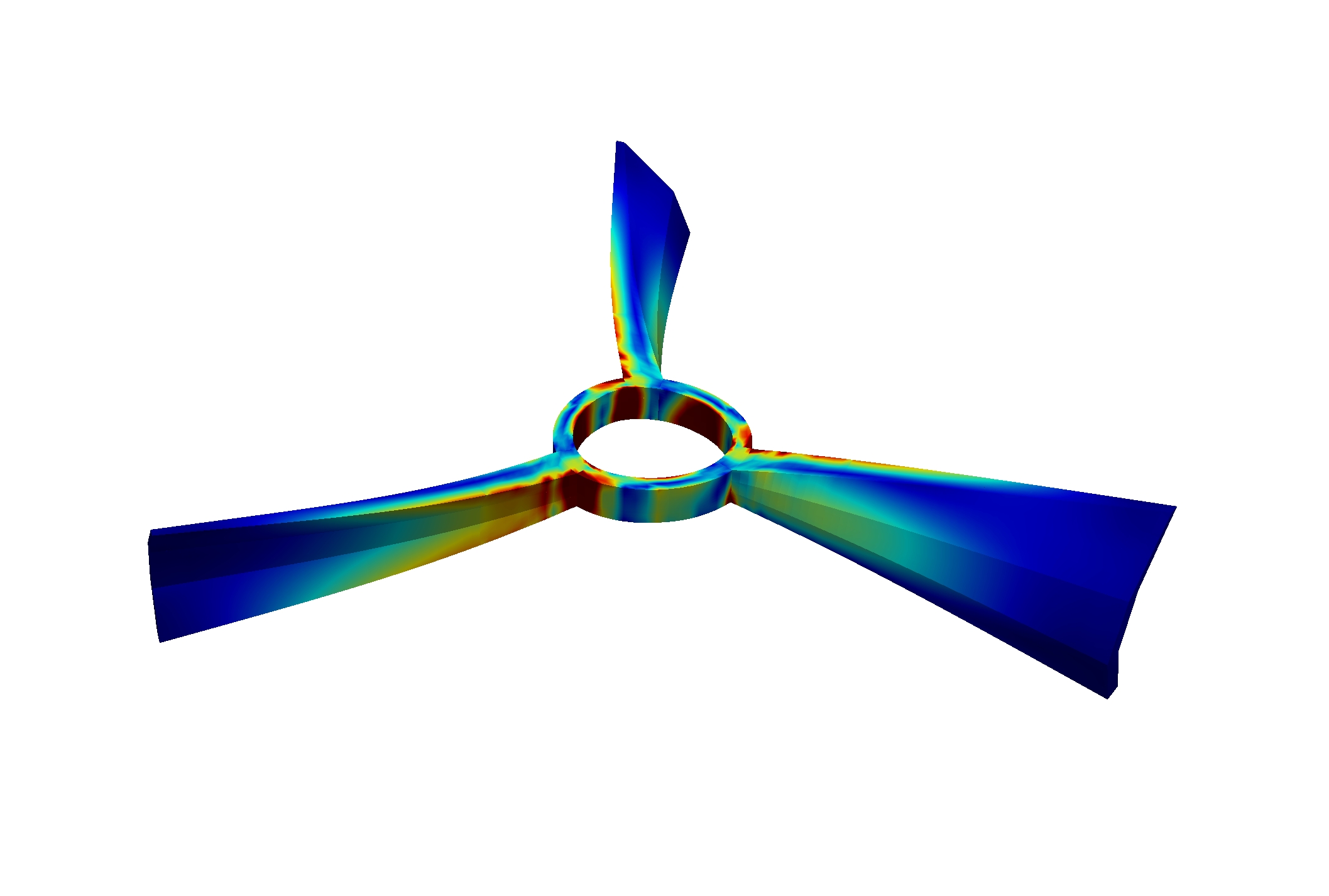} &
\includegraphics[angle=0, trim=150 330 150 220, clip=true, scale = 0.1]{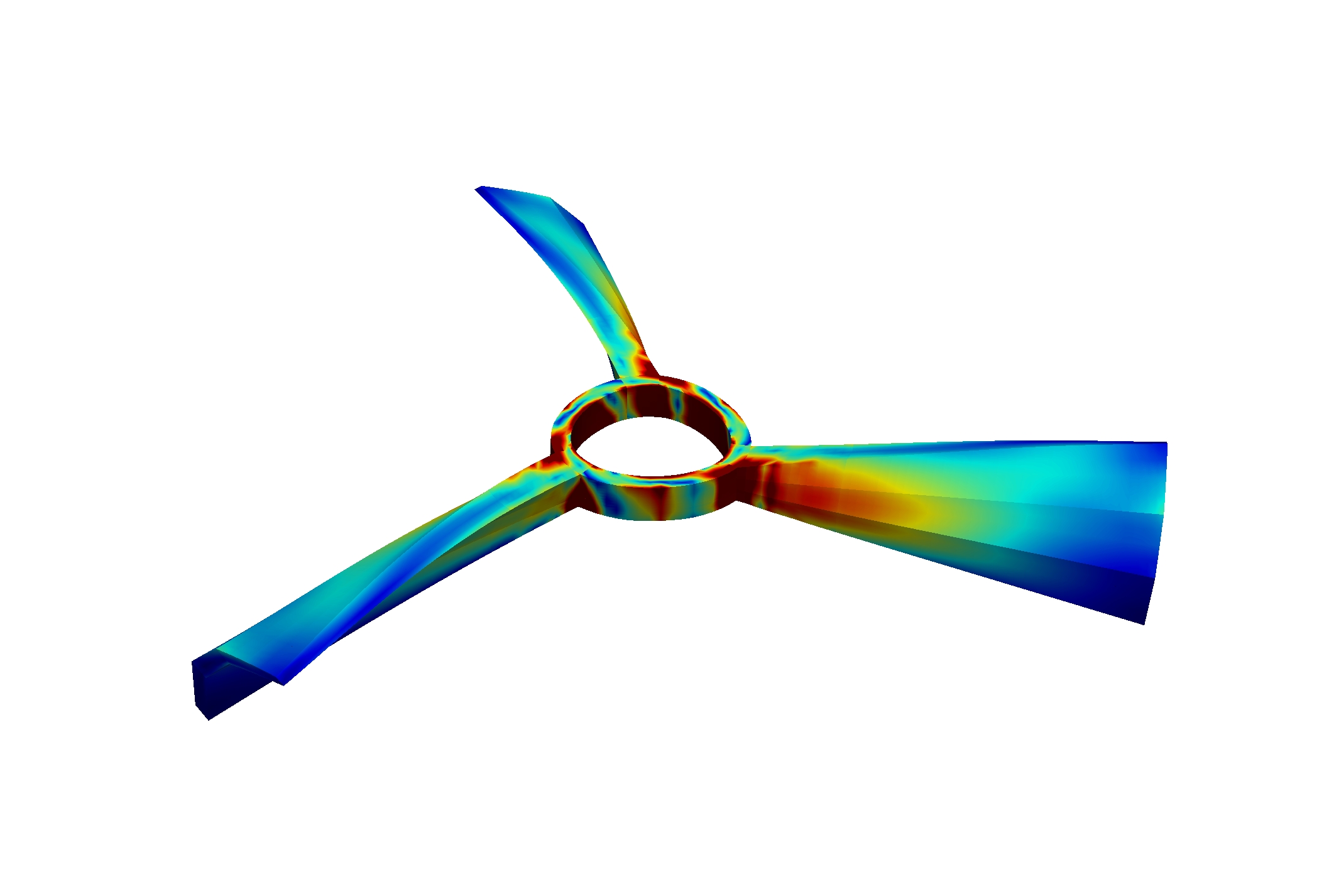} \\
\includegraphics[angle=0, trim=150 330 150 220, clip=true, scale = 0.1]{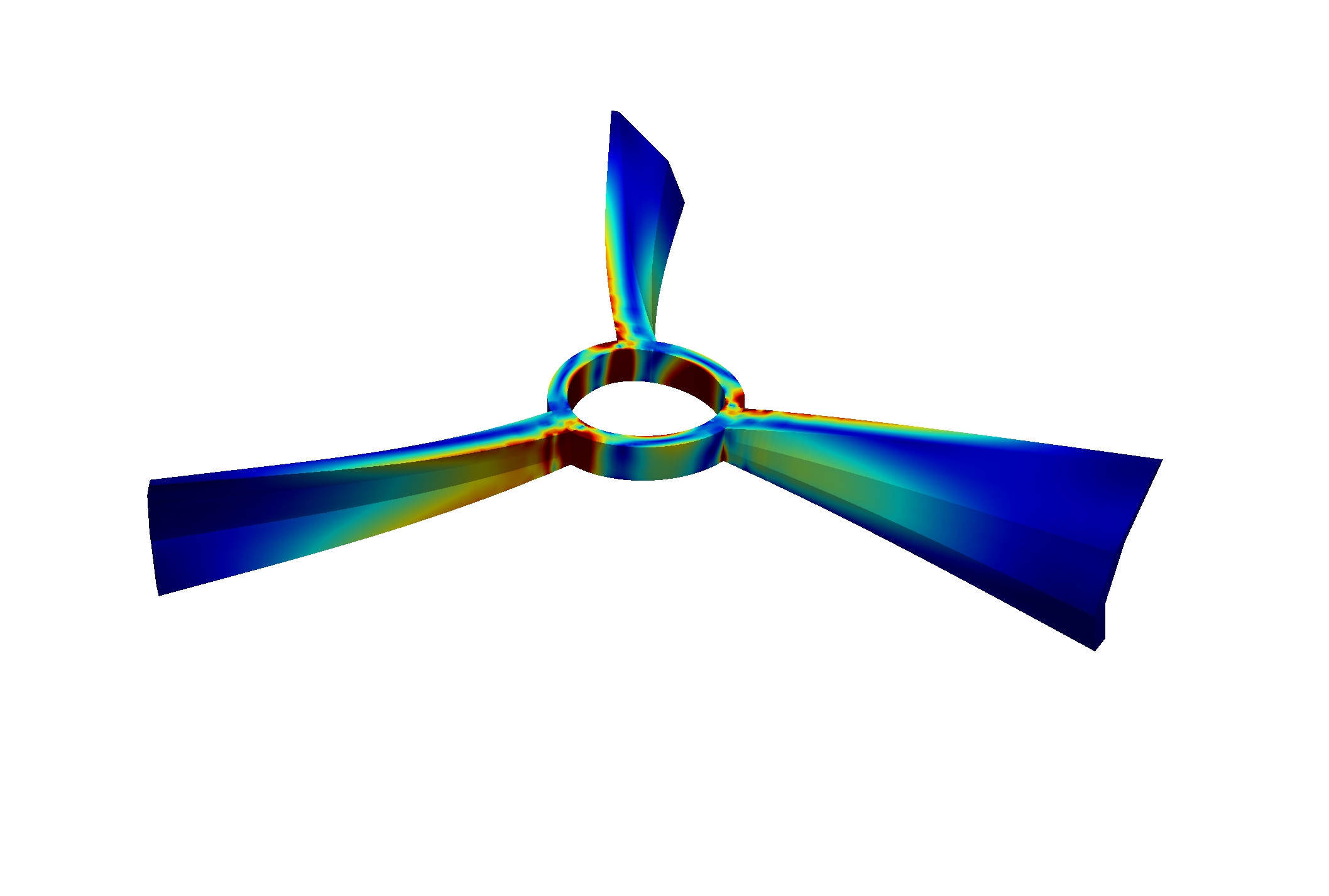} &
\includegraphics[angle=0, trim=150 330 150 220, clip=true, scale = 0.1]{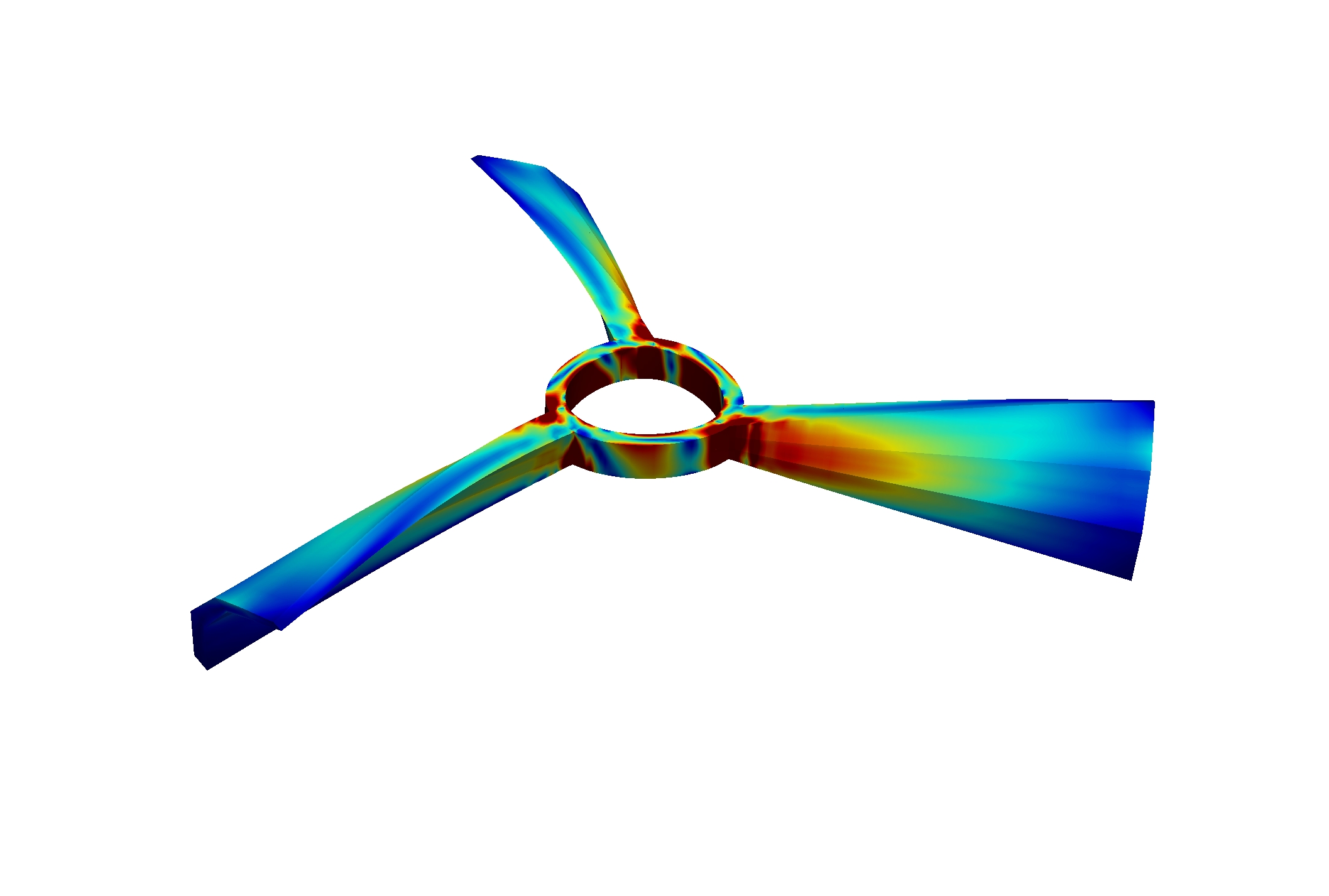} \\
$t = 17.5$ & $t = 28.5$ \\
\multicolumn{2}{c}{ \includegraphics[angle=0, trim=200 400 200 880, clip=true, scale = 0.25]{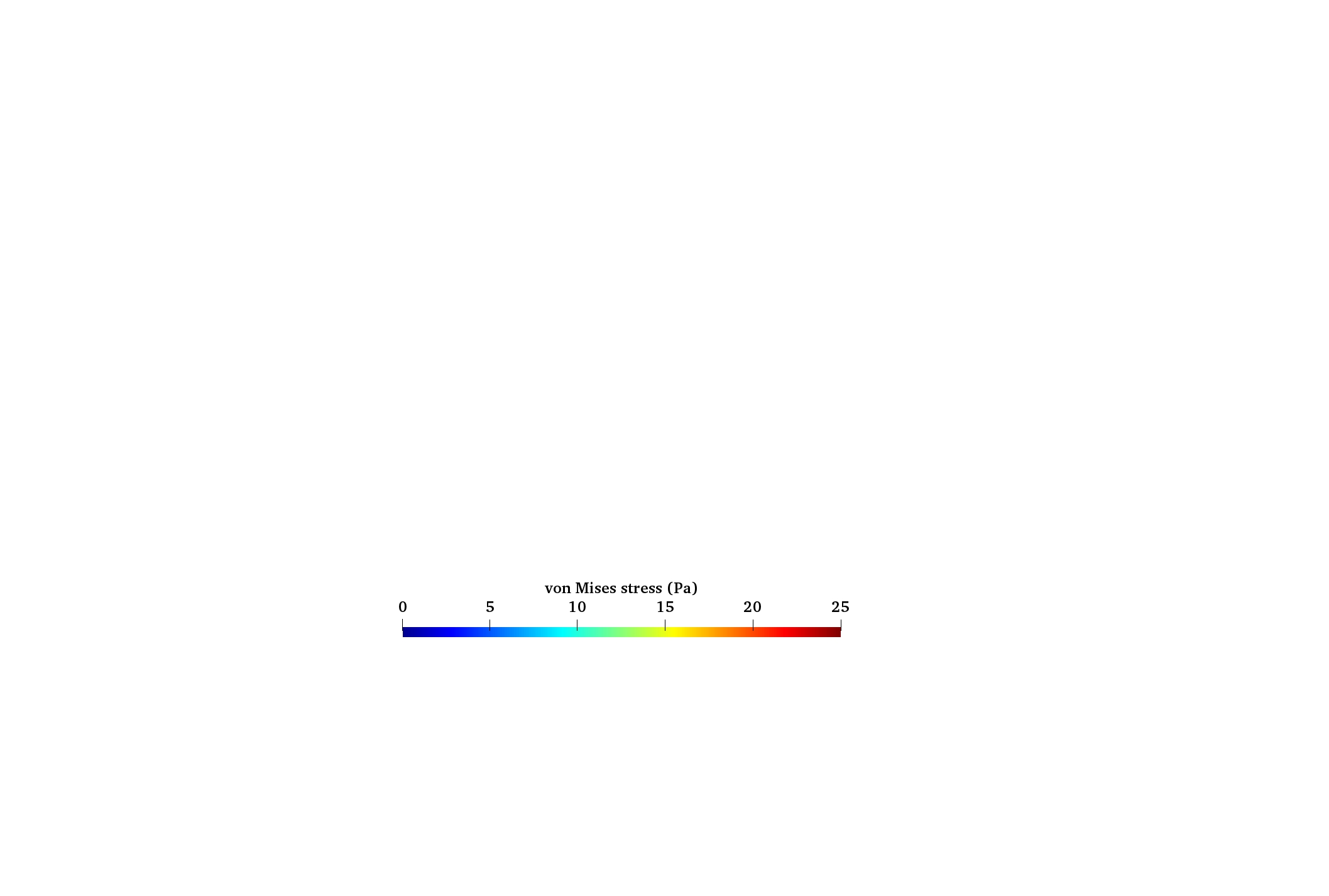} }
\end{tabular}
\caption{Comparison of the von Mises stress obtained by the coarse (top) and fine (bottom) meshes with $\Delta t = 0.05$.}
\label{fig:propeller_stress_convergence}
\end{center}
\end{figure}

\begin{figure}
\begin{center}
\begin{tabular}{cc}
\includegraphics[angle=0, trim=100 180 100 100, clip=true, scale = 0.21]{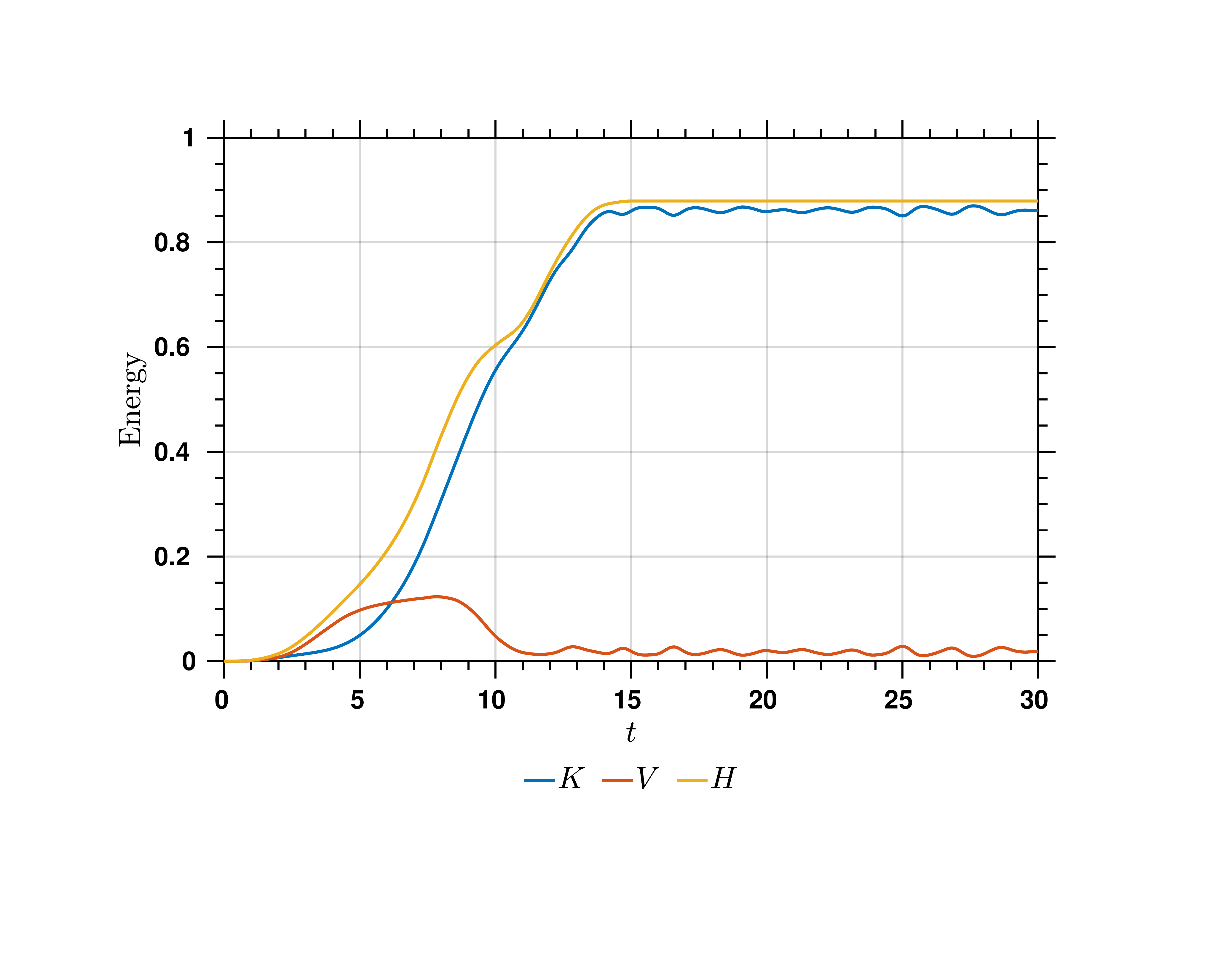} &
\includegraphics[angle=0, trim=100 180 100 100, clip=true, scale = 0.21]{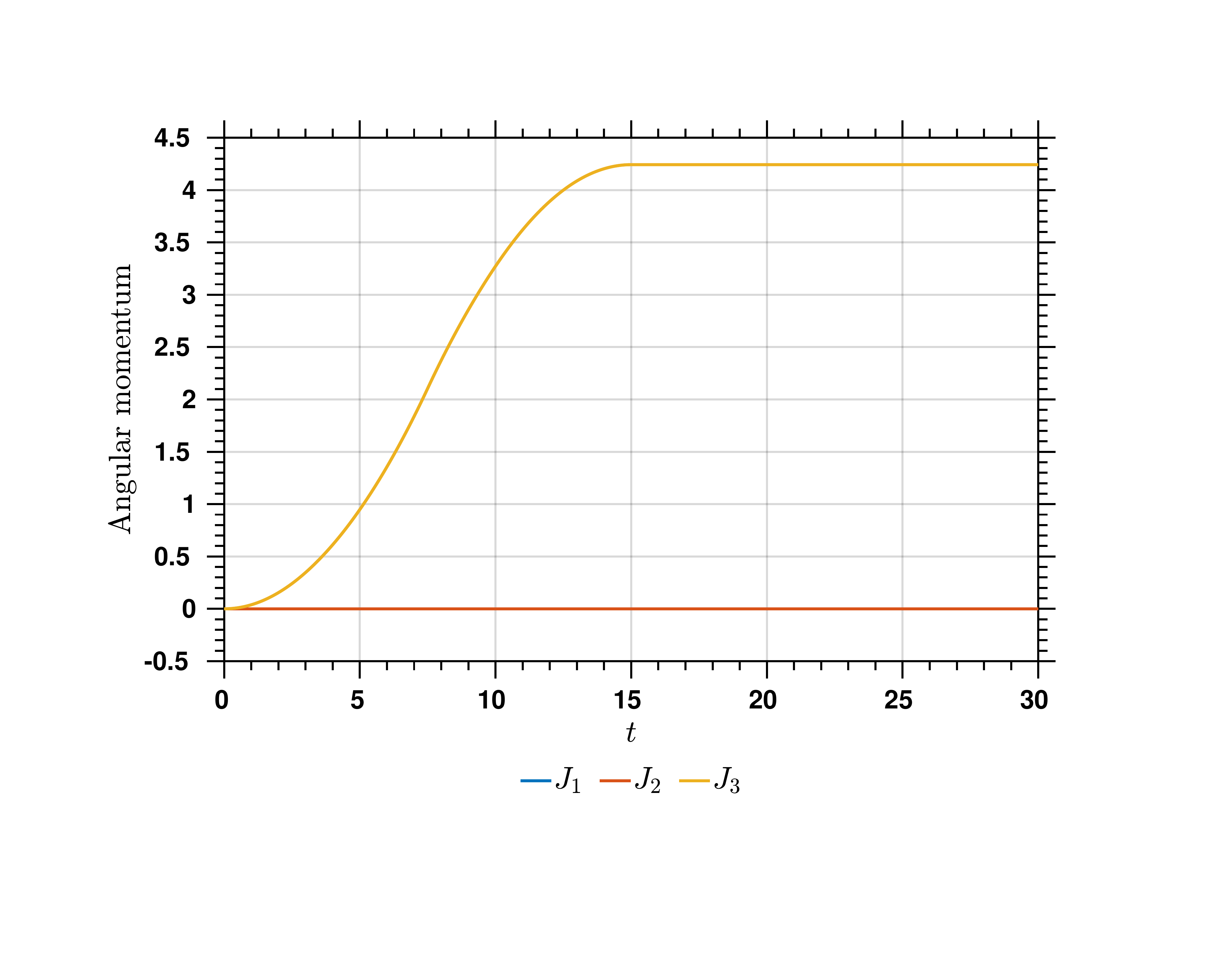}
\end{tabular}
\caption{The time histories of the discrete energy and angular momentum for the rotating propeller with $\Delta t = 0.05$.}
\label{fig:propeller_energy_momenta}
\end{center}
\end{figure}

\begin{figure}
\begin{center}
\begin{tabular}{cc}
\includegraphics[angle=0, trim=80 80 100 120, clip=true, scale = 0.21]{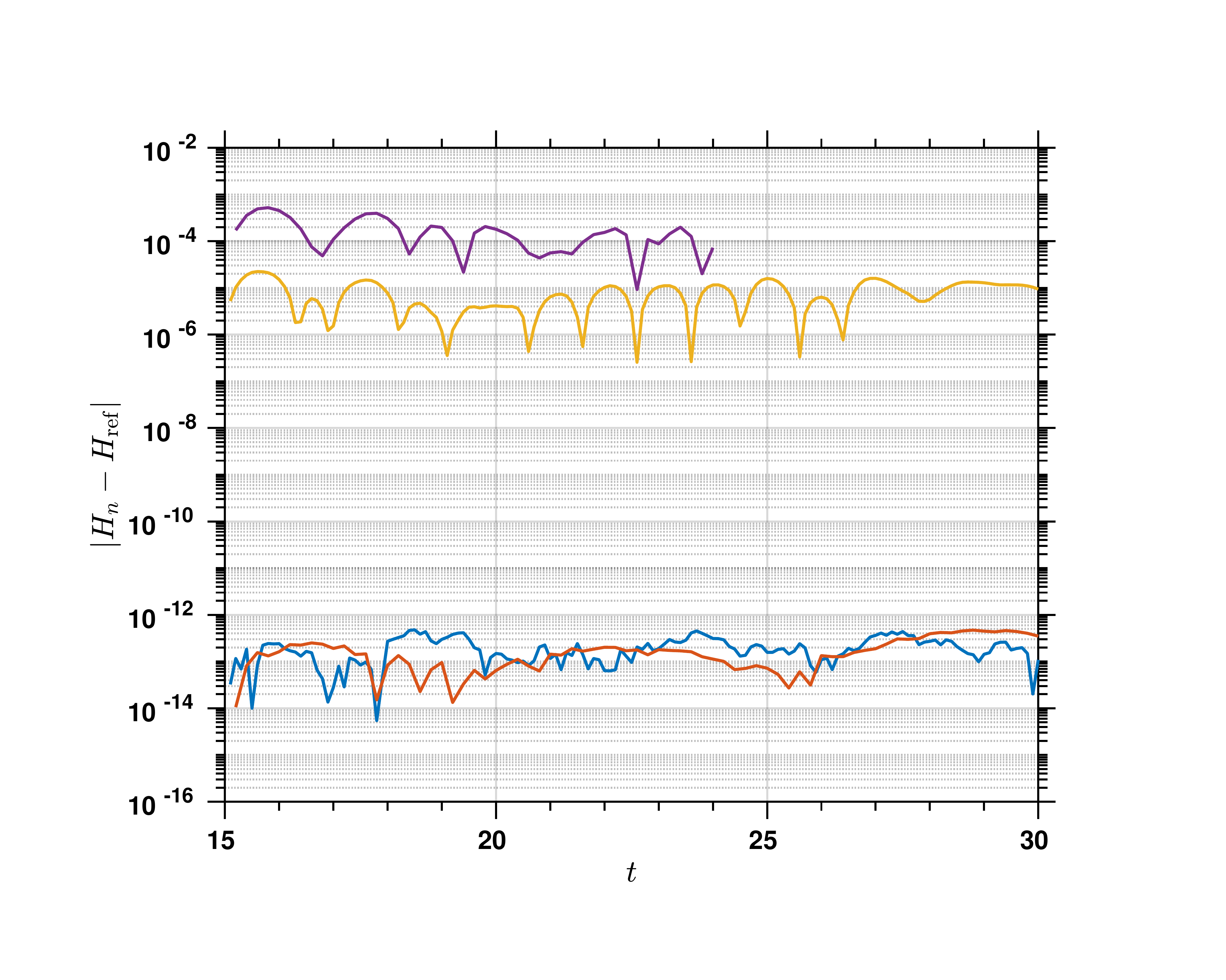} &
\includegraphics[angle=0, trim=80 80 100 120, clip=true, scale = 0.21]{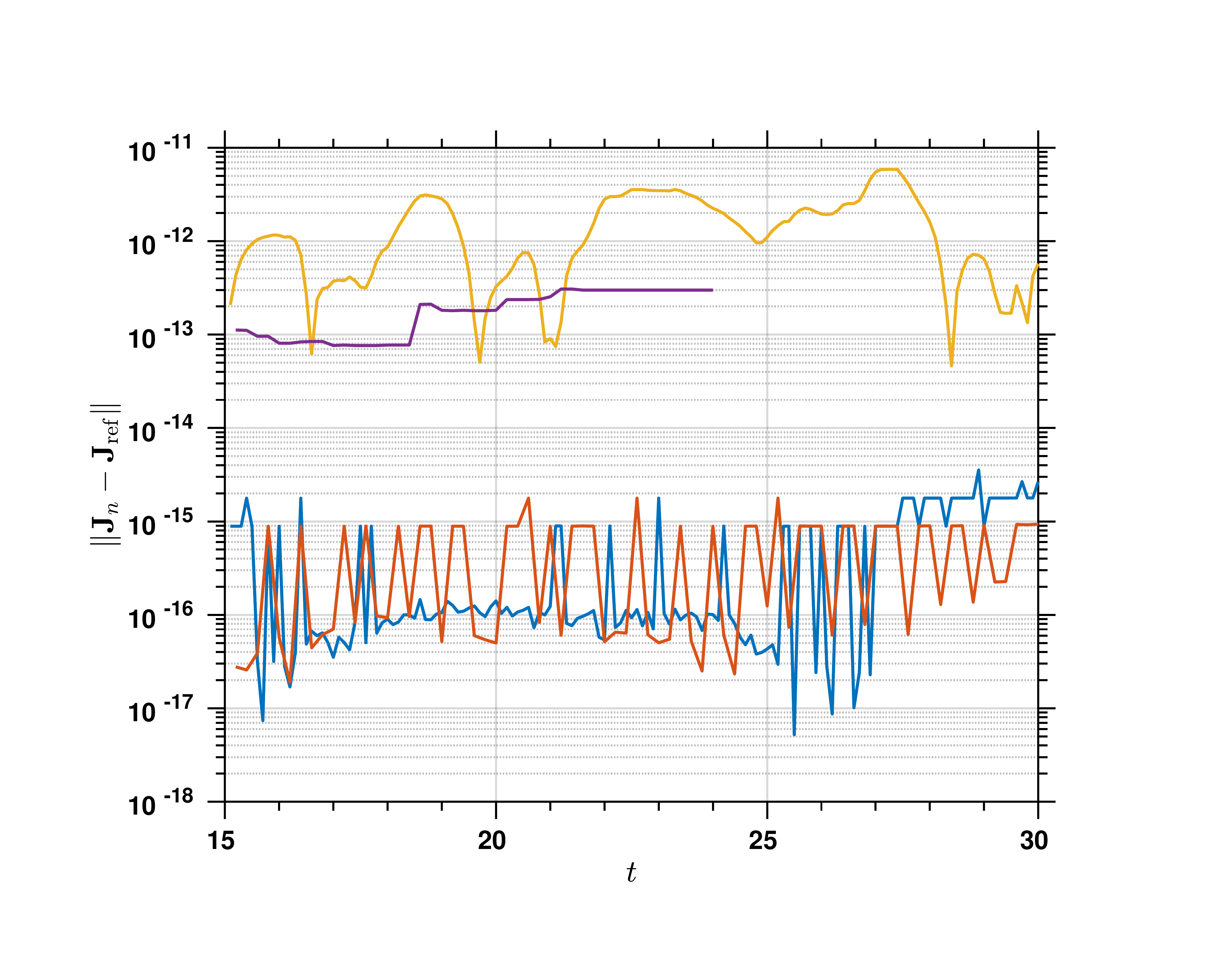} \\
\multicolumn{2}{c}{ \includegraphics[angle=0, trim=0 165 0 800, clip=true, scale = 0.35]{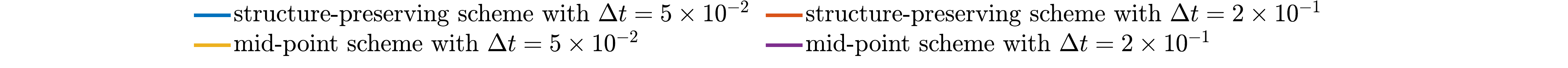} }
\end{tabular}
\caption{The time histories of the absolute errors of the total Hamiltonian and angular momentum given by the structure-preserving scheme and the mid-point scheme with two different time step sizes. The reference quantities are chosen as their values at time $\bar{T}$, that is, $(\cdot)_{\mathrm{ref}} = (\cdot)|_{t=15}$.}
\label{fig:propeller_relative_errors}
\end{center}
\end{figure}

\begin{figure}
\begin{center}
\begin{tabular}{cc}
\includegraphics[angle=0, trim=100 85 120 110, clip=true, scale = 0.22]{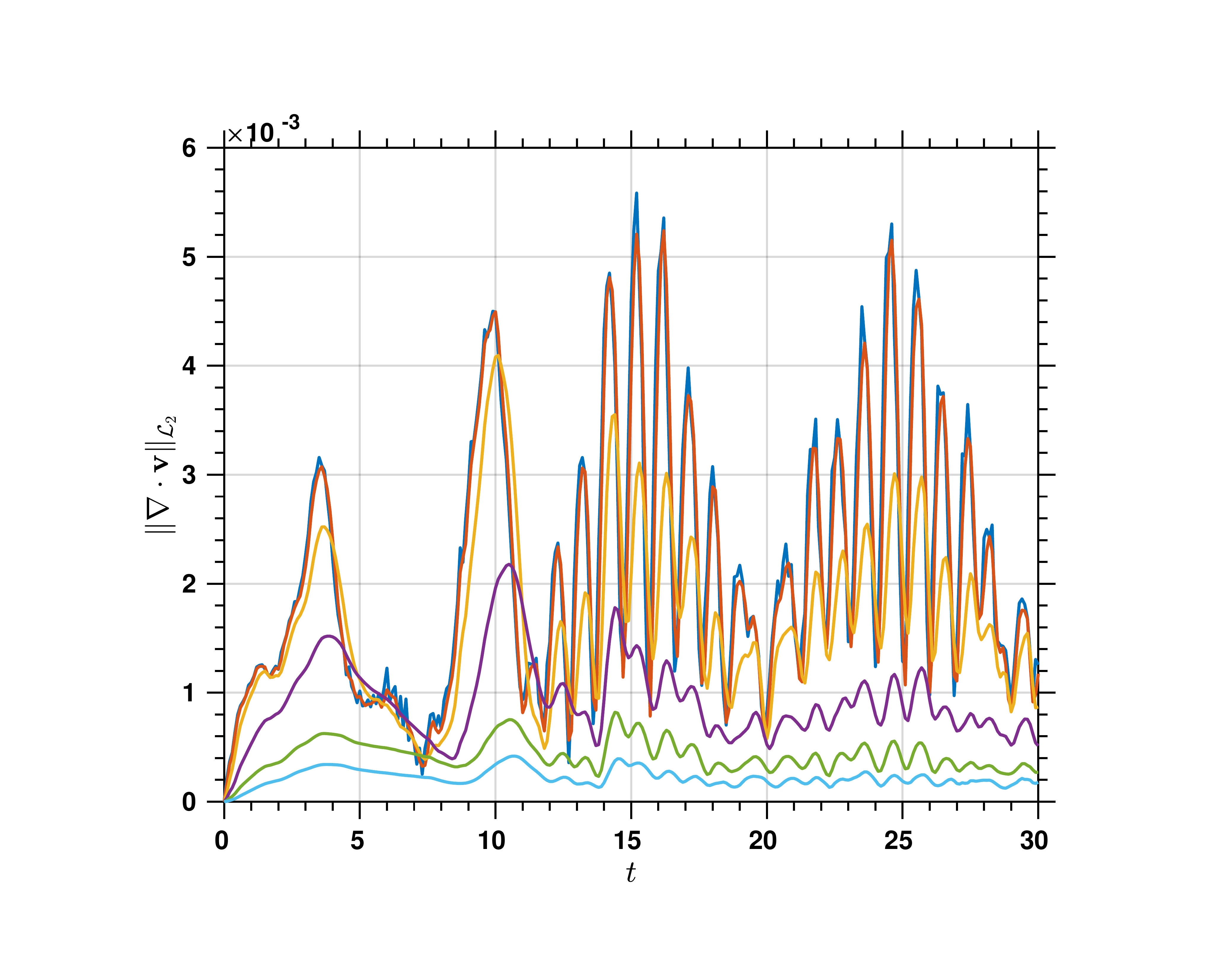} &
\includegraphics[angle=0, trim=100 85 120 110, clip=true, scale = 0.22]{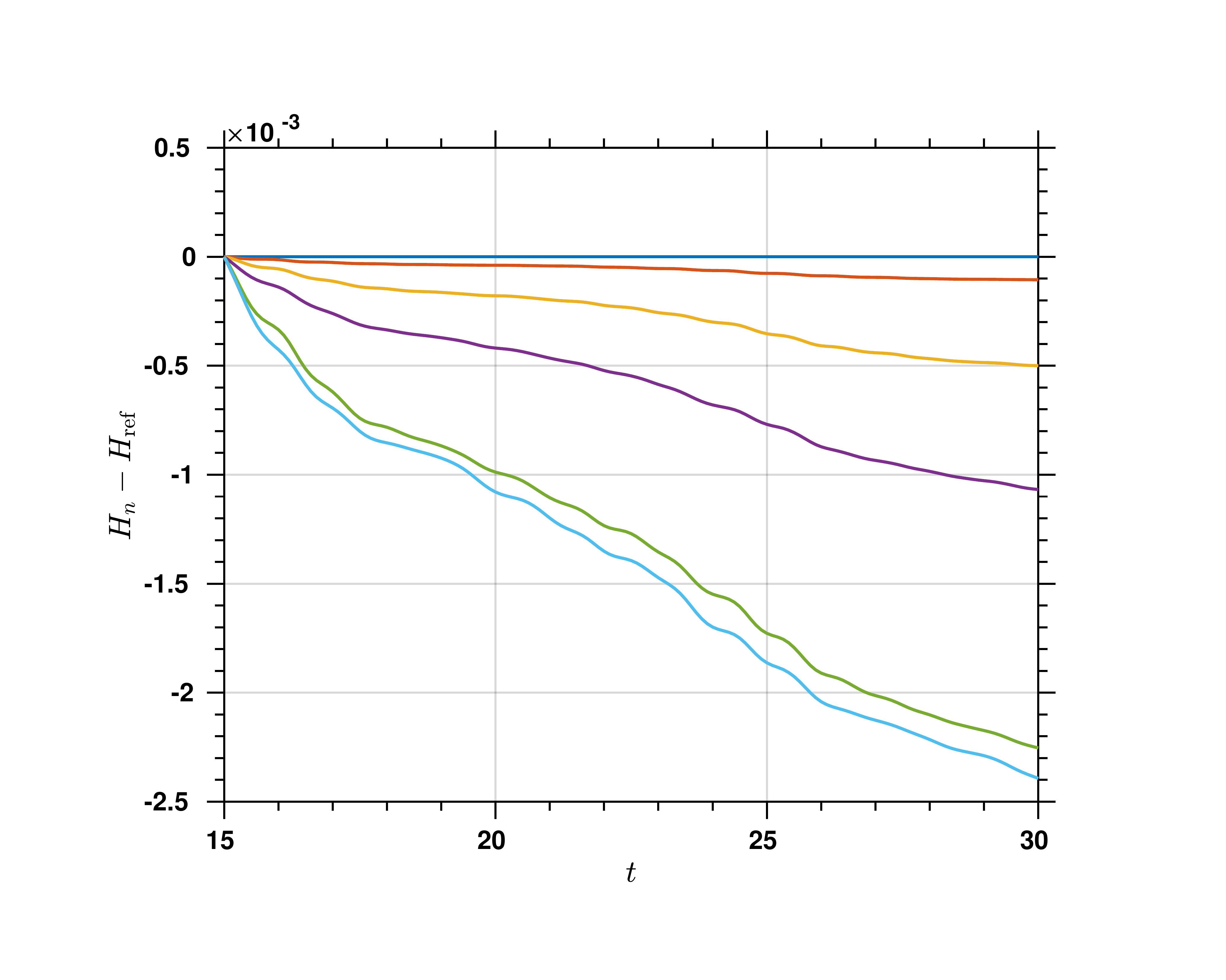} \\
(a) & (b) \\
\includegraphics[angle=0, trim=80 85 120 110, clip=true, scale = 0.22]{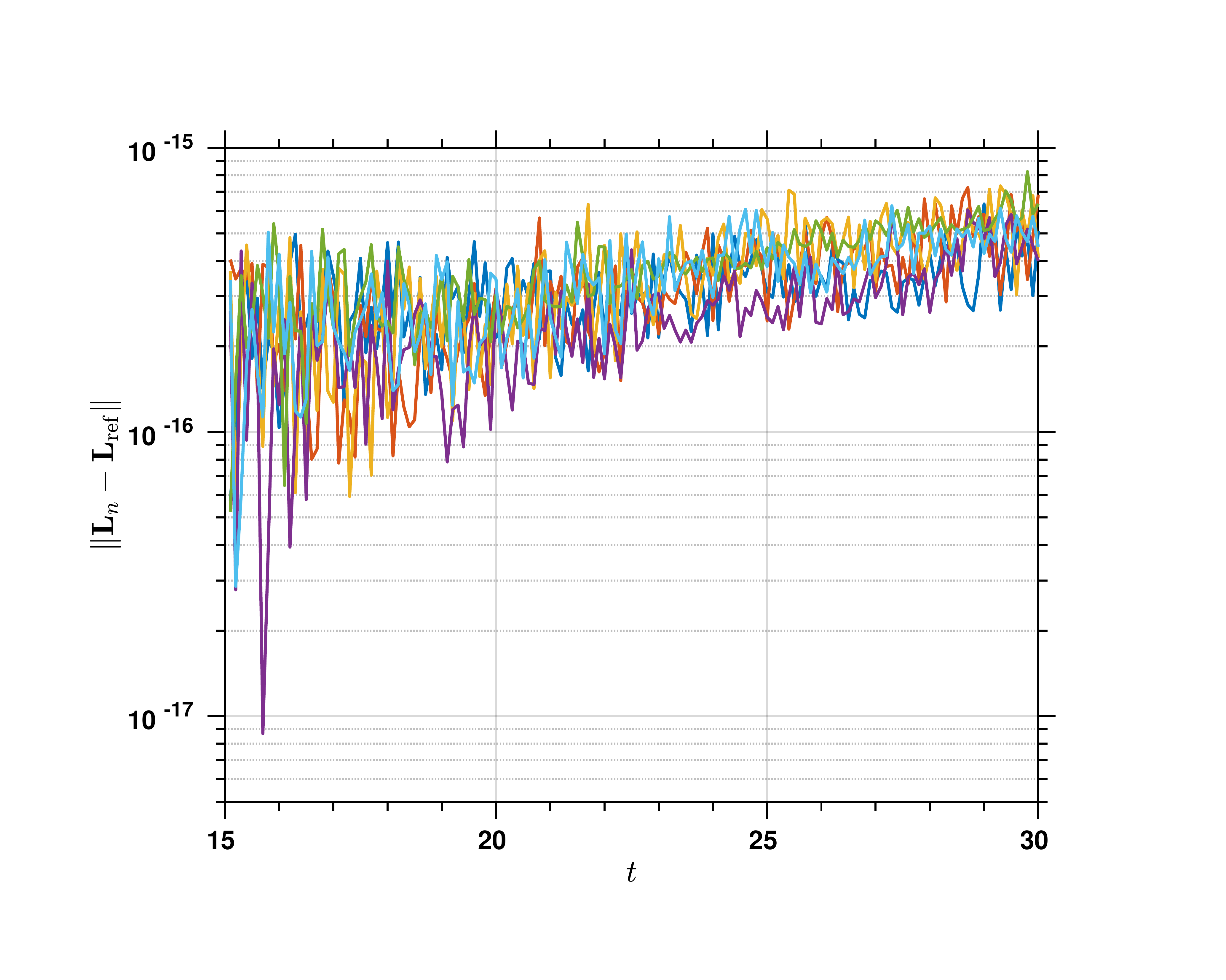} &
\includegraphics[angle=0, trim=80 85 120 110, clip=true, scale = 0.22]{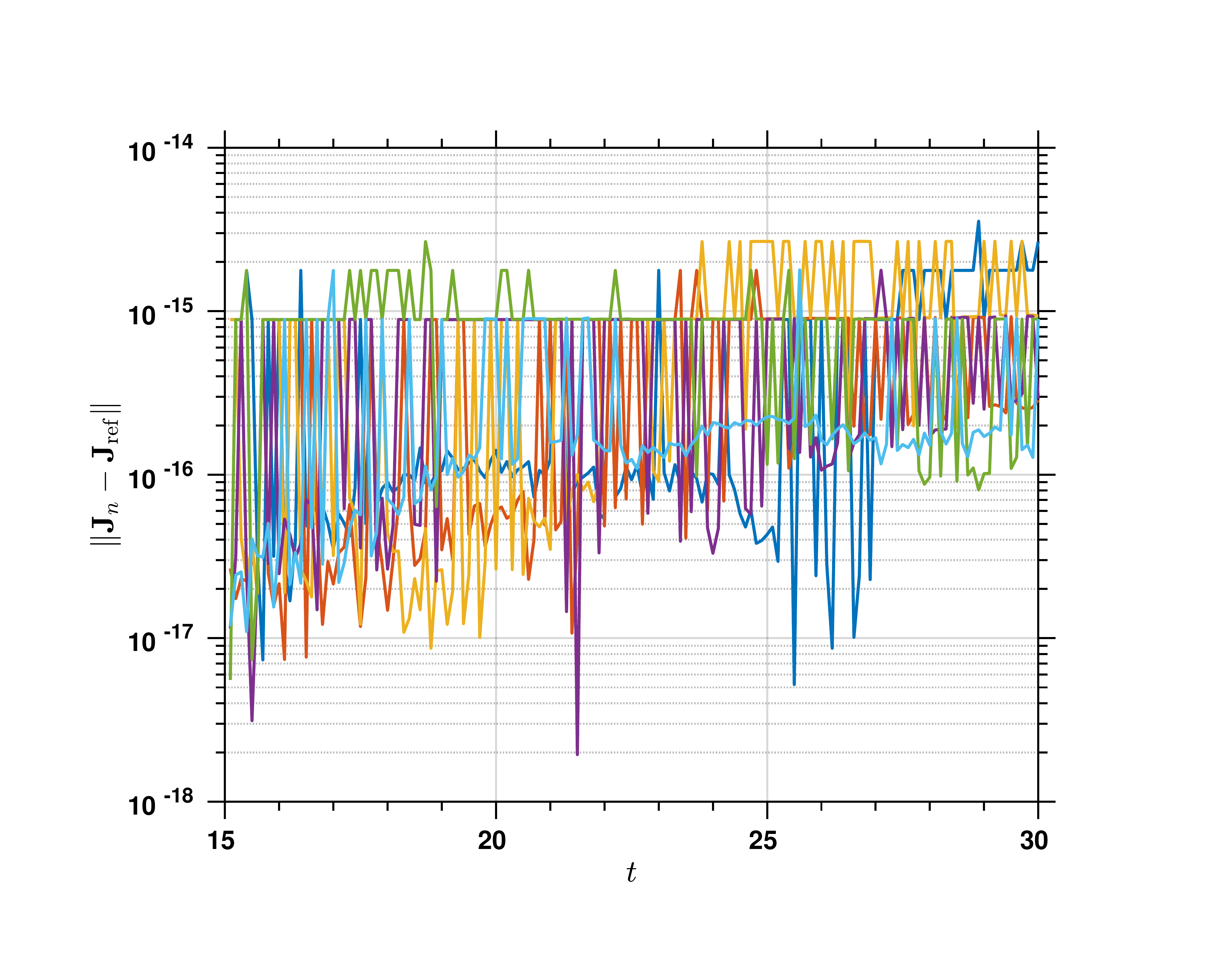} \\
(c) & (d) \\
\multicolumn{2}{c}{ \includegraphics[angle=0, trim=0 165 0 750, clip=true, scale = 0.35]{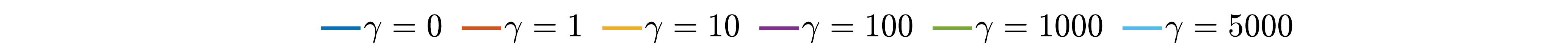} }
\end{tabular}
\caption{The effect of the grad-div stabilization parameter $\gamma$ for the rotating propeller: (a) $\| \nabla \cdot \bm{v} \|_{\mathcal{L}_2}$, (b) the absolute errors of the total Hamiltonian, (c) the absolute errors of the total linear momentum, and (d) the absolute errors of the angular momentum over time.}
\label{fig:propeller_graddiv}
\end{center}
\end{figure}

The problem is integrated up to $ T=30$ using the proposed scheme and the mid-point scheme. Two spatial meshes are prepared with $\mathsf p=2$. The coarse (fine) mesh consists of 45 (180) elements in the central ring and 36 (192) elements in each of the three blades. Unless otherwise specified, we set the grad-div stabilization parameter $\gamma$ to zero.  In Figure \ref{fig:propeller_deformation}, a sequence of the deformed states with the pressure distribution at different time instants are depicted; in Figure \ref{fig:propeller_stress_convergence}, the von Mises stress obtained by the two meshes is displayed at $t = 17.5$ and $28.5$, respectively. It is observed that the von Mises stress obtained by the two meshes is almost indistinguishable, suggesting the stress has been accurately resolved by the coarse mesh.

The time histories of the Hamiltonian and angular momentum are depicted in Figure \ref{fig:propeller_energy_momenta}. To better illustrate the conserving properties, the absolute errors of the Hamiltonian and angular momentum are plotted in Figure \ref{fig:propeller_relative_errors}. It can be observed from the figures that the magnitudes of the Hamiltonian errors are smaller than $10^{-12}$ for the structure-preserving scheme, regardless of the time step size. The mid-point scheme exhibits an absolute error of the order $10^{-5}$ with $\Delta t= 5\times 10^{-2}$ and an early termination with $\Delta t=2\times 10^{-1}$. In terms of the angular momentum, the absolute errors of the structure-preserving integrator are of order $10^{-15}$, which is approximately three orders of magnitude smaller than the errors of the mid-point scheme.

The effect of the grad-div stabilization is studied next. The time step size is fixed to be $\Delta t = 0.05$. The parameter $\gamma$ varies from $0$ to $5\times 10^3$. In Figure \ref{fig:propeller_graddiv}(a), the $\mathcal L_2$-norm of $\nabla \cdot \bm{v}$ over time is depicted. It can be observed that $\| \nabla \cdot \bm{v} \|_{\mathcal L_2}$ decreases with the increase of the parameter $\gamma$. In Figure \ref{fig:propeller_graddiv}(b), the time histories of $H_n - H_{\mathrm{ref}}$ is plotted, from which one may observe the dissipation effect of the grad-div stabilization term. The amount of the numerical dissipation is proportional to the value of $\gamma$. In the meantime, we do notice that the curves corresponding to $\gamma=1000$ and $5000$ are getting closer. In our numerical experience, there are cases when increasing the value of $\gamma$ leads to a reduction of dissipation. Indeed, there is a competing mechanism in the definition of the numerical dissipation $\mathcal{D}_{m}$, and we expect that the dissipation will vanish when the value of $\gamma$ gets large enough (see Remark \ref{remark:grad-div-dissipation-vanishment}). The effects of the grad-div stabilization on the momentum conservation are depicted in Figure \ref{fig:propeller_graddiv}(c) and (d). Results demonstrate that the conservation of the total linear and angular momentum is not disturbed by the addition of grad-div stabilization.

\subsection{Twisting column}
\label{sec:twisting-column}
In the third example, we consider a twisting column problem originally proposed in \cite{Janz2019}. Its setting is summarized in Table \ref{table:column}. The motion is initiated by an initial velocity field, and no traction boundary condition is applied on $\Gamma_{\bm X}=\Gamma^{H}_{\bm X}$. The problem is studied up to $T=5$ with the time step size $\Delta t = 0.01$. 

\begin{table}[htbp]
  \centering
  \begin{tabular}{ m{.4\textwidth}   m{.5\textwidth} }
    \hline
    \begin{minipage}{.35\textwidth}
      \includegraphics[trim=220 170 185 280, clip=true, scale = 0.65]{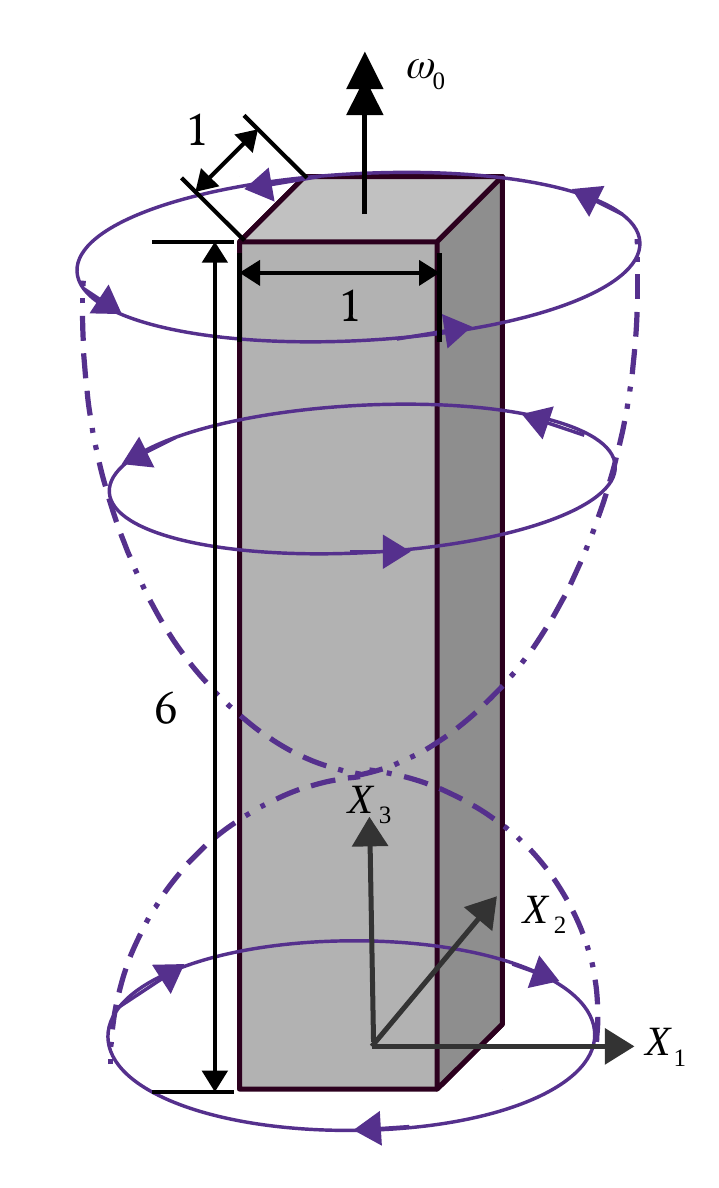}
    \end{minipage}
    &
    \begin{minipage}{.45\textwidth}
      \begin{itemize}
        \item[] Material properties:
        \item[] $\rho_0 = 1.0\times 10^3, \quad N = 3$,
        \item[] $G_{\mathrm{ich}}( \tilde{\lambda}_1, \tilde{\lambda}_2, \tilde{\lambda}_3 ) = \sum\limits_{a=1}^{3} \sum\limits_{p=1}^{N} \frac{\mu_p}{\alpha_p}( \tilde{\lambda}_{a}^{\alpha_p} - 1 )$,
        \item[] $\alpha_1 = 1.3, \quad \mu_1 = 6.3\times 10^5$,
        \item[] $\alpha_2 = 5.0, \quad \mu_2 = 1.2\times 10^3$,
        \item[] $\alpha_3 = -2.0, \quad \mu_3 = -1.0\times 10^4$.
        \item[] Initial velocity:
        \item[] $\bm{V}_0 = \bm{\omega}_0 \times \bm{X}$,
        \item[] $\bm{\omega}_0 = \left[ 0,~0,~\Omega_1 \sin \left( \frac{\pi \left(X_3-0.5L \right)}{2L} \right) + \Omega_2  \right]$,
        \item[] $\Omega_1 = 20$, $\Omega_2 = 5$, $L=6$.  
      \end{itemize}
    \end{minipage}   
\\    
    \hline
  \end{tabular}
    \caption{The three-dimensional twisting column: problem setting.} 
\label{table:column}
\end{table}

\begin{figure}
\begin{center}
\begin{tabular}{cccccccc}
\includegraphics[angle=0, trim=450 120 450 0, clip=true, scale = 0.045]{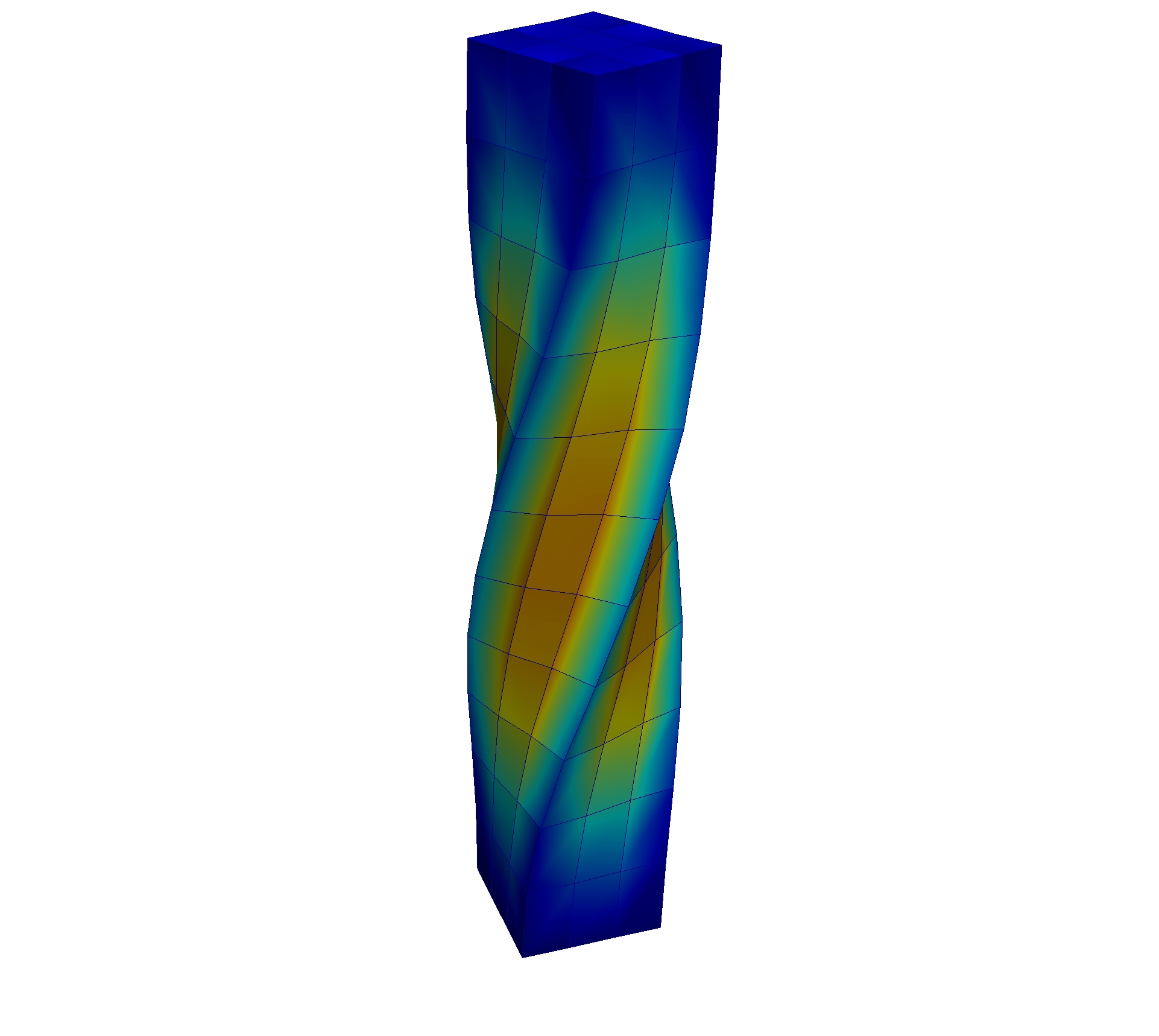} &
\includegraphics[angle=0, trim=450 120 450 0, clip=true, scale = 0.045]{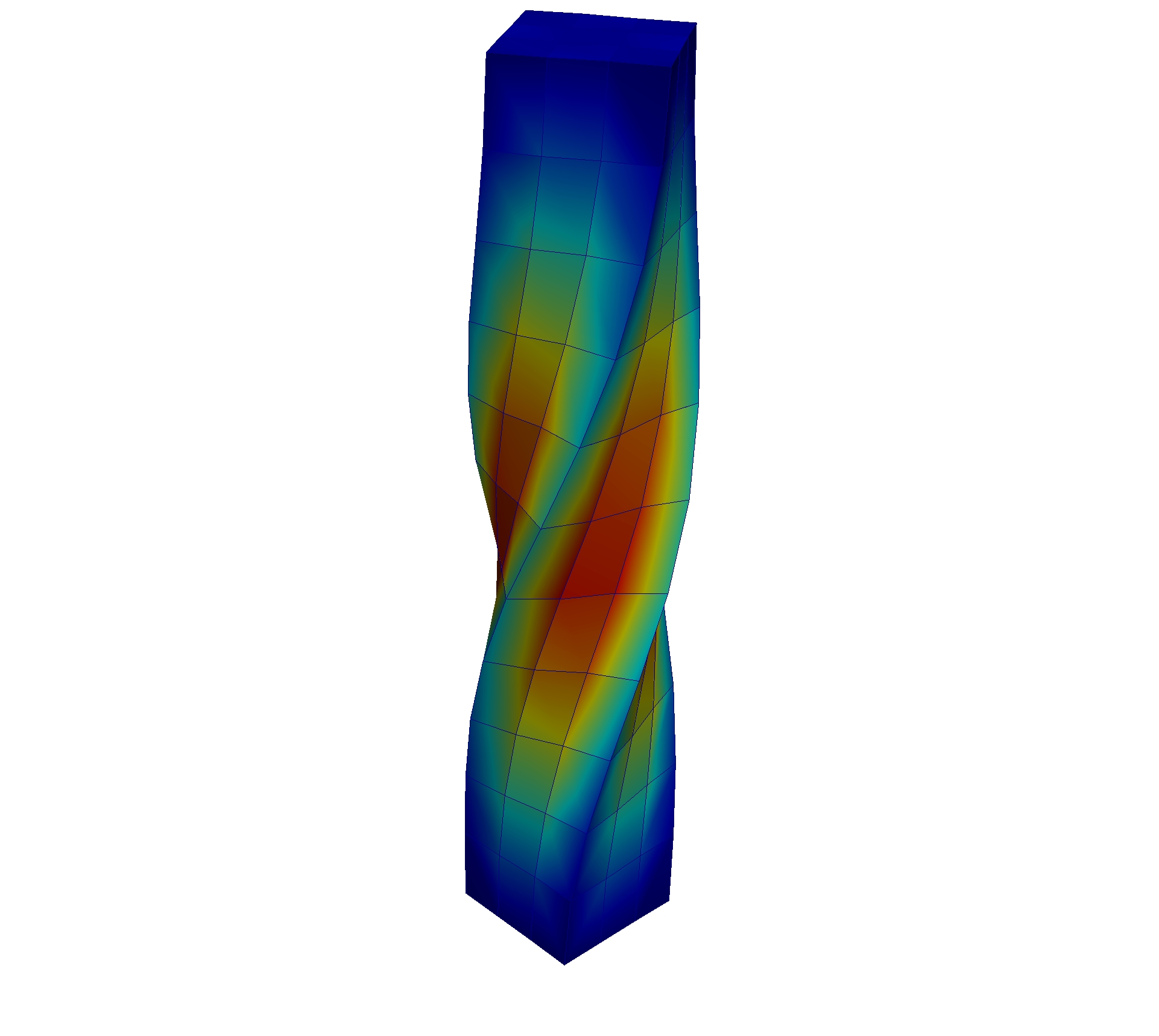} &
\includegraphics[angle=0, trim=450 120 450 0, clip=true, scale = 0.045]{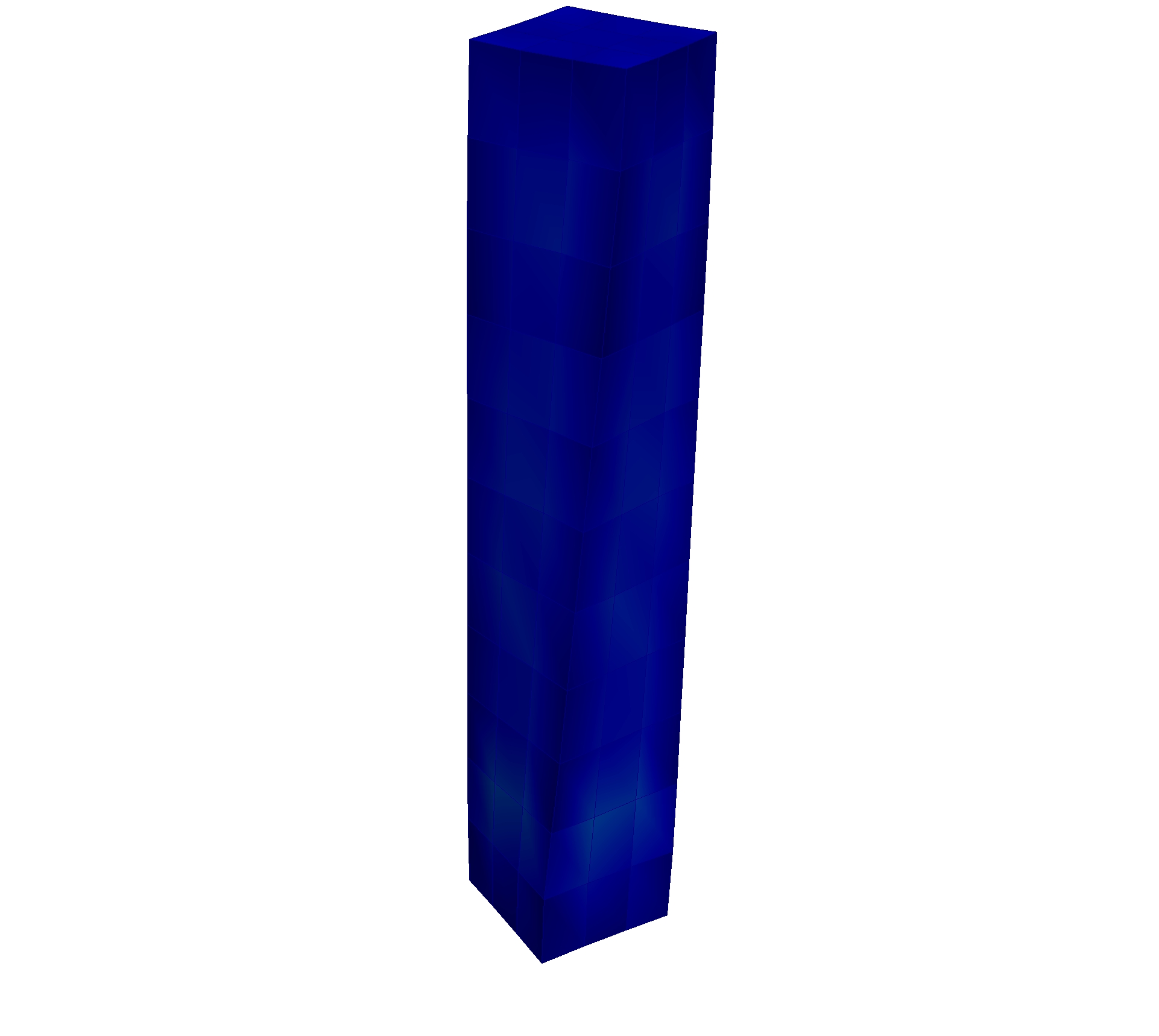} &
\includegraphics[angle=0, trim=450 120 450 0, clip=true, scale = 0.045]{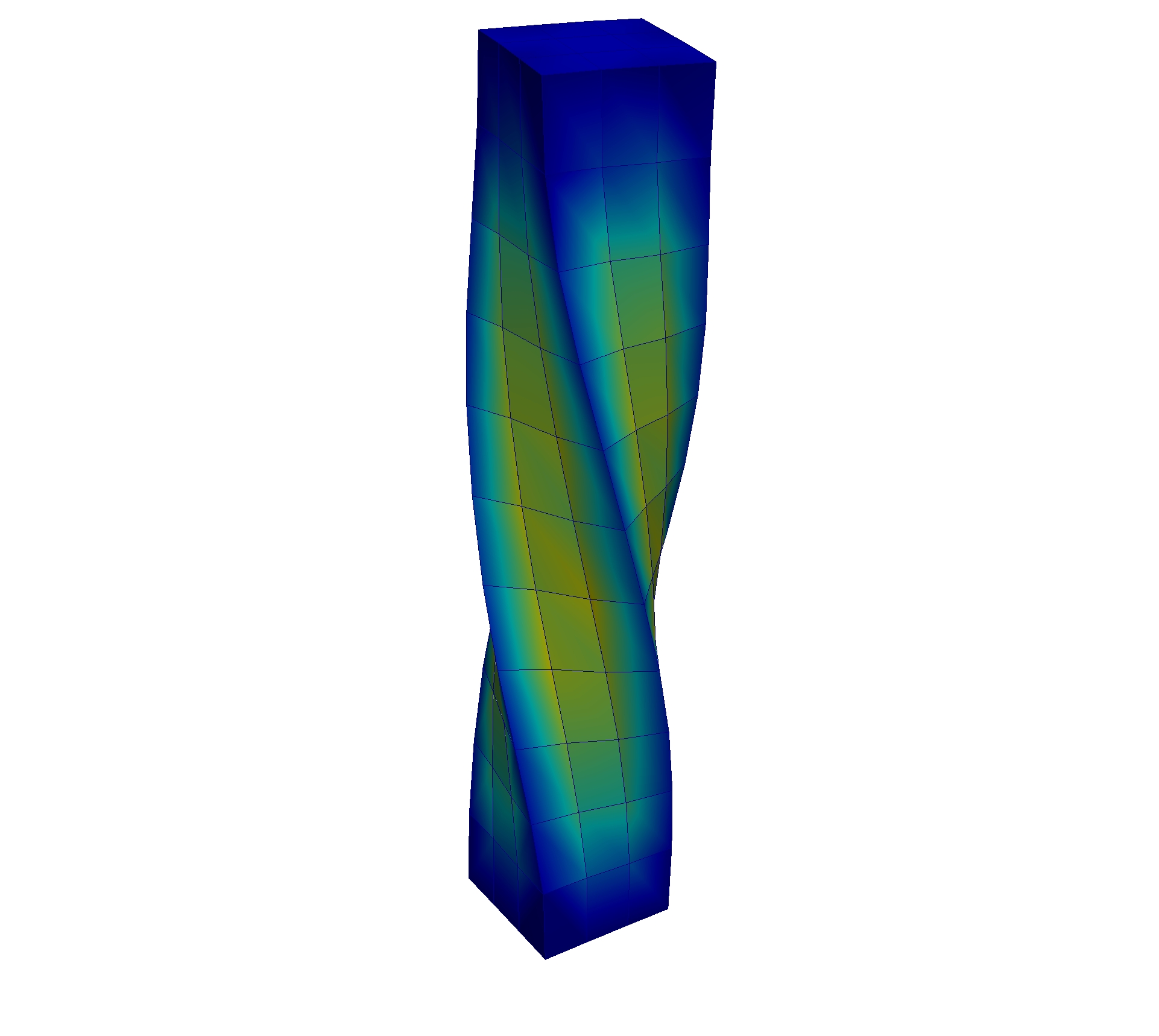} &
\includegraphics[angle=0, trim=450 120 450 0, clip=true, scale = 0.045]{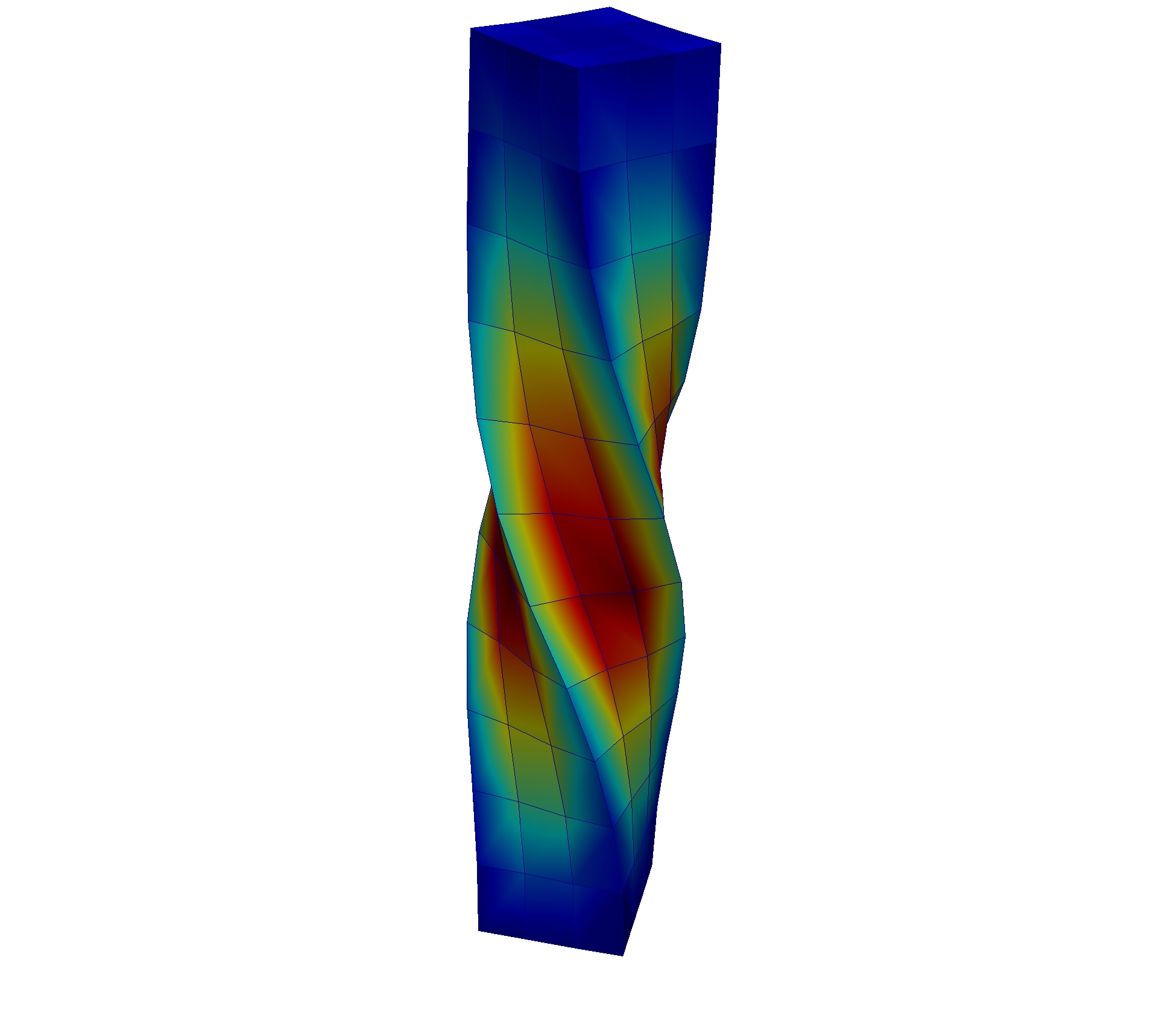} &
\includegraphics[angle=0, trim=450 120 450 0, clip=true, scale = 0.045]{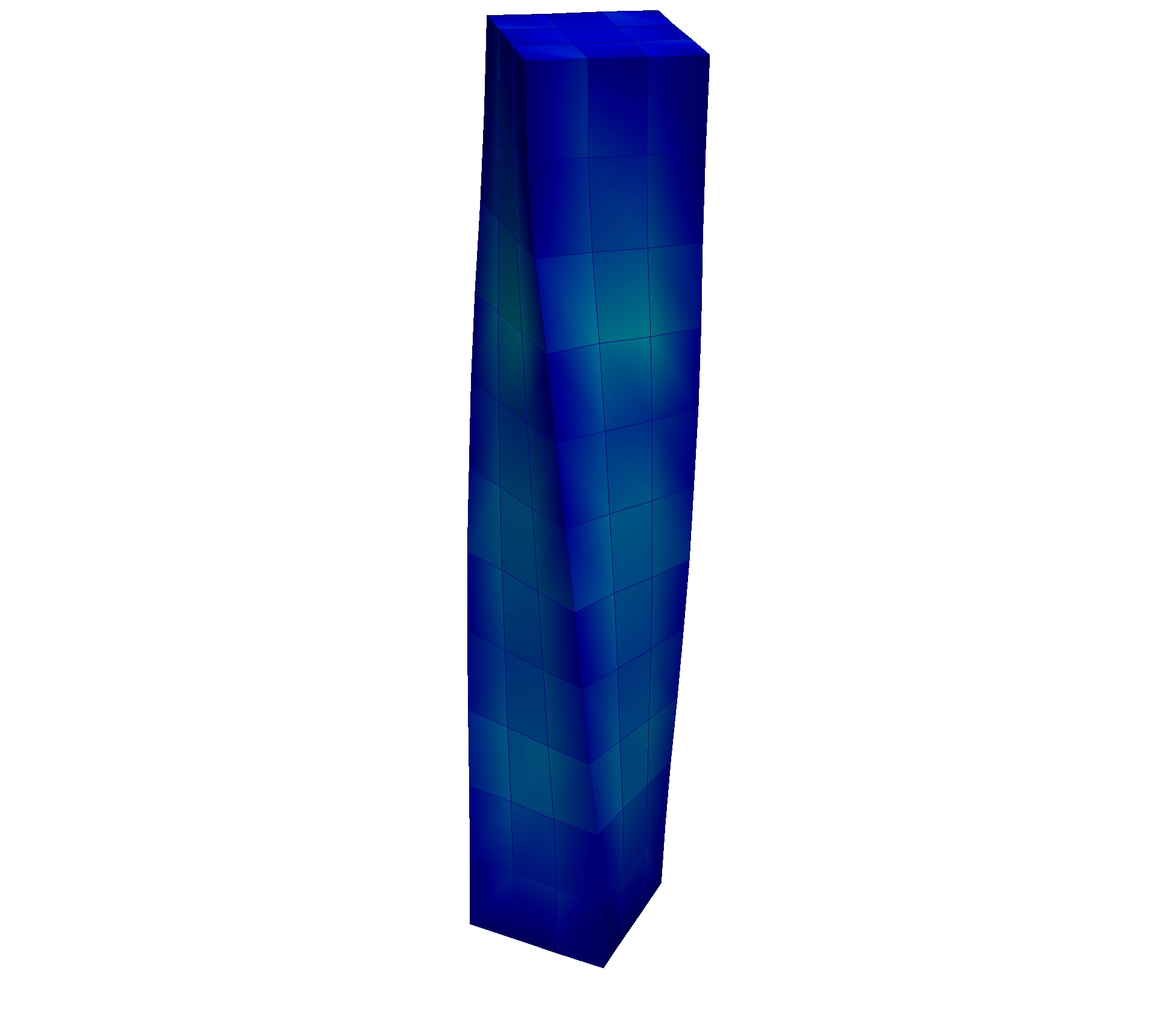} &
\includegraphics[angle=0, trim=450 120 450 0, clip=true, scale = 0.045]{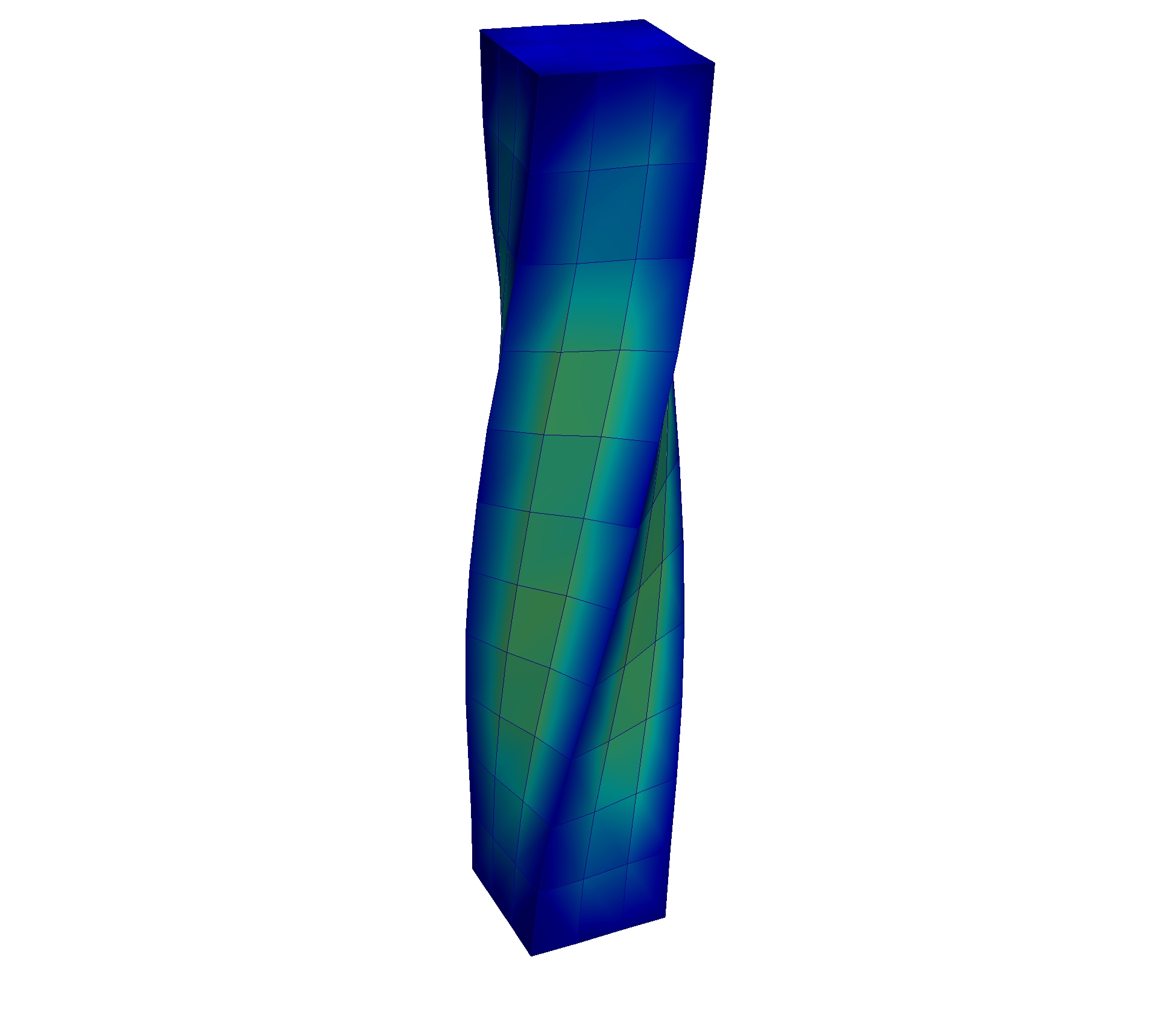} &
\includegraphics[angle=0, trim=450 120 450 0, clip=true, scale = 0.045]{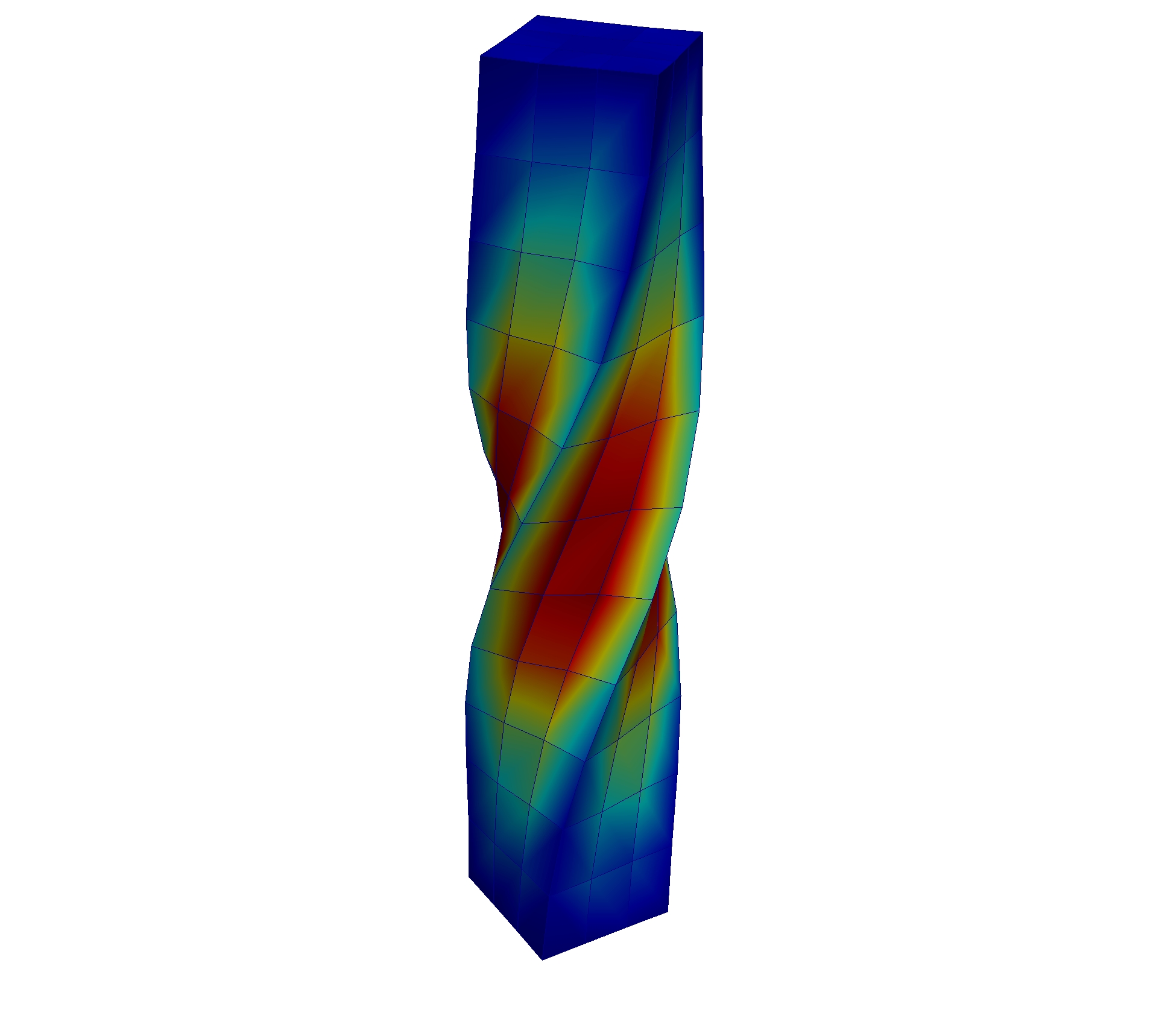} \\
\includegraphics[angle=0, trim=450 120 450 0, clip=true, scale = 0.045]{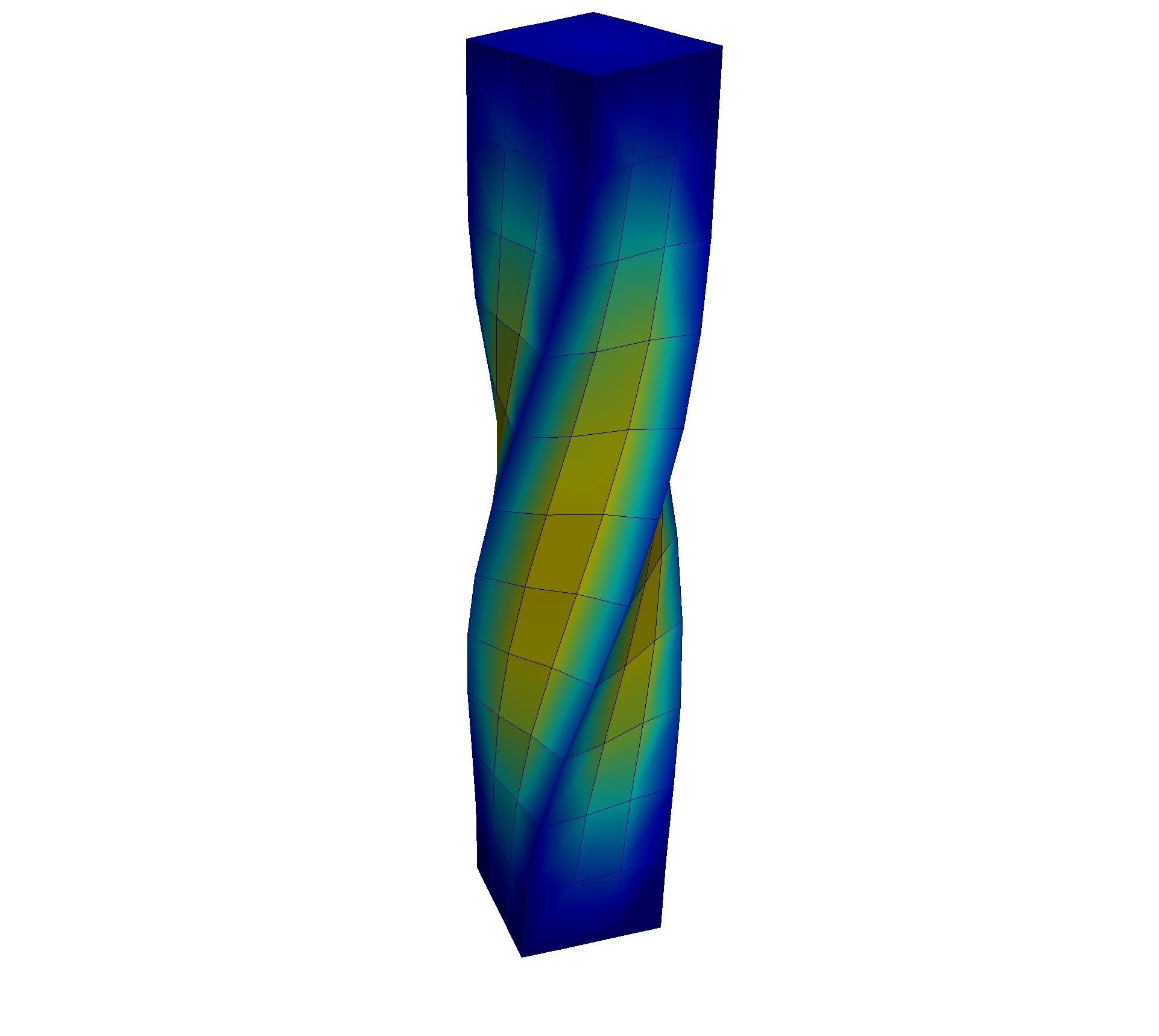} &
\includegraphics[angle=0, trim=450 120 450 0, clip=true, scale = 0.045]{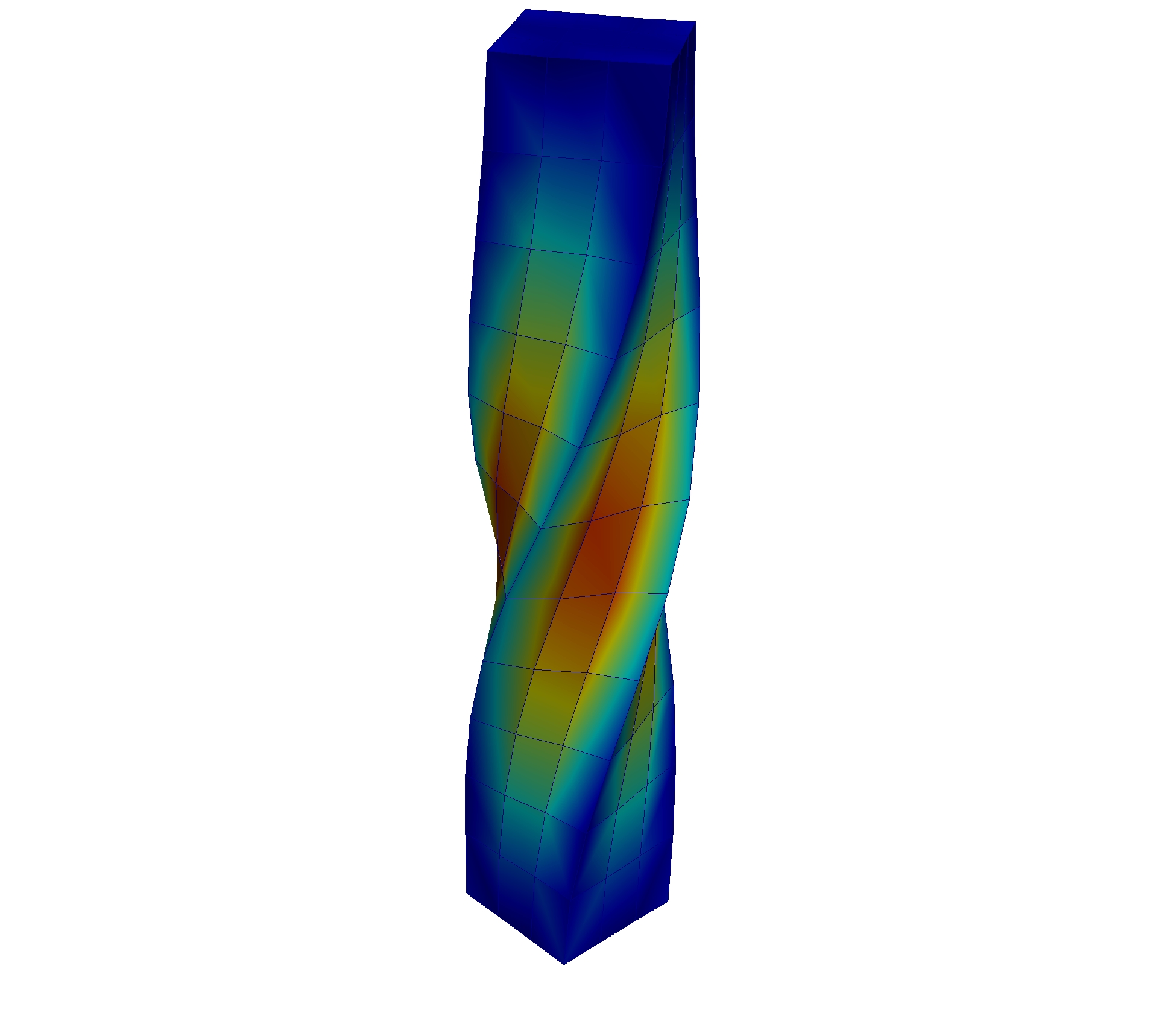} &
\includegraphics[angle=0, trim=450 120 450 0, clip=true, scale = 0.045]{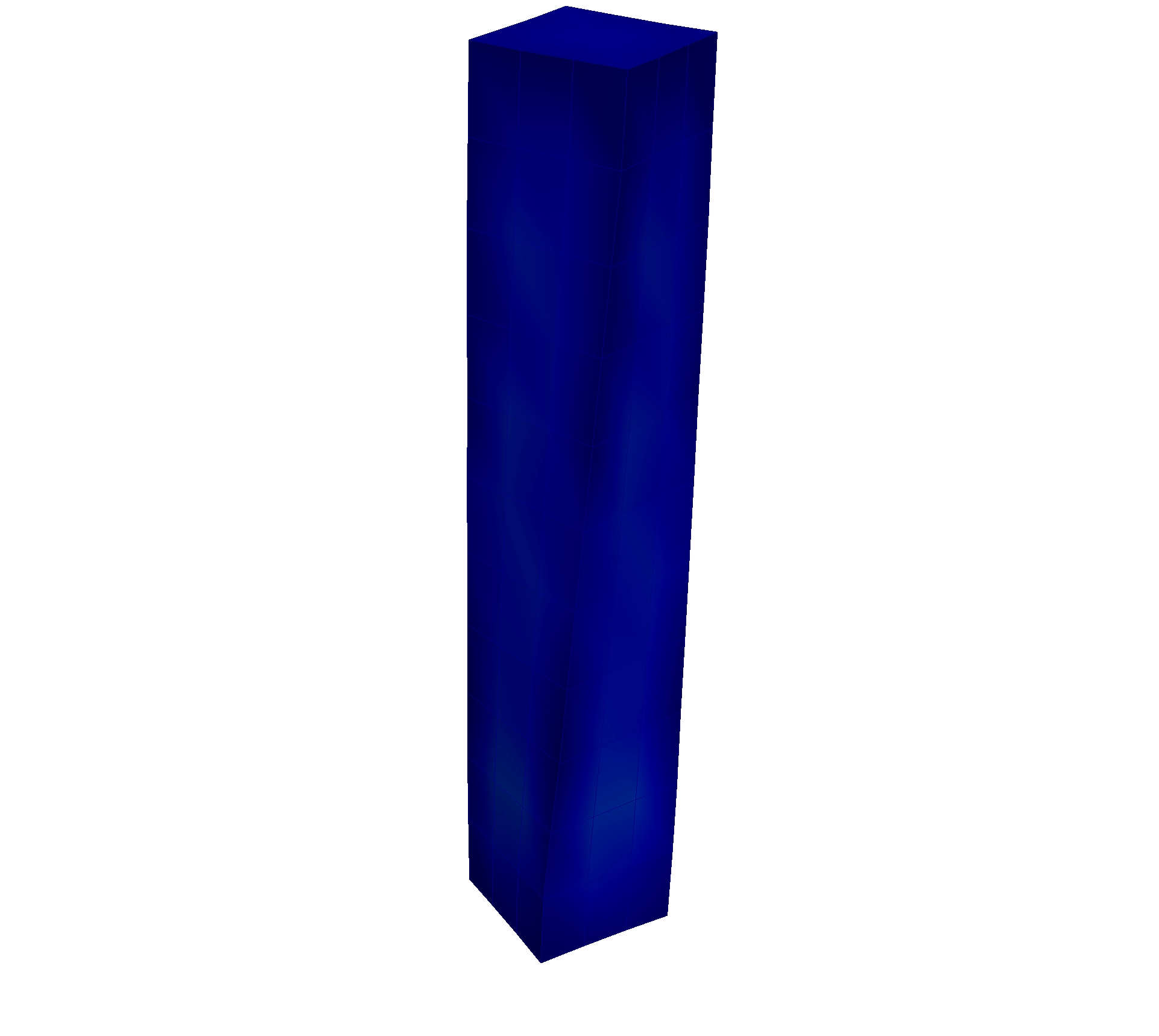} &
\includegraphics[angle=0, trim=450 120 450 0, clip=true, scale = 0.045]{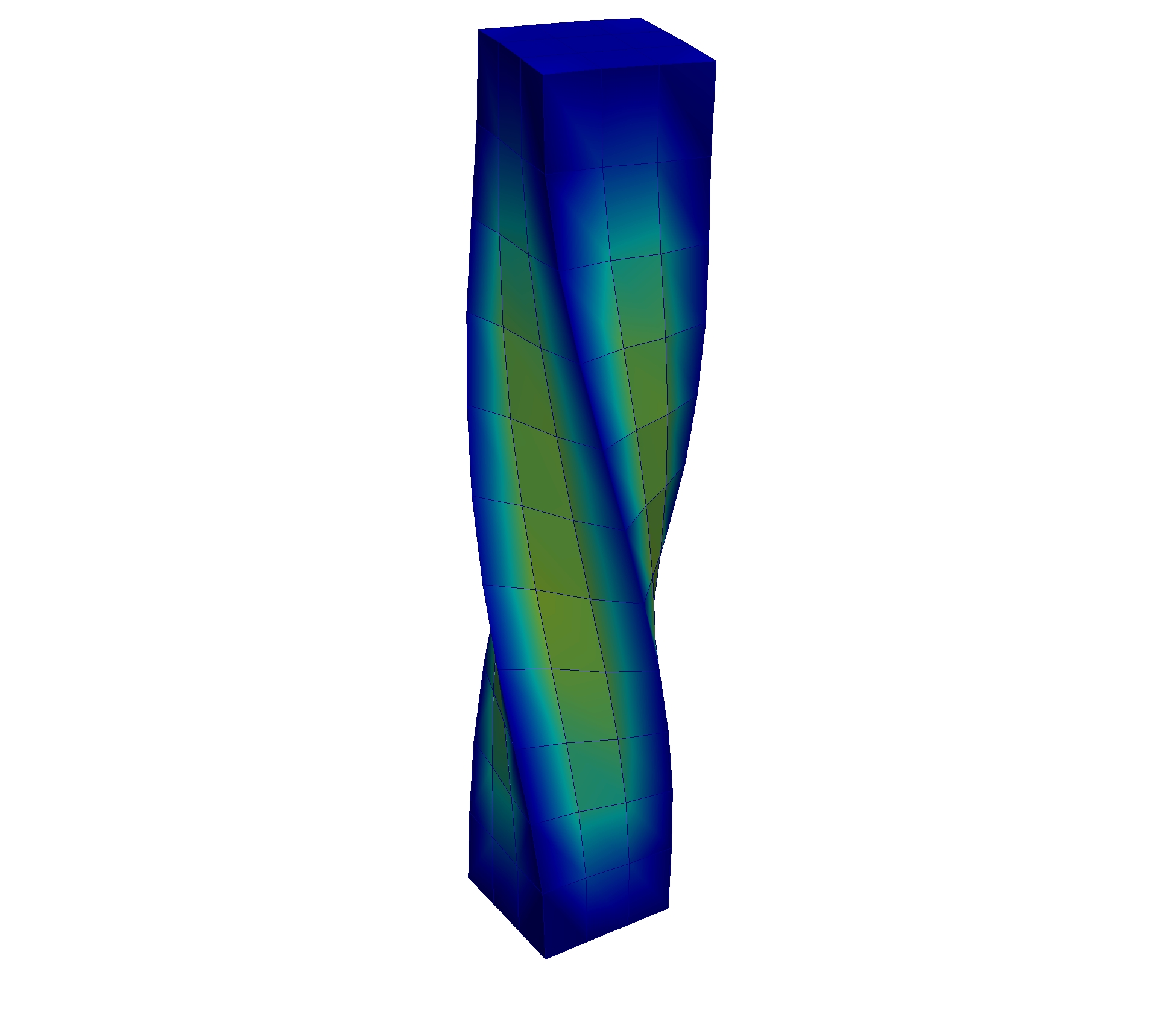} &
\includegraphics[angle=0, trim=450 120 450 0, clip=true, scale = 0.045]{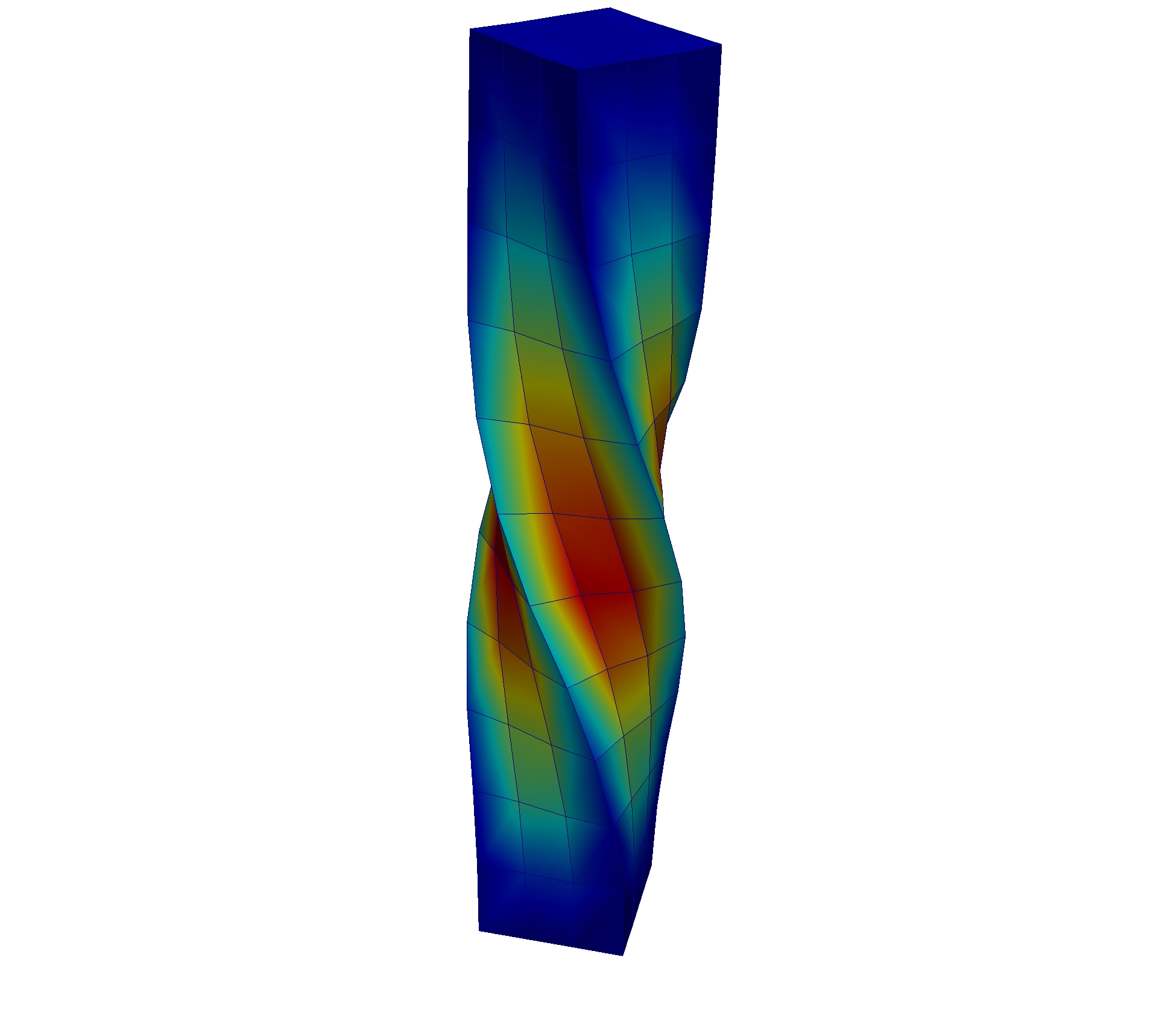} &
\includegraphics[angle=0, trim=450 120 450 0, clip=true, scale = 0.045]{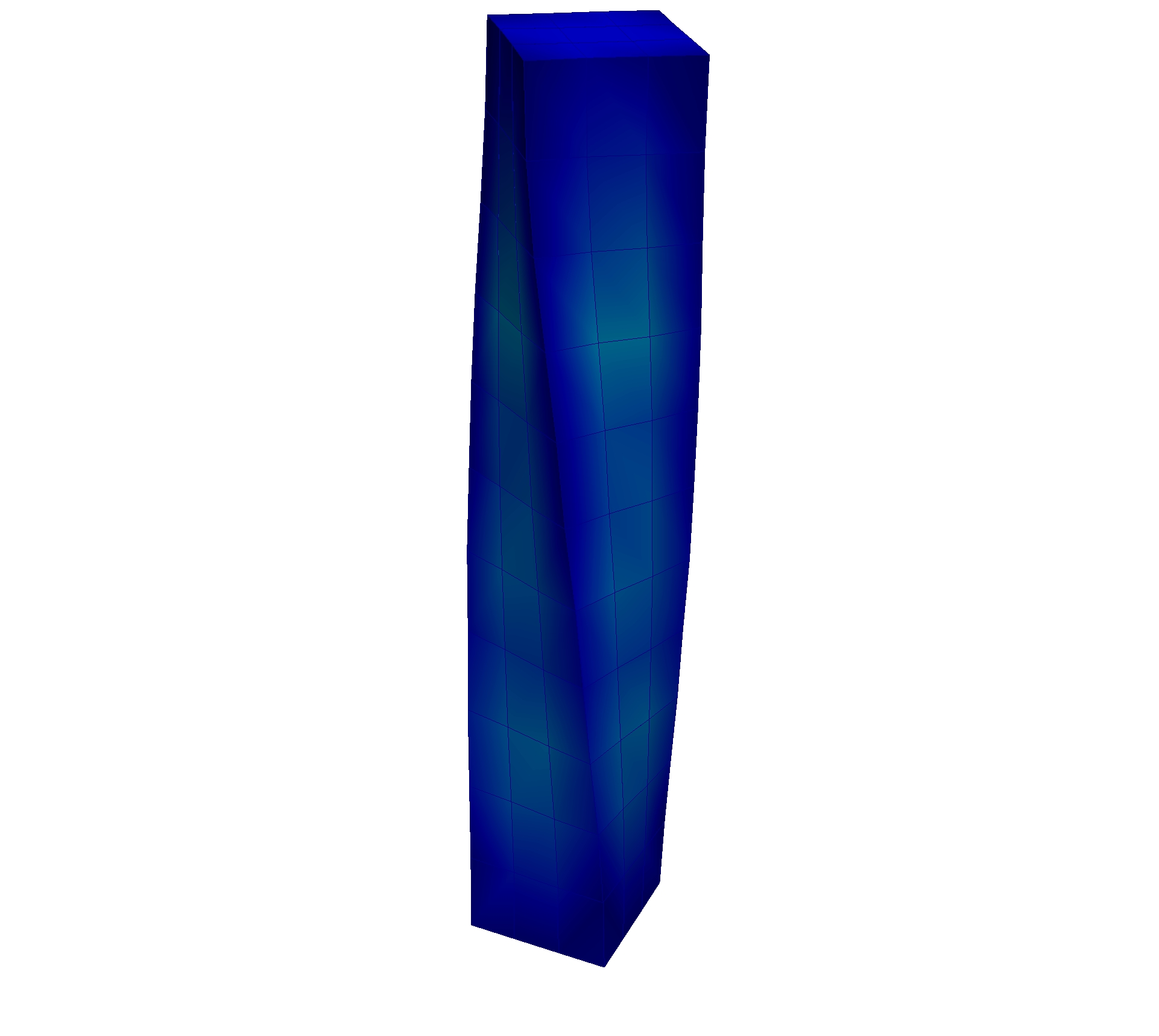} &
\includegraphics[angle=0, trim=450 120 450 0, clip=true, scale = 0.045]{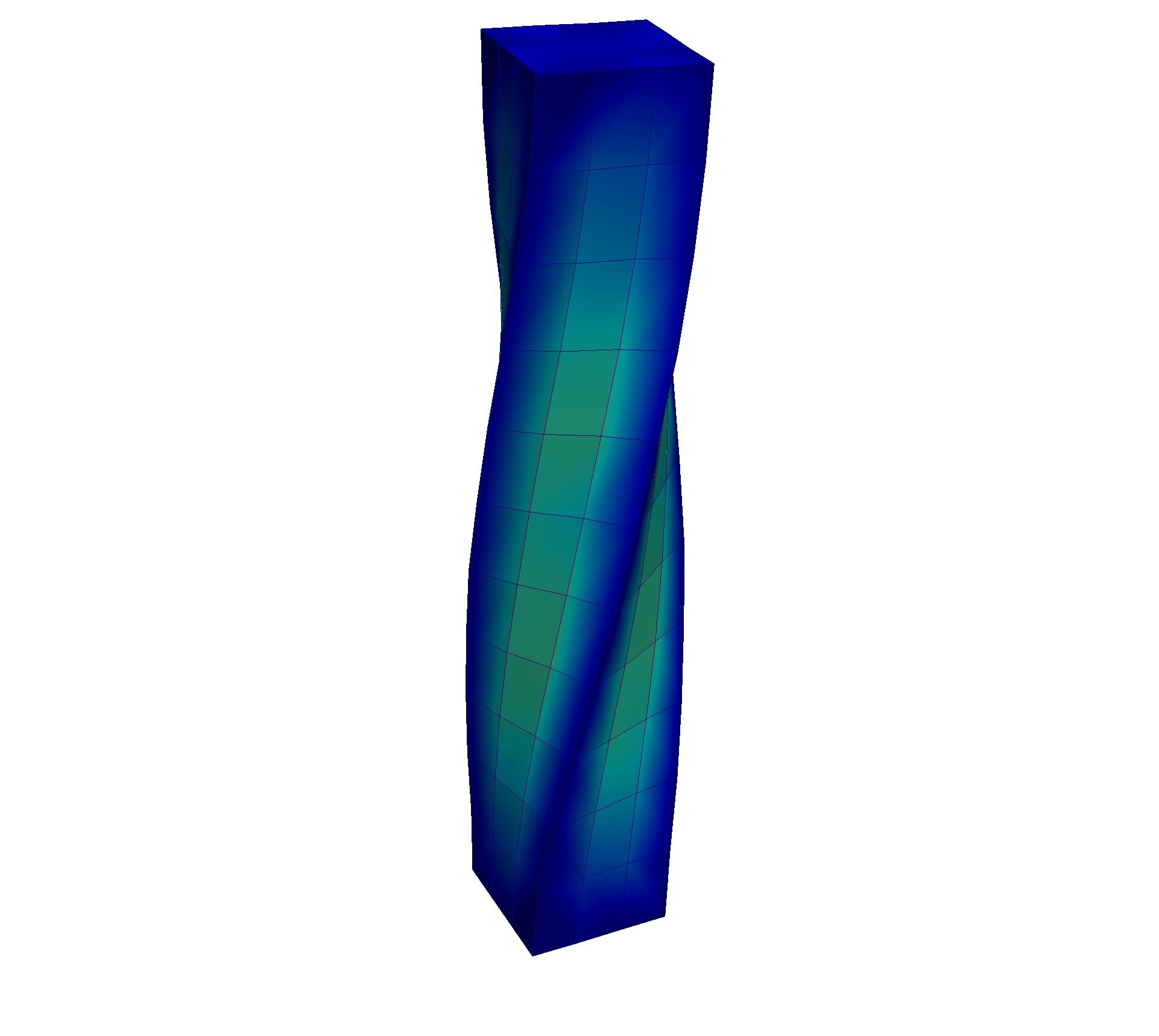} &
\includegraphics[angle=0, trim=450 120 450 0, clip=true, scale = 0.045]{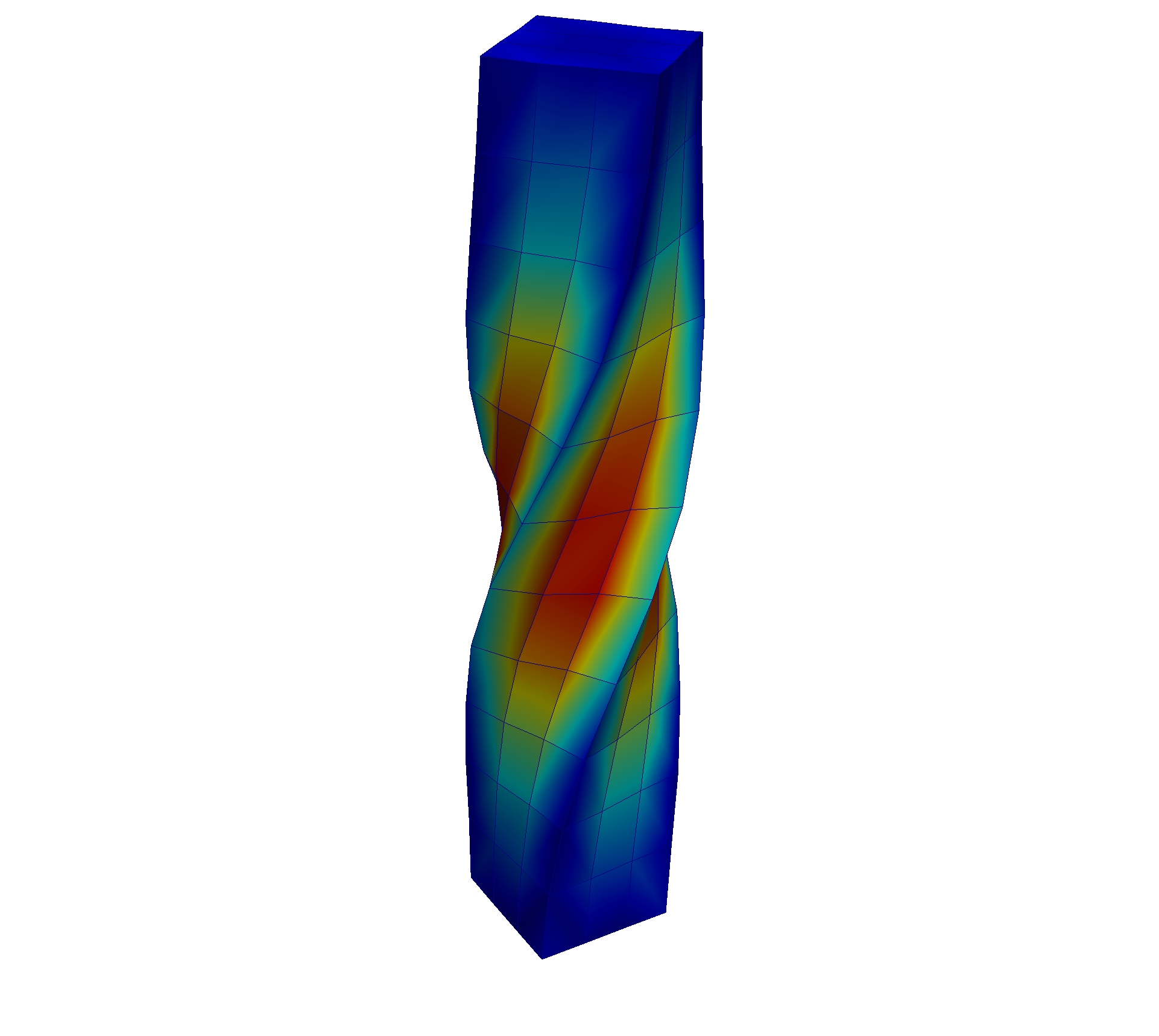} \\
t = 0.1 & 0.2 & 0.3 & 0.4 & 0.5 & 0.6 & 0.7 & 0.8 \\
\multicolumn{8}{c}{ \includegraphics[angle=0, trim=300 650 300 700, clip=true, scale = 0.25]{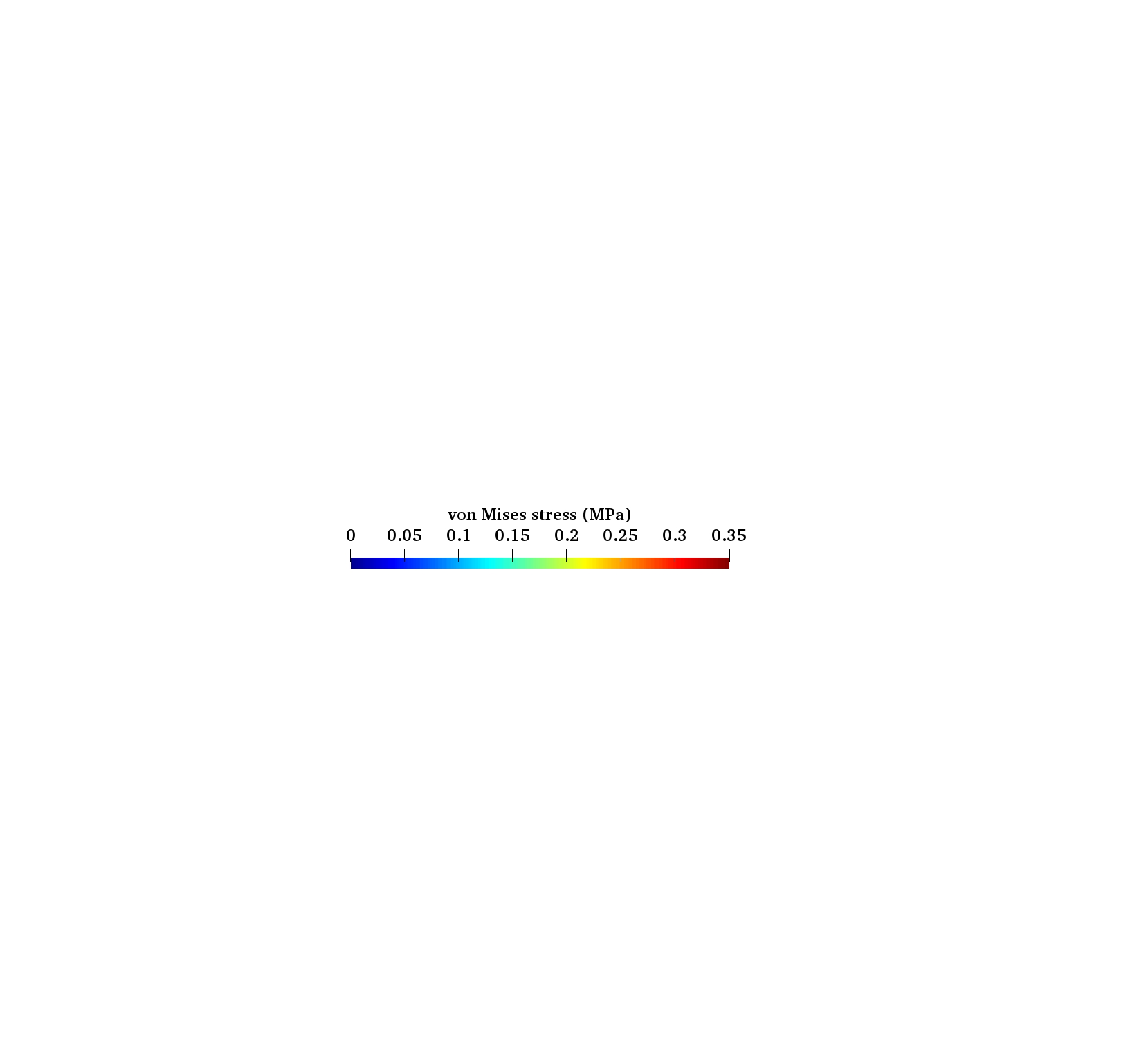} }
\end{tabular}
\caption{Snapshots of the von Mises stress on the current configuration using the coarse mesh (top) and fine mesh (bottom) with $\Delta t = 0.01$. The coarse mesh consists of $3 \times 3 \times 11$ elements with $\mathsf{p}=1$; the fine mesh consists of $3 \times 3 \times 11$ elements with $\mathsf{p}=2$.}
\label{fig:column_deformation}
\end{center}
\end{figure}

\begin{figure}
\begin{center}
\begin{tabular}{cc}
\includegraphics[angle=0, trim=50 180 100 100, clip=true, scale = 0.21]{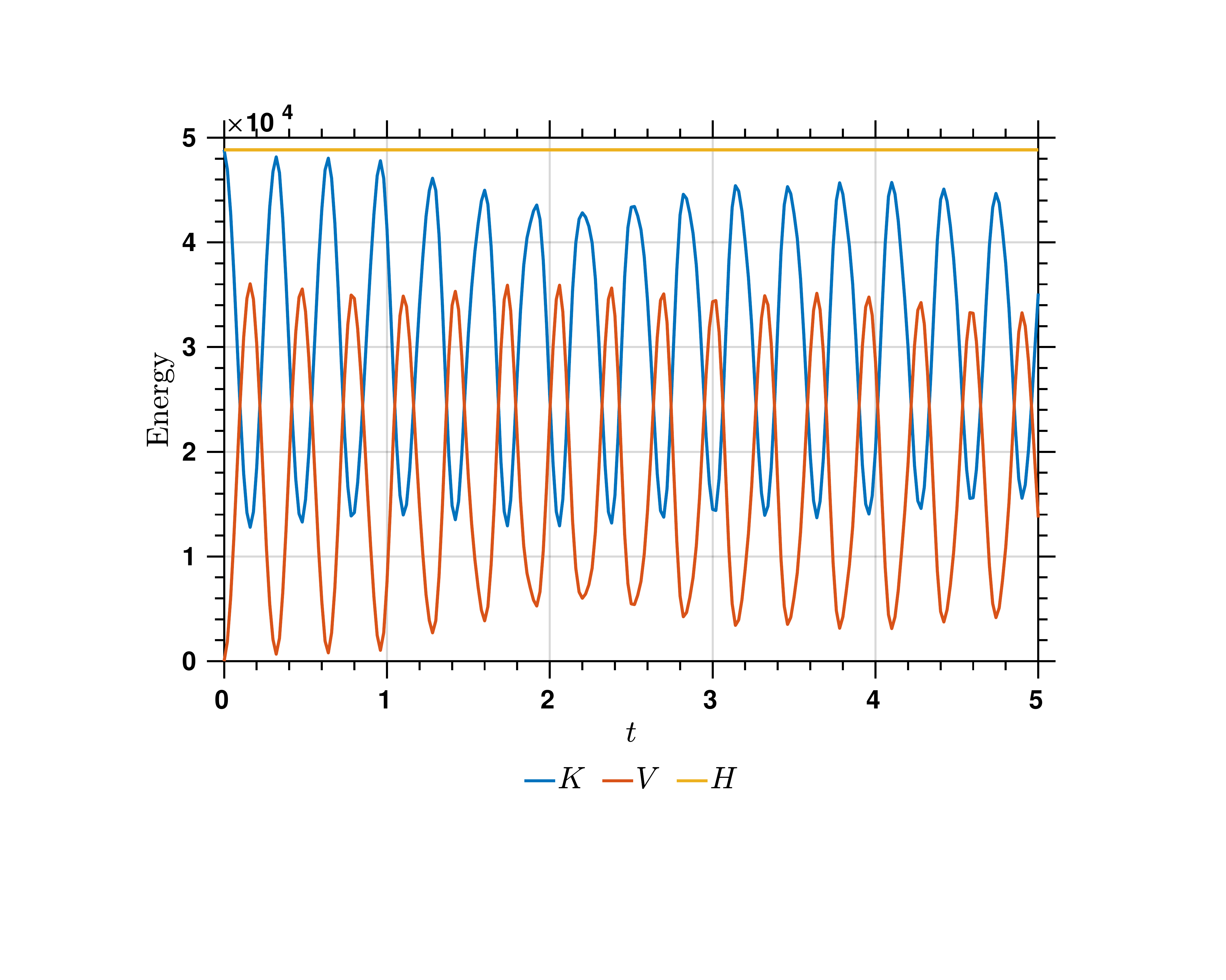} &
\includegraphics[angle=0, trim=50 180 100 100, clip=true, scale = 0.21]{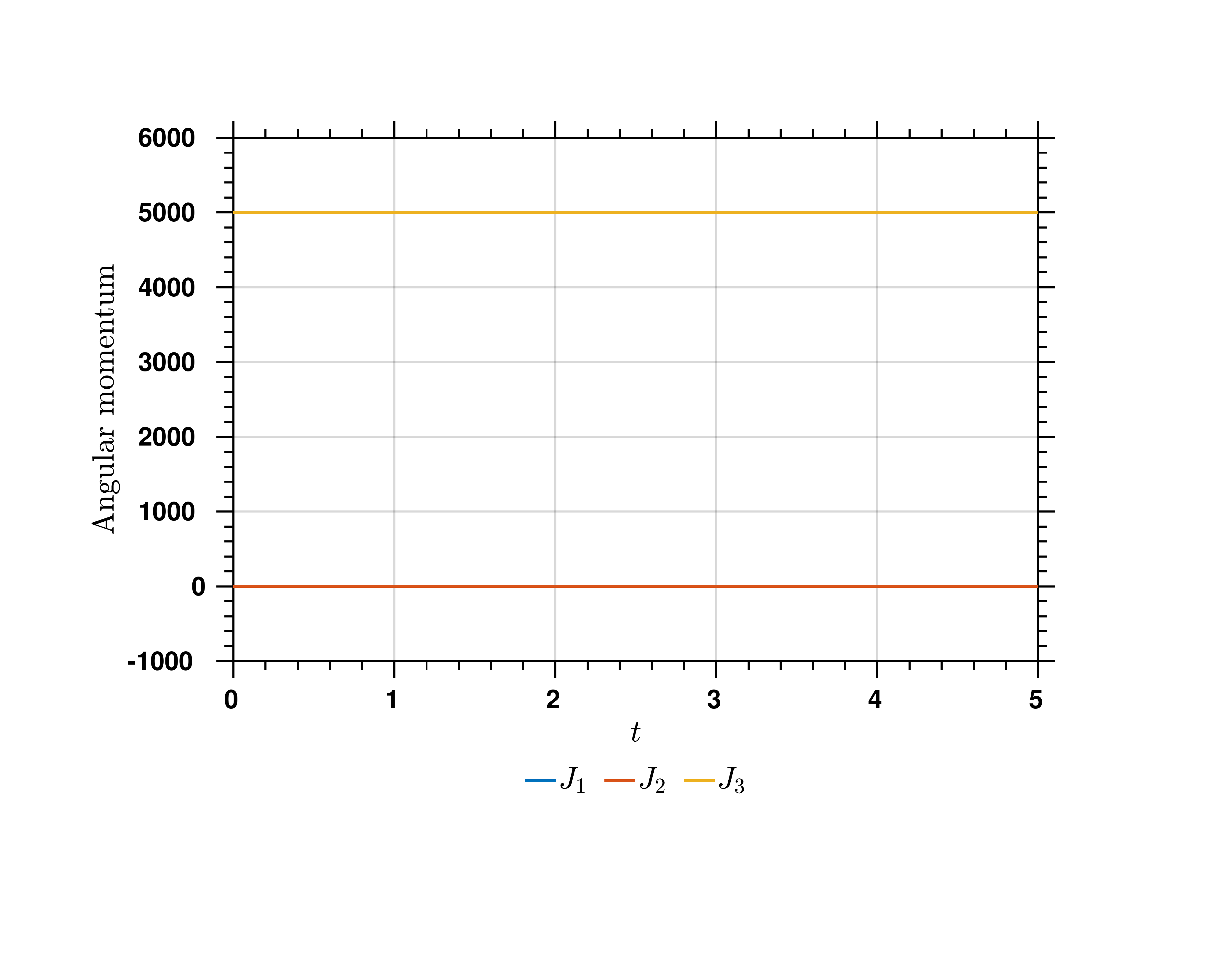} \\
\end{tabular}
\caption{The time histories of the total Hamiltonian (left) and angular momentum (right) with $T=5$ and $\Delta t = 0.01$.}
\label{fig:column_energy_momenta}
\end{center}
\end{figure}

\begin{figure}
\begin{center}
\begin{tabular}{cc}
\includegraphics[angle=0, trim=50 85 100 120, clip=true, scale = 0.21]{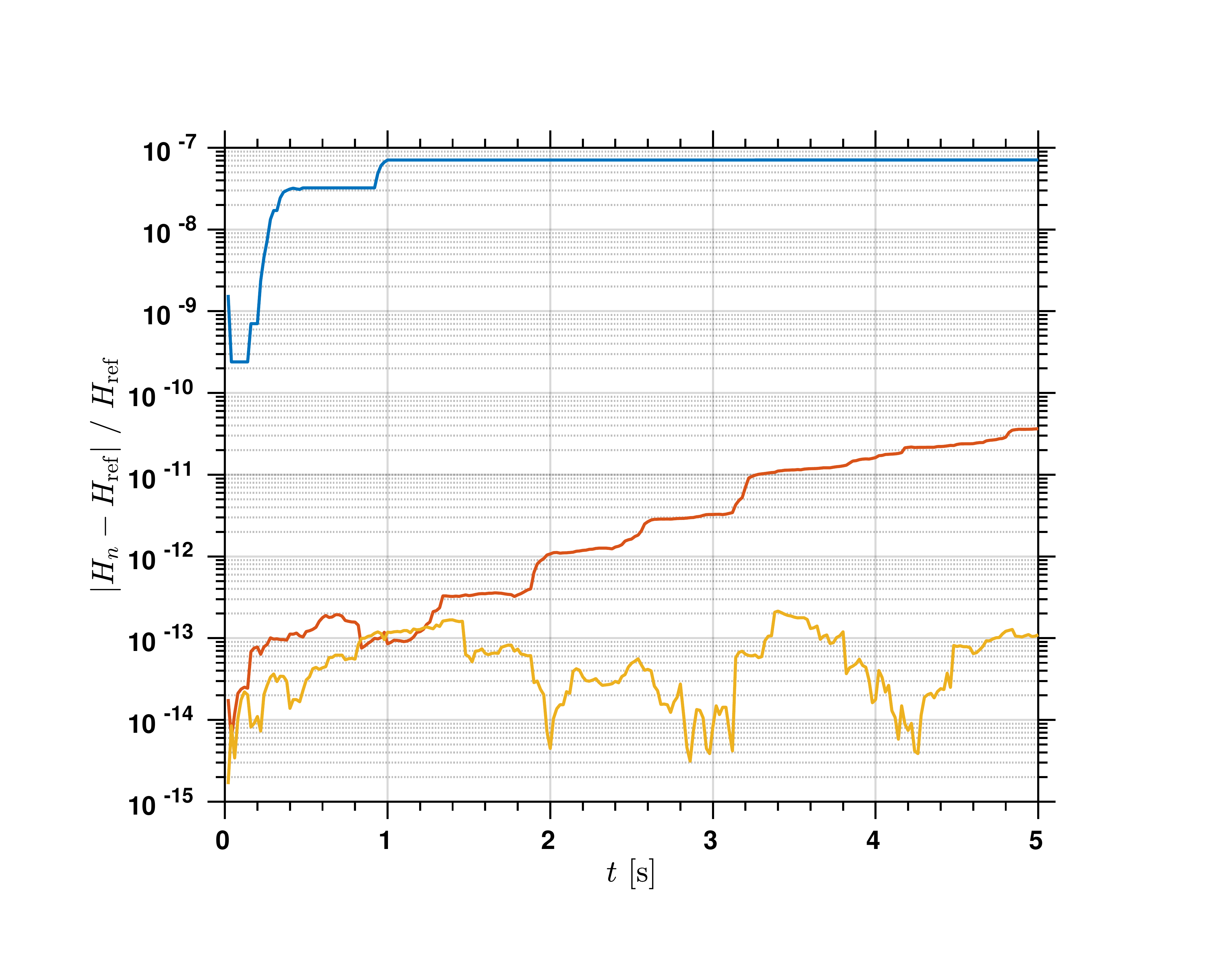} &
\includegraphics[angle=0, trim=50 85 100 120, clip=true, scale = 0.21]{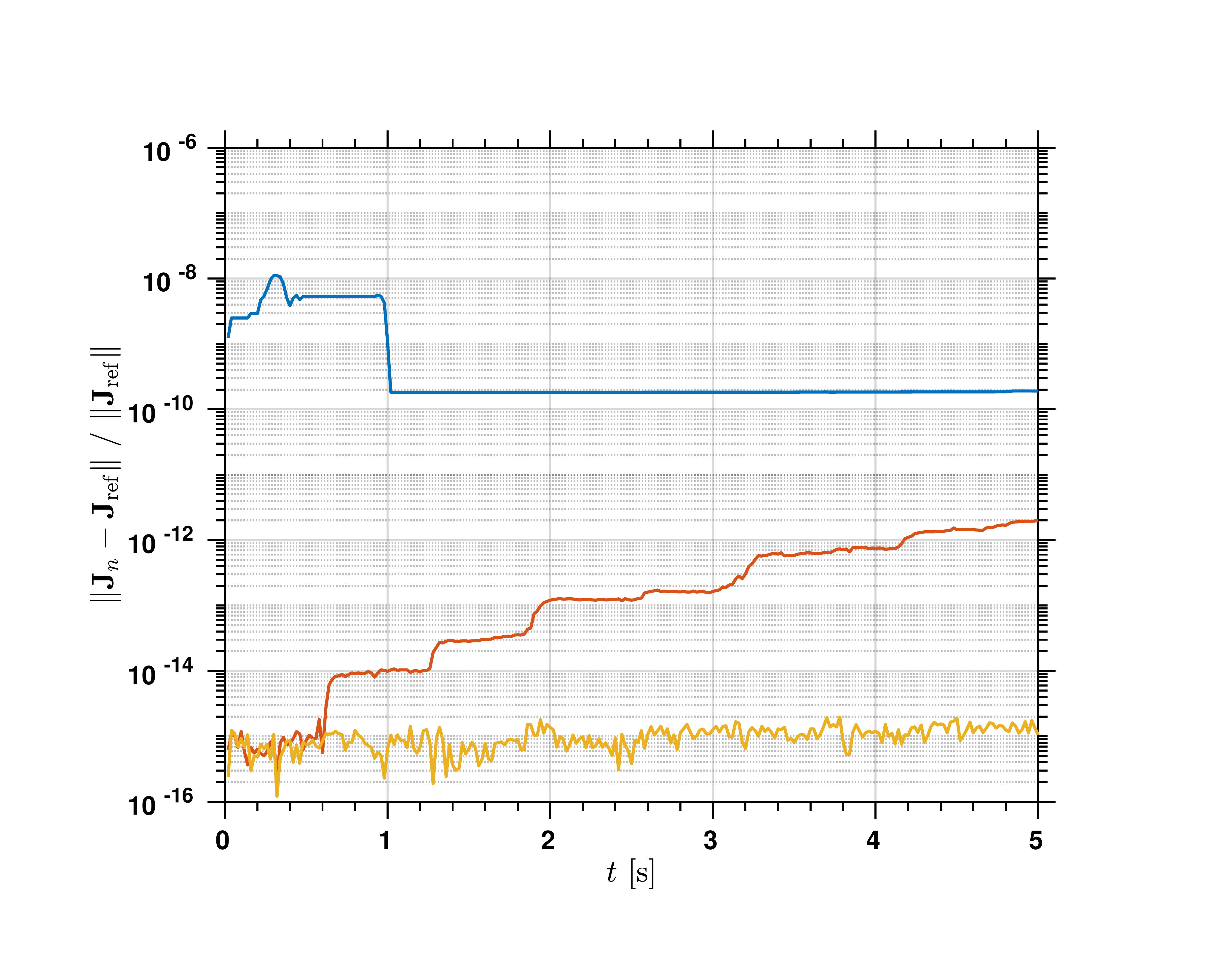} \\
\multicolumn{2}{c}{ \includegraphics[angle=0, trim=50 175 100 735, clip=true, scale = 0.30]{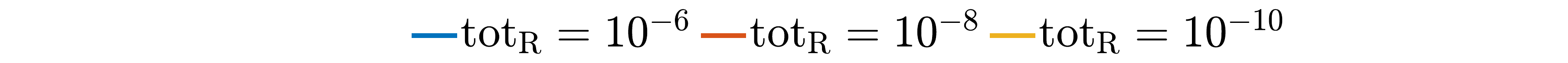} }
\end{tabular}
\caption{The time histories of the relative errors of the total Hamiltonian (left) and angular momentum (right) with $T=5$ and $\Delta t = 0.01$. The reference values of the quantities are taken at time $t=0$.}
\label{fig:column_energy_momenta_relative_errors}
\end{center}
\end{figure}

\begin{figure}
\begin{center}
\begin{tabular}{cc}
\includegraphics[angle=0, trim=50 85 100 120, clip=true, scale = 0.21]{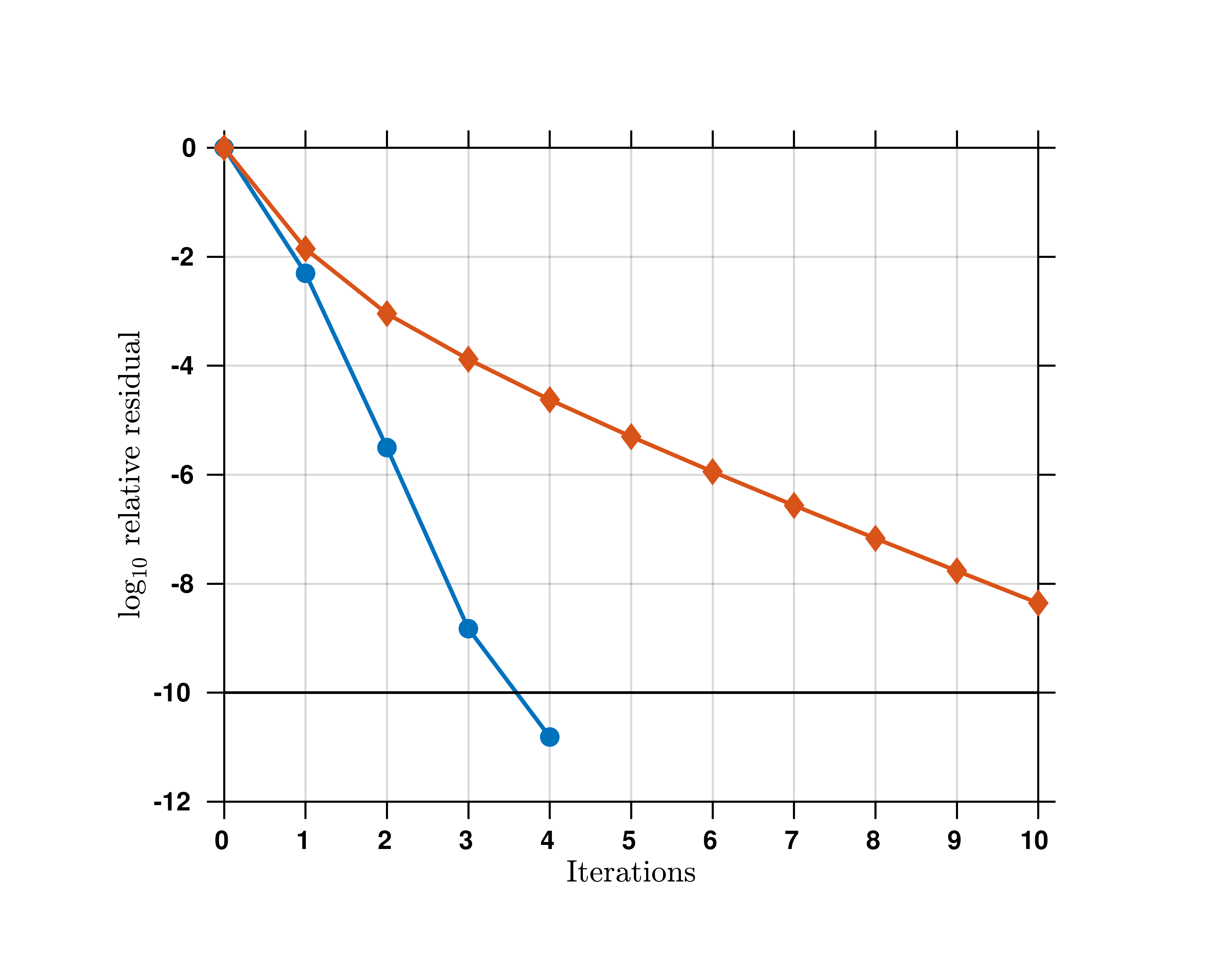} &
\includegraphics[angle=0, trim=50 85 100 120, clip=true, scale = 0.21]{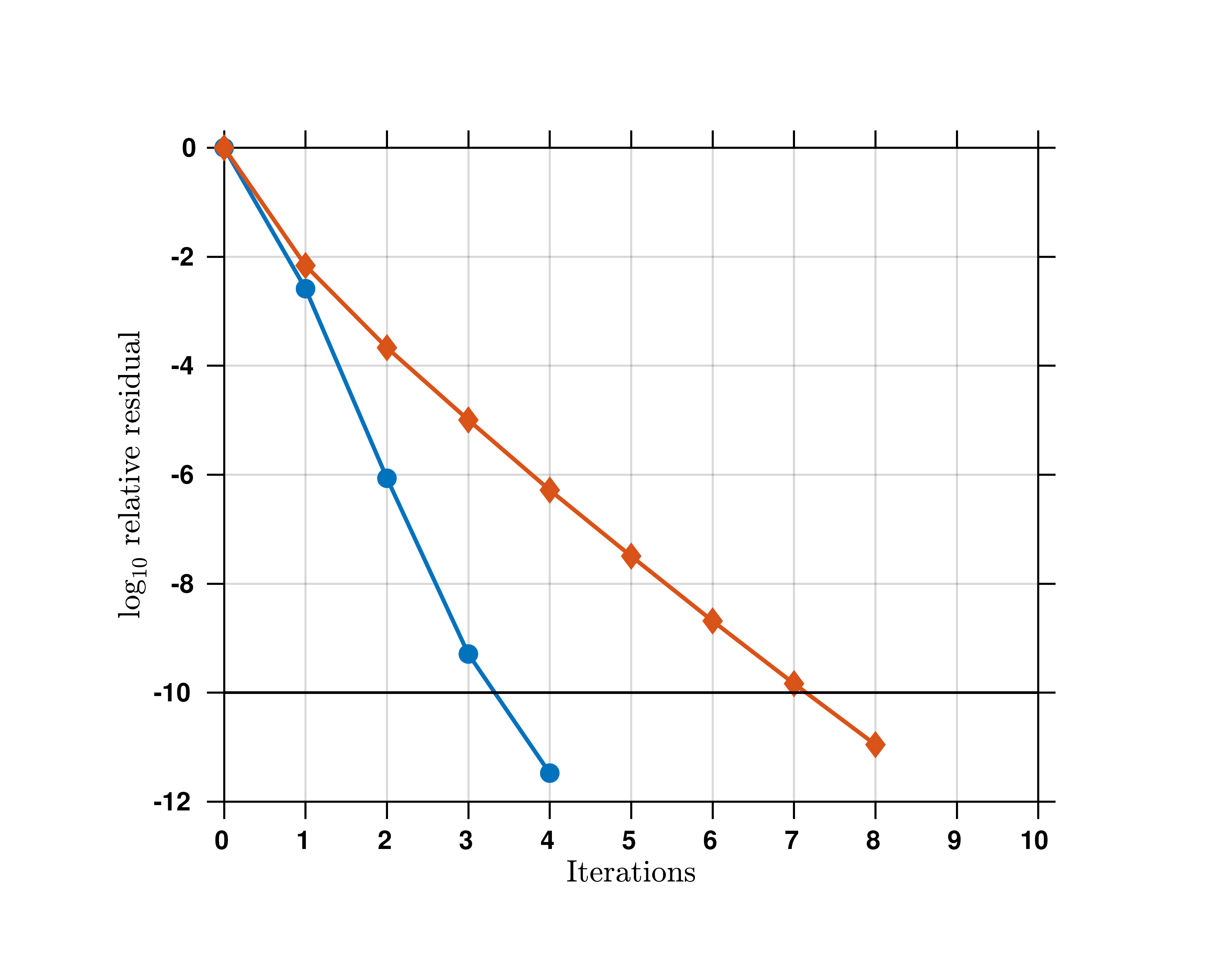} \\
(a) & (b ) \\
\multicolumn{2}{c}{ \includegraphics[angle=0, trim=100 170 100 740, clip=true, scale = 0.30]{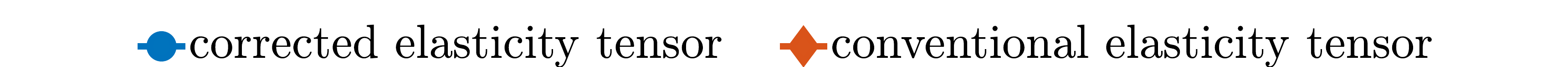} }
\end{tabular}
\caption{Comparison of the two elasticity tensors in the Newton-Raphson solution procedure: (a) the $11$th time step, (b) the $92$nd time step.}
\label{fig:convergence_NT}
\end{center}
\end{figure}

\begin{figure}
\begin{center}
\begin{tabular}{cc}
\includegraphics[angle=0, trim=90 85 120 110, clip=true, scale = 0.22]{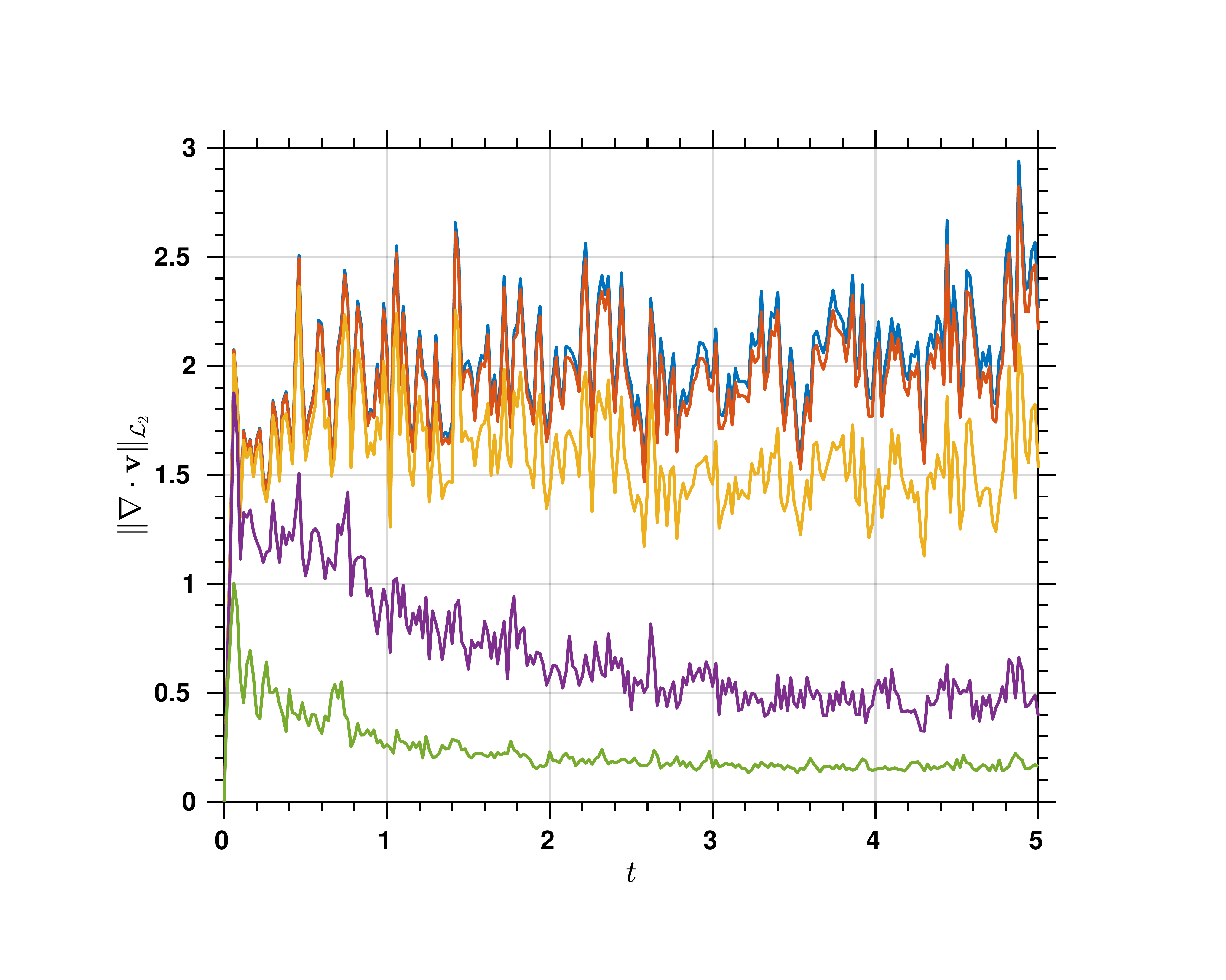} &
\includegraphics[angle=0, trim=90 85 120 110, clip=true, scale = 0.22]{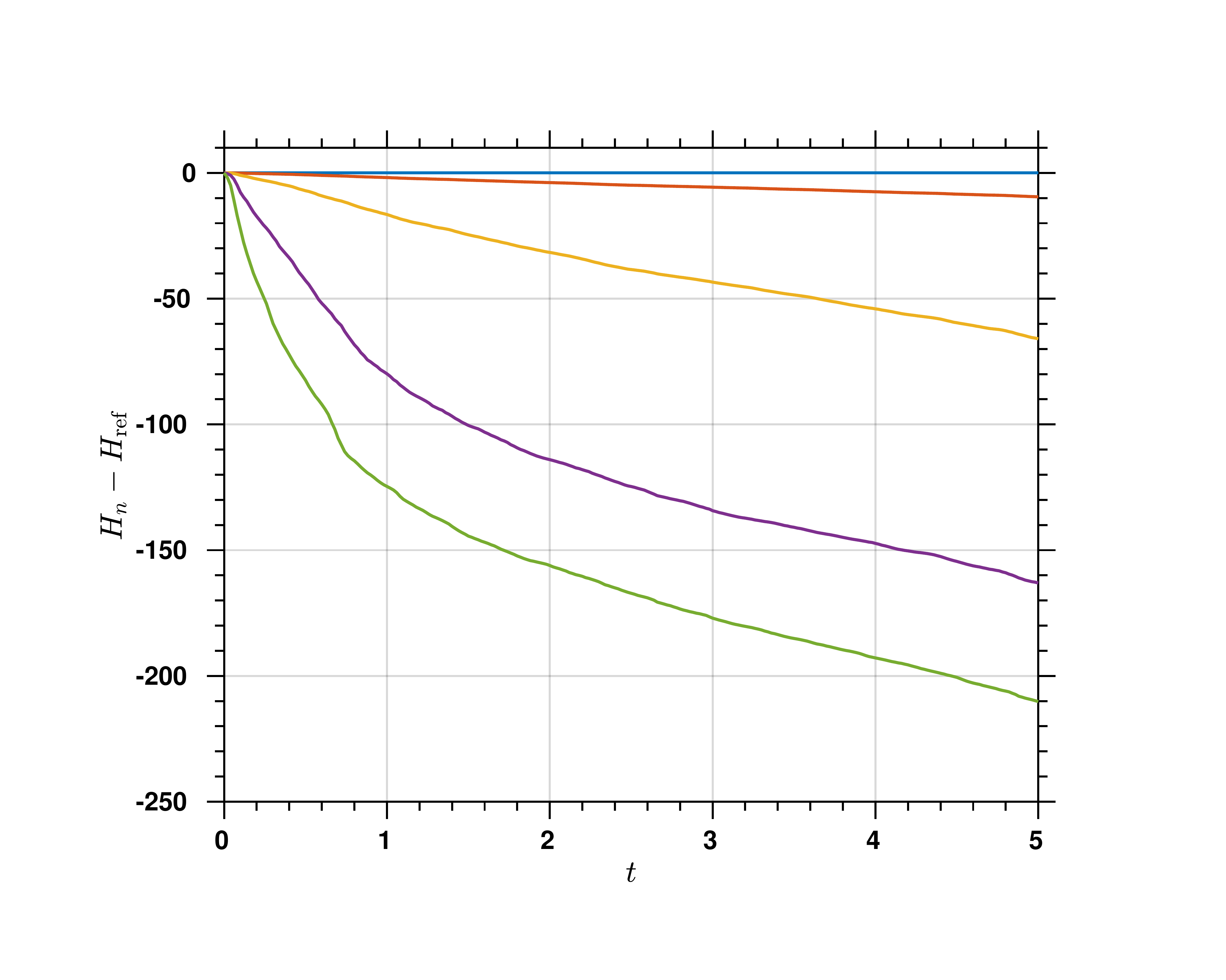} \\
(a) & (b) \\
\includegraphics[angle=0, trim=90 85 120 110, clip=true, scale = 0.22]{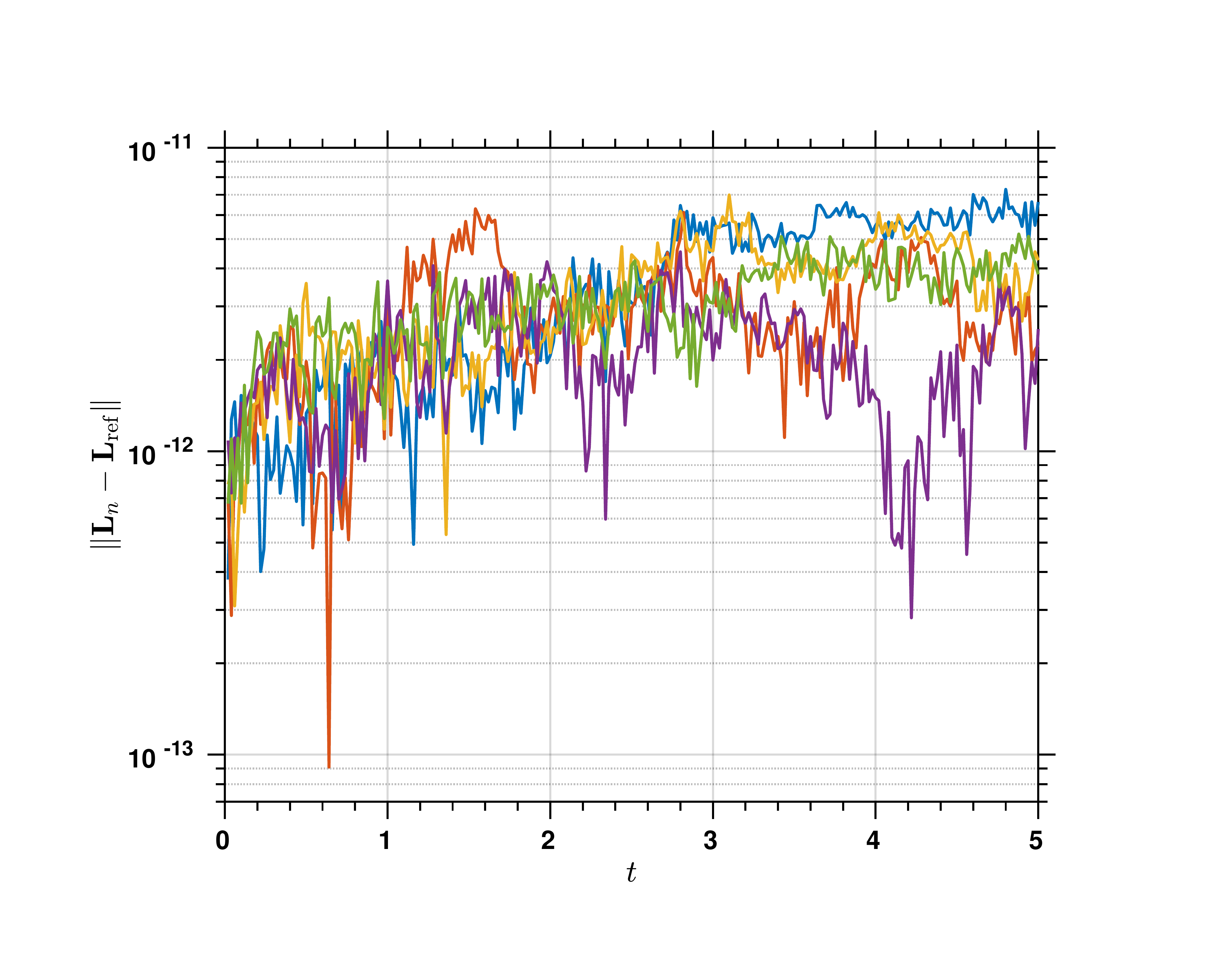} &
\includegraphics[angle=0, trim=90 85 120 110, clip=true, scale = 0.22]{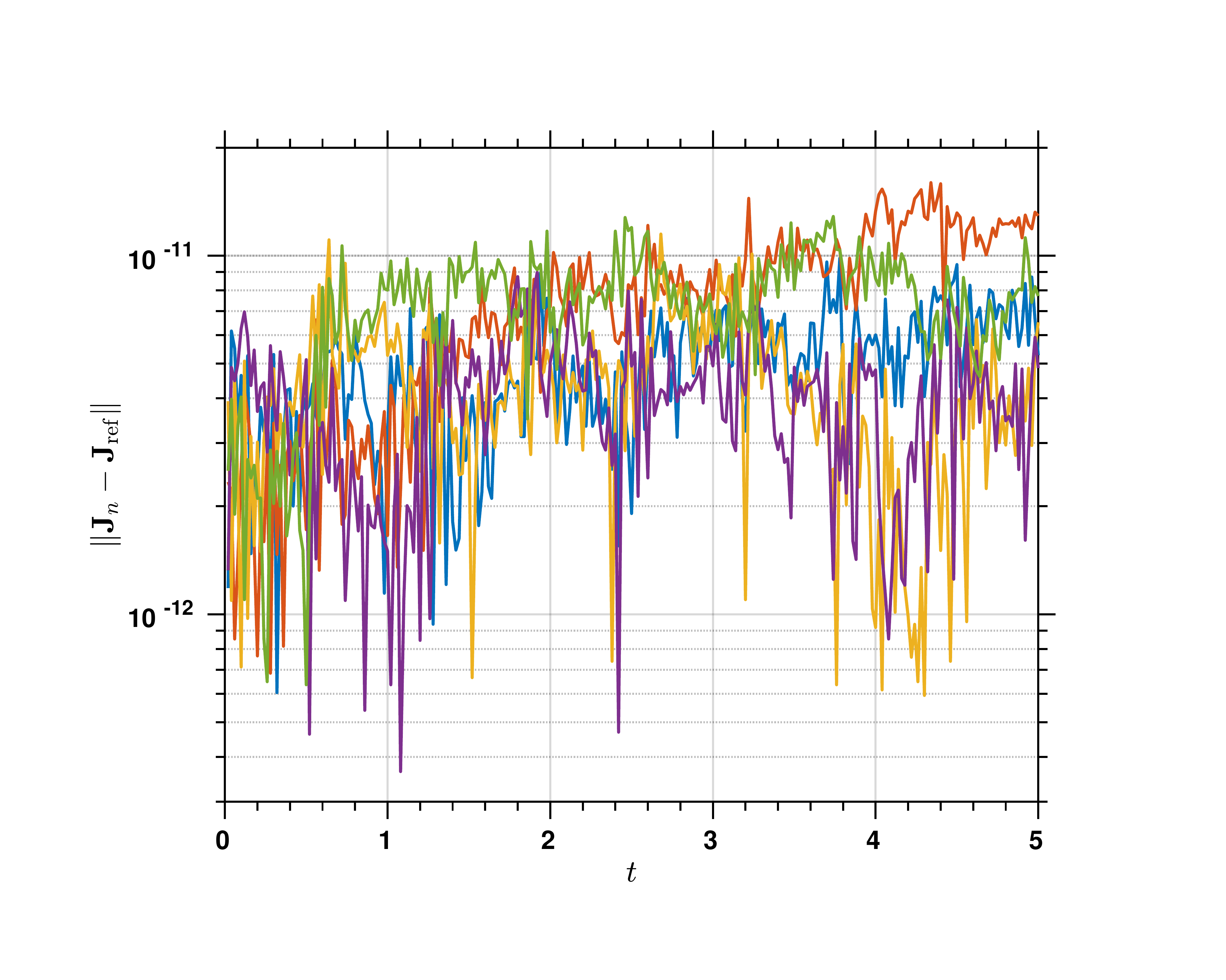} \\
(c) & (d) \\
\multicolumn{2}{c}{ \includegraphics[angle=0, trim=0 165 0 750, clip=true, scale = 0.35]{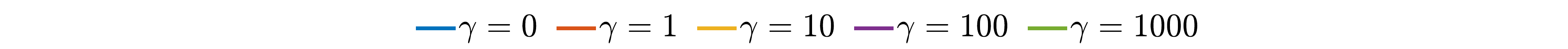} }
\end{tabular}
\caption{The effect of the grad-div stabilization parameter $\gamma$ for the twisting column: (a) $\| \nabla \cdot \bm{v} \|_{\mathcal{L}_2}$, (b) the absolute errors of the total Hamiltonian, (c) the absolute errors of the total linear momentum, and (d) the absolute errors of the angular momentum over time.}
\label{fig:column_graddiv}
\end{center}
\end{figure}

Snapshots of the von Mises stress on the deformed shapes at different time instances are depicted in Figure \ref{fig:column_deformation}. We observe almost indistinguishable stress results from the two different meshes. Considering that the coarse mesh consists of only $99$ elements with $\mathsf p = 1$, the superior stress resolving capability of the element technology is again confirmed. In Figure \ref{fig:column_energy_momenta}, the total Hamiltonian and three components of the total angular momentum are plotted over time; in Figure \ref{fig:column_energy_momenta_relative_errors}, the impact of the tolerance $\mathrm{tol}_R$ on the solution quality is investigated. Under the default setting (i.e., $\mathrm{tol}_{\mathrm{R}} = 10^{-10}$), the relative errors of the Hamiltonian and the angular momentum are of the orders of $10^{-13}$ and $10^{-15}$, respectively. When setting $\mathrm{tol}_{\mathrm{R}} = 10^{-8}$ and $\mathrm{tol}_{\mathrm{R}} = 10^{-6}$, the magnitudes of the relative errors grows by several order of magnitude. This indicates the discrete conservation properties from the designed algorithm is contingent upon the accuracy of the Newton-Raphson solution procedure. If the nonlinear solver cannot solve the algebraic equations accurately, the discrete conservation properties will get polluted. Based on this fact, we adopt this benchmark to examine the elasticity tensor for the Ogden model. As was mentioned in Section \ref{sec:hyperelasticity}, there exist a few missing terms in the existing formula. We refer to that formula and the formula \eqref{eq:CC-rate-form}-\eqref{eq:CC-spectral-form} as the \textit{conventional} elasticity tensor and the \textit{corrected} elasticity tensor, respectively. Although the missing term does not have a significant impact on the convergence rate for most of the time, there are indeed scenarios when the conventional elasticity tensor demands more iterations (Figure \ref{fig:convergence_NT} (b)) or fail to satisfy the convergence criteria within the maximum number of iterations (Figure \ref{fig:convergence_NT} (a)). This signifies the importance in maintaining the consistency of the elasticity tensor in the Newton-Raphson solution procedure when one performs structure-preserving simulations.

The effects of the grad-div stabilization are displayed in Figure \ref{fig:column_graddiv}. The parameter $\gamma$ varies from 0 to $10^3$. In Figure \ref{fig:column_graddiv} (a), the $\mathcal{L}_2$-norm of $\nabla \cdot \bm{v}$ decreases with the increase of the parameter $\gamma$, and when $\gamma$ gets larger than 100, the reduction of $\mathcal{L}_2$-norm of $\nabla \cdot \bm{v}$ is more significant. Figure \ref{fig:column_graddiv} (b) displays the dissipation effect of the grad-div stabilization term over time. In Figure \ref{fig:column_graddiv} (c) and (d), it can be observed that the total momenta are conserved regardless of the grad-div stabilization.

\subsection{Vibrating cantilever}
In this example, we investigate the chaotic dynamics of a three-dimensional vibrating cantilever under harmonic excitation \cite{Cao2006}. The problem setting is summarized in Table \ref{table:cantilever}. The displacement is fully clamped on the surface $X_3=0$, and we apply a harmonic load on the boundary surface $X_3=0.3$. The rest surfaces are traction-free. The problem is discretized by a mesh of $2 \times 2 \times 21$ elements with $\mathsf{p} = 1$. Simulations are integrated up to $T=250$ with $\Delta t = 0.01$ using the structure-preserving scheme and the generalized-$\alpha$ scheme. In the generalized-$\alpha$ scheme, we consider three options for the spectral radius of the amplification matrix at the highest mode, that is, $\rho_{\infty} = 1$, $0.5$, and $0$.

\begin{table}[htbp]
  \centering 
  \begin{tabular}{ m{.45\textwidth}   m{.45\textwidth} }
    \hline
    \begin{minipage}{.45\textwidth}
      \includegraphics[width=1.0\linewidth, trim=80 490 120 150, clip]{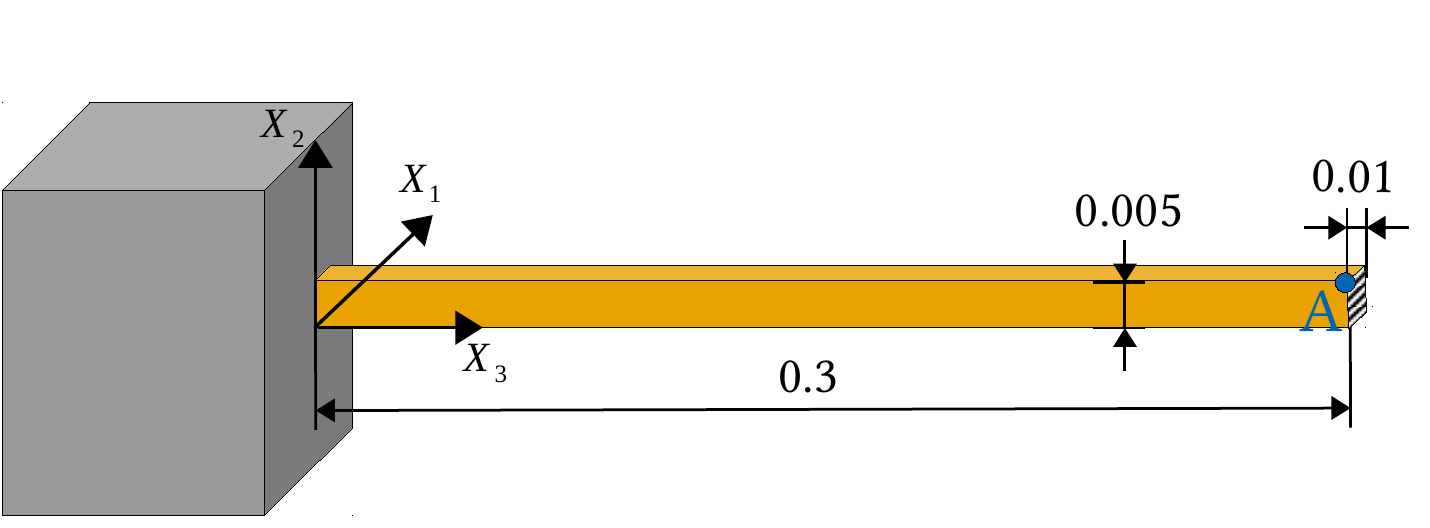}
    \end{minipage}
    &
    \begin{minipage}{.4\textwidth}
    \vspace{1mm}
      \begin{itemize}
        \item[] Material properties:
        \item[] $\rho_0 = 3.0\times 10^3$,
        \item[] $G_{\mathrm{ich}}( \tilde{\bm{C}} ) = \frac{\mu}{2}(\mathrm{tr}(\tilde{\bm{C}}) -3)$,
        \item[] $\mu = 6.93\times 10^7$.
         \item[] The traction on the surface $X_3=0.3$:
         \item[] $\bm{H} = [200 \times \mathrm{cos}(8t), ~100 \times \mathrm{sin}(8t), ~0]$.                
      \end{itemize}
      \vspace{1mm}
    \end{minipage}   
    \\
    \hline
  \end{tabular}
  \caption{The three-dimensional vibrating cantilever: problem definition. The position of the A point is $(0.0, 0.005, 0.3)$. A harmonic `dead' load acts on the surface $X_3=0.3$ denoted by a shadow area.}
\label{table:cantilever}
\end{table}

\begin{figure}
\begin{center}
\begin{tabular}{ccc}
\includegraphics[angle=0, trim=90 180 120 120, clip=true, scale = 0.14]{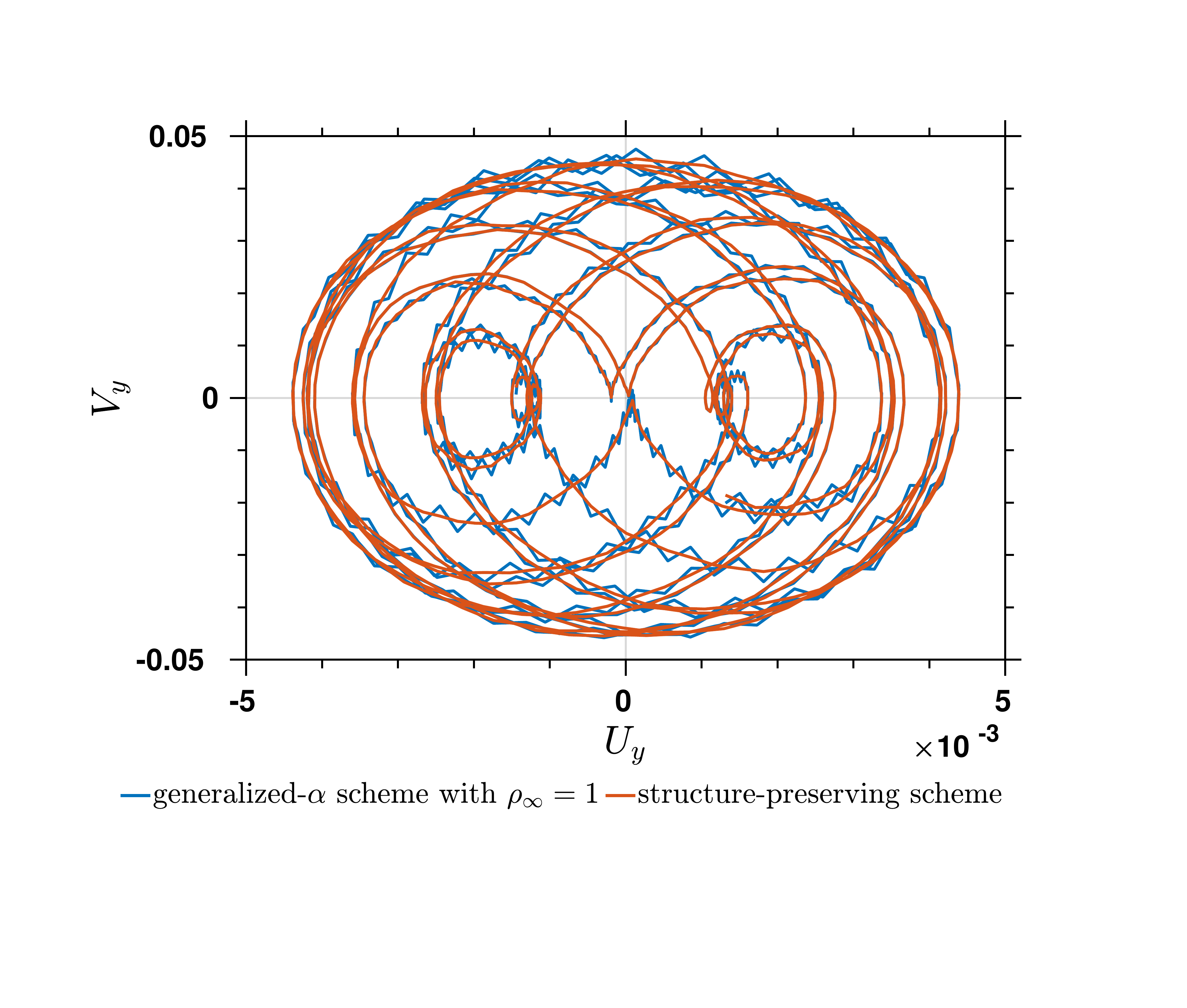} &
\includegraphics[angle=0, trim=90 180 120 120, clip=true, scale = 0.14]{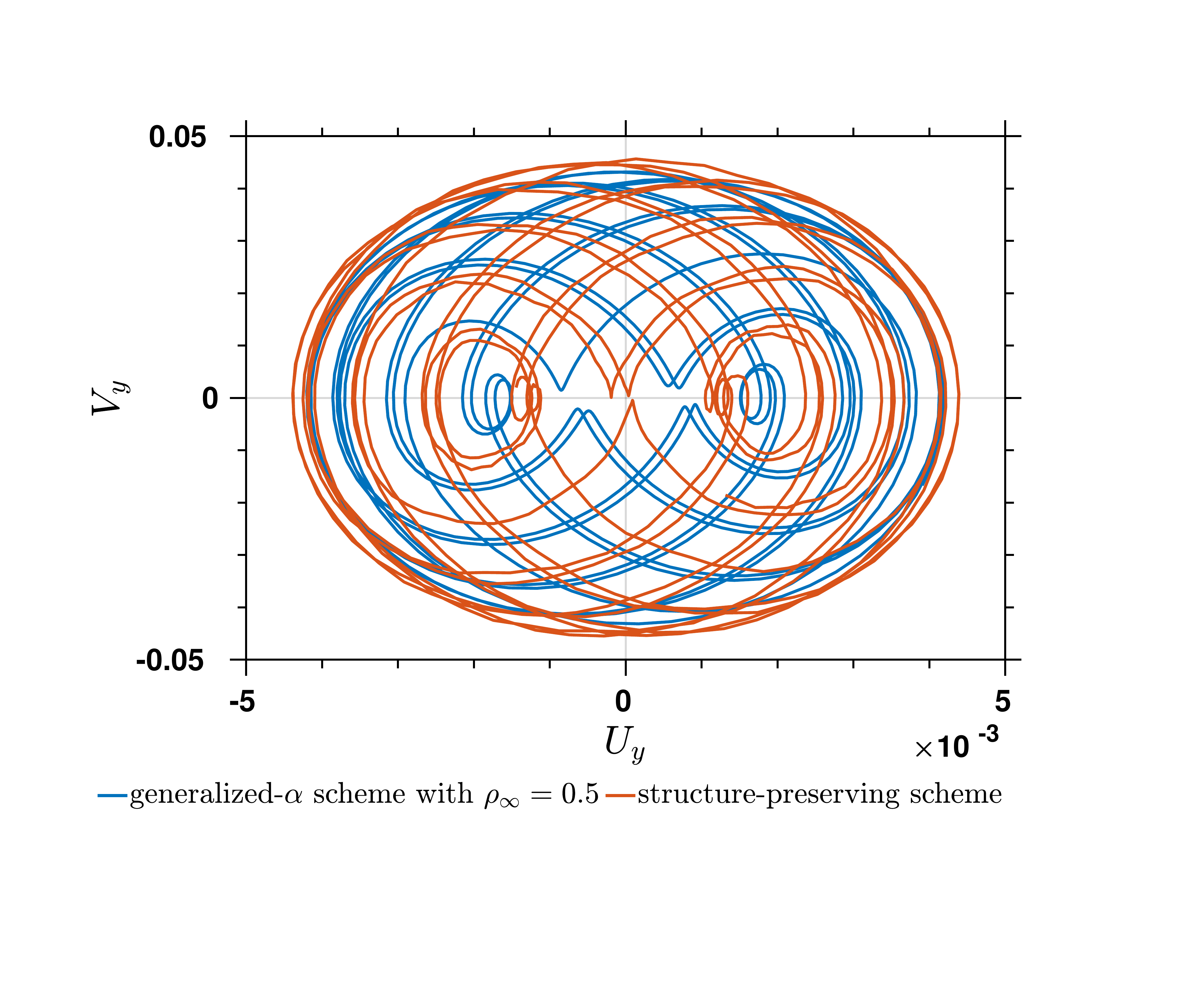} & 
\includegraphics[angle=0, trim=90 180 120 120, clip=true, scale = 0.14]{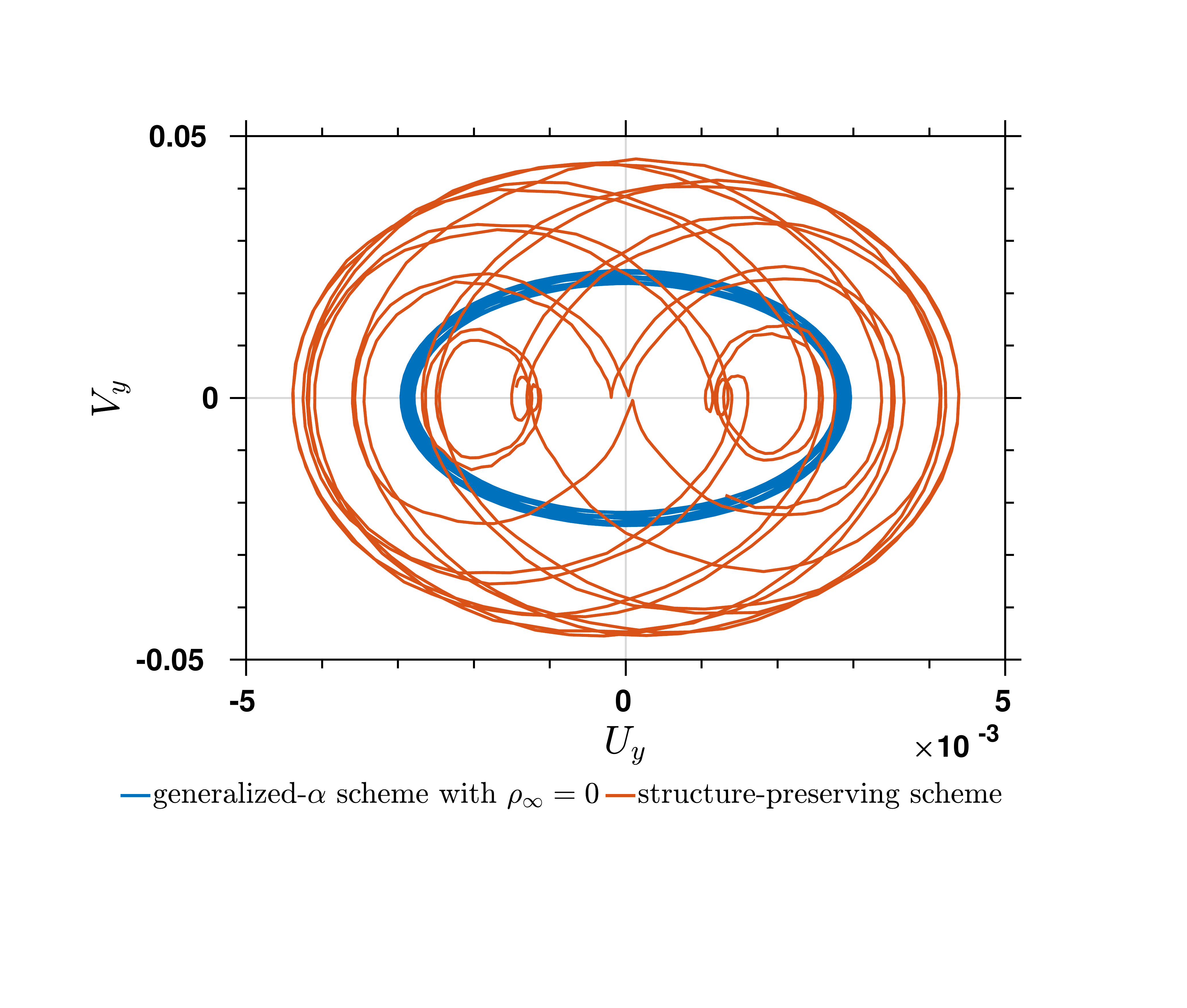} \\
(a) & (b) & (c) \\
\end{tabular}
\caption{The trajectories of point A in the displacement-velocity phase space obtained from different time schemes with $\Delta t = 0.01$.}
\label{fig:cantilever_phase_space}
\end{center}
\end{figure}

\begin{figure}
\begin{center}
\begin{tabular}{c}
\includegraphics[angle=0, trim=80 90 120 120, clip=true, scale = 0.35]{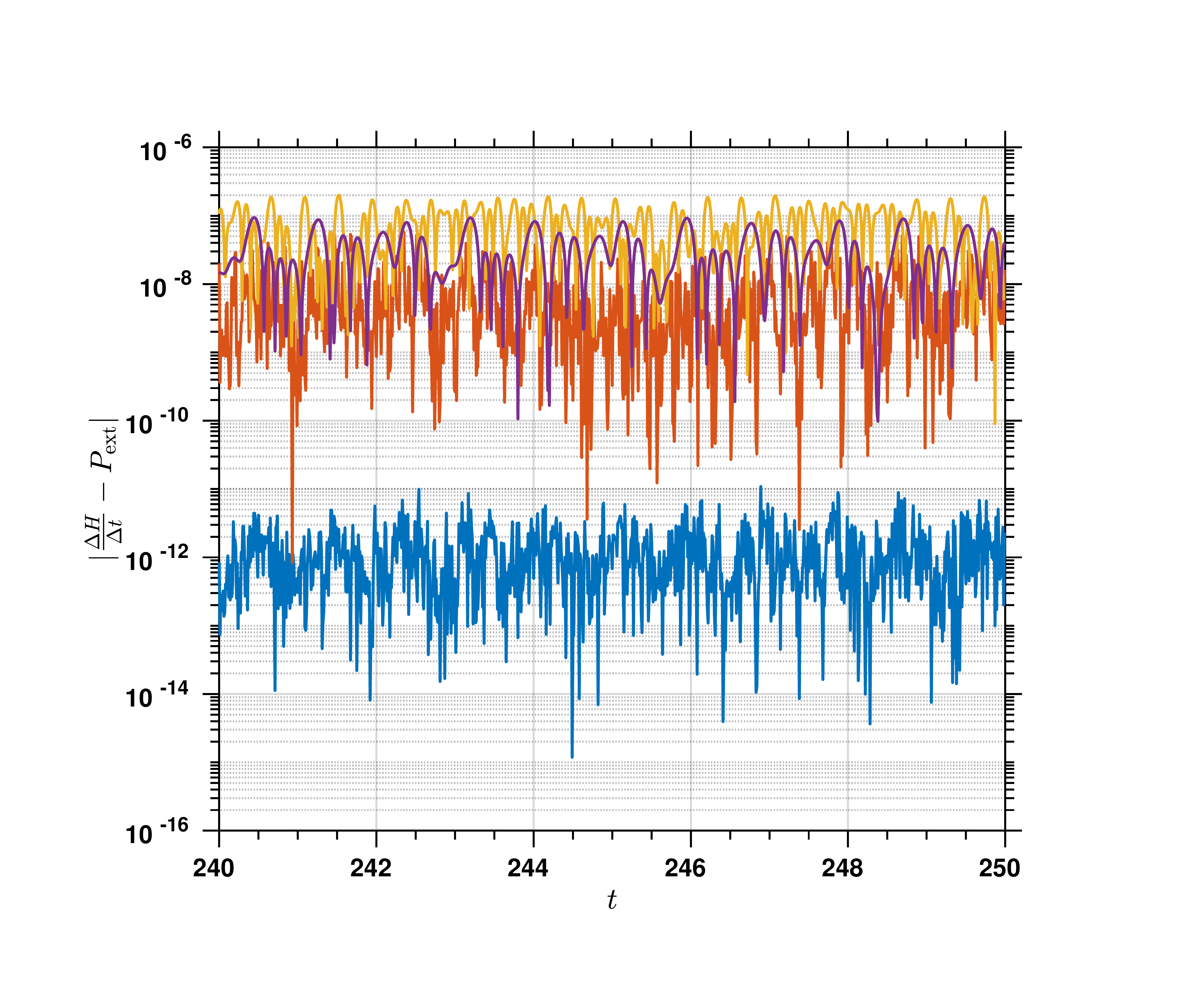} \\
\multicolumn{1}{c}{ \includegraphics[angle=0, trim=80 160 120 765, clip=true, scale = 0.30]{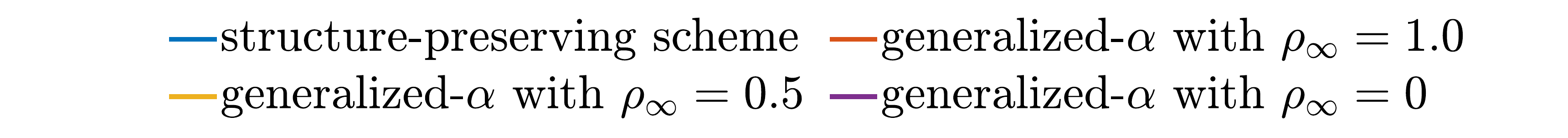} }
\end{tabular}
\caption{Residual errors of the power expenditure relation \eqref{eq:em-prop-1} by using the proposed structure-preserving scheme and the generalized-$\alpha$ schemes, where $P_{\mathrm{ext}}$ denotes the power of the external loads. The time step size is $\Delta t = 0.01$.}
\label{fig:cantilever_errors_power}
\end{center}
\end{figure}

\begin{figure}
\begin{center}
\begin{tabular}{cc}
\includegraphics[angle=0, trim=90 85 120 110, clip=true, scale = 0.22]{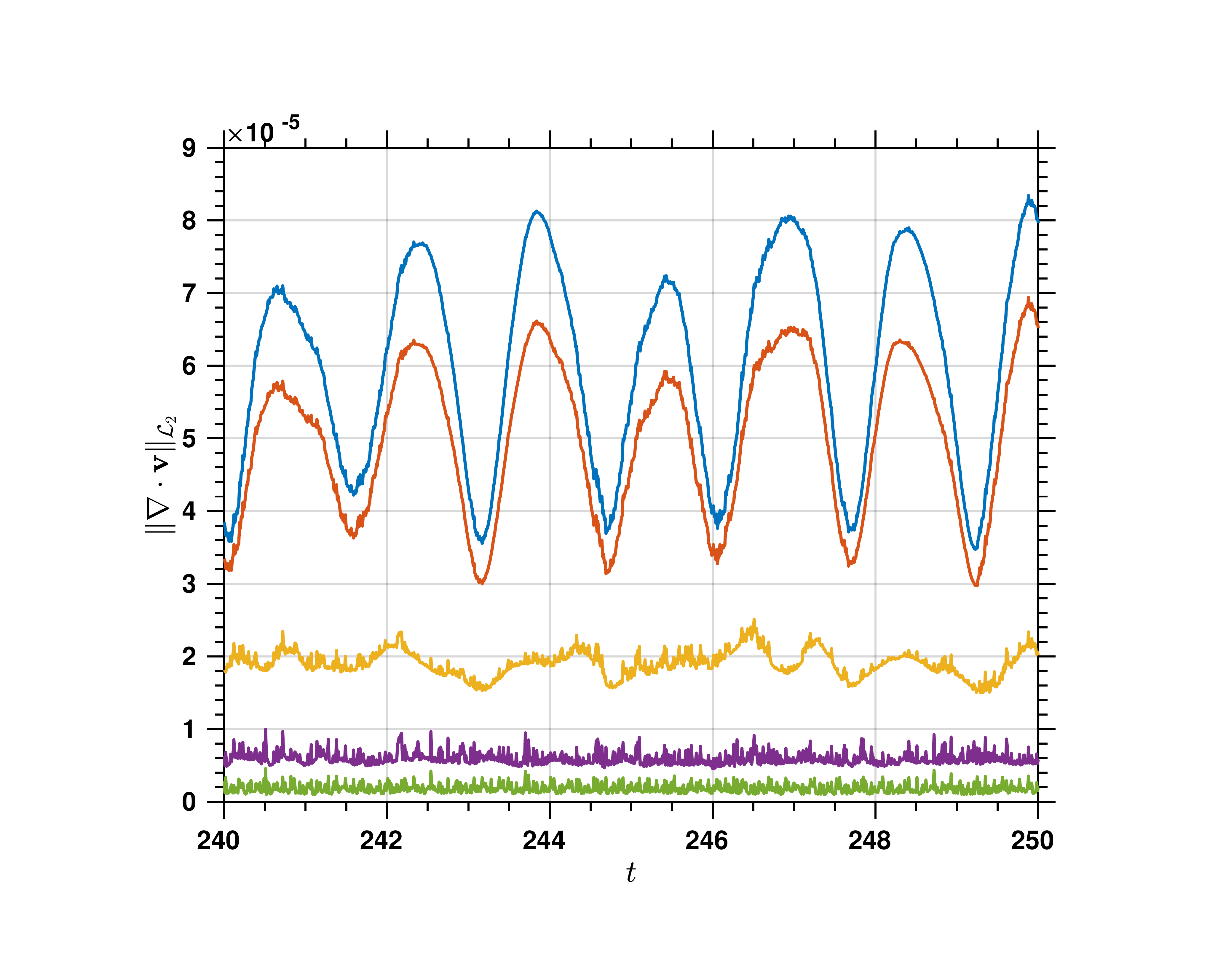} &
\includegraphics[angle=0, trim=90 85 120 110, clip=true, scale = 0.22]{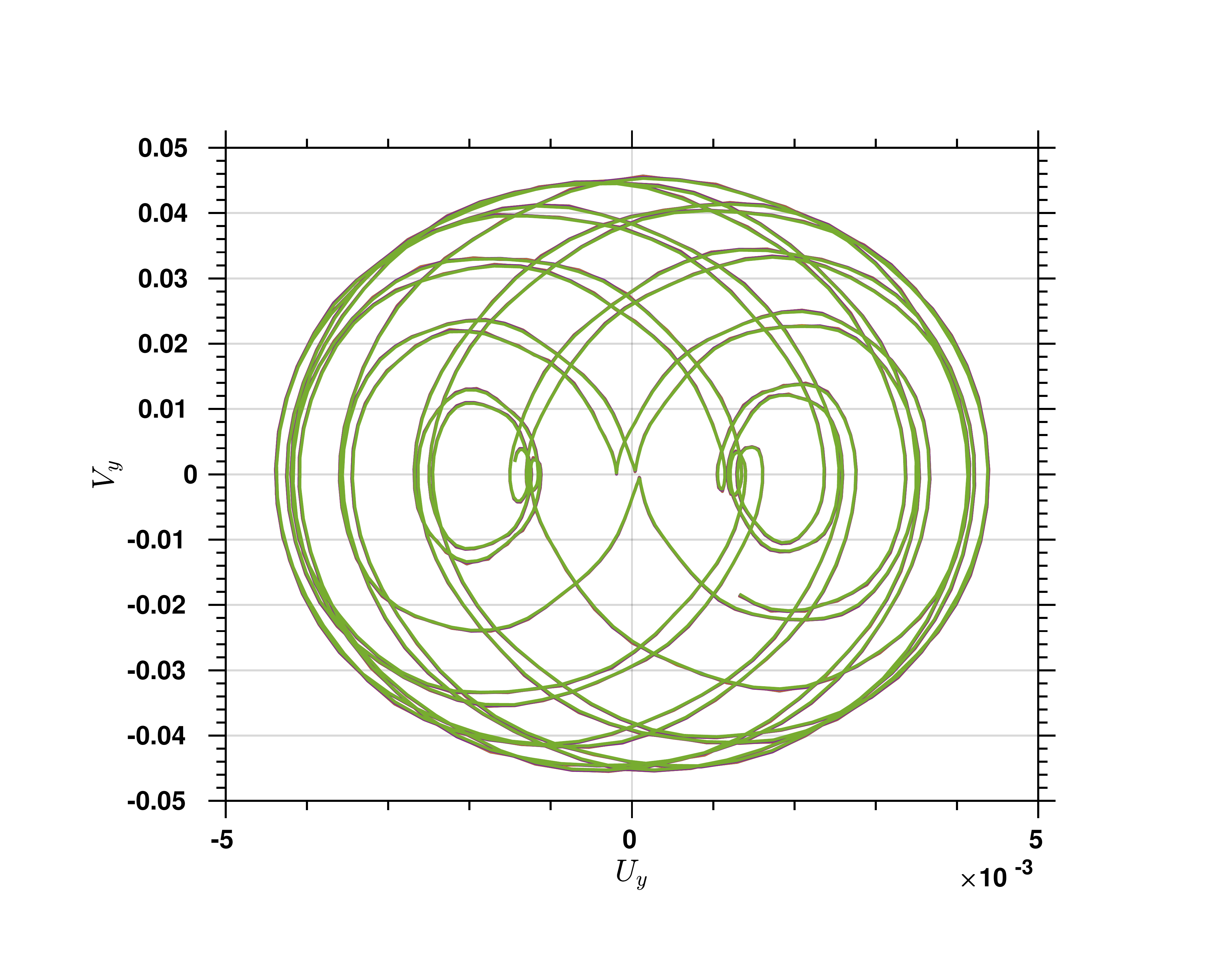} \\
(a) & (b) \\
\multicolumn{2}{c}{ \includegraphics[angle=0, trim=0 165 0 750, clip=true, scale = 0.35]{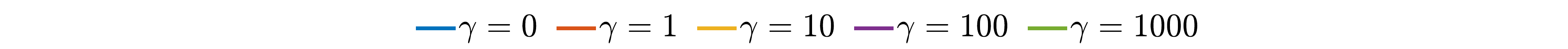} }
\end{tabular}
\caption{The effect of the grad-div stabilization parameter  $\gamma$ for the vibrating cantilever: (a) the time histories of $\| \nabla \cdot \bm{v} \|_{\mathcal{L}_2}$, (b) the trajectories of point A in the displacement-velocity phase space.}
\label{fig:cnatilever_graddiv}
\end{center}
\end{figure}

First, we record the trajectory of point A located at $(0, 0.005, 0.3)$ in the velocity-displacement phase space. Since we are concerned with the integrator performance for long-time integration, the last $1000$ steps are depicted in Figure \ref{fig:cantilever_phase_space}. In Figure \ref{fig:cantilever_phase_space} (a), we observe obvious oscillations of the trajectory given by the generalized-$\alpha$ scheme with $\rho_{\infty} = 1$ in comparison with the result of the structure-preserving scheme. In Figure \ref{fig:cantilever_phase_space} (b) and (c), the phase diagram of the generalized-$\alpha$ scheme is rather distinct from that of the structure-preserving scheme, likely due to the dissipative nature of the time integration scheme.

Due to the external harmonic surface load and the Dirichlet boundary condition, the Hamiltonian and the momenta are not strictly conserved. Still, we may assess the energy stability by calculating the residual of the power expenditure relation \eqref{eq:em-prop-1}. In Figure \ref{fig:cantilever_errors_power}, the magnitude of residual errors resulting from the proposed consistent scheme is $\mathcal{O}(10^{-11})$. In contrast, the magnitudes of residual errors given by the generalized-$\alpha$ method with different $\rho_{\infty}$ are several orders of magnitude larger than those of the structure-preserving integrator.

Finally, the performance of the grad-div stabilization on the $\mathcal{L}_2$-norm of $\nabla \cdot \bm{v}$ and the trajectories of the point A are exhibited in Figure \ref{fig:cnatilever_graddiv}. Figure \ref{fig:cnatilever_graddiv} (a) shows that the grad-div stabilization can effectively improve the discrete satisfaction of the divergence-free constraint. Figure \ref{fig:cnatilever_graddiv} (b) demonstrates that the trajectories calculated with different values of $\gamma$ are indistinguishable, meaning that, at least for this case considered here, the impact of the numerical dissipation $\mathcal D_m$ on the trajectories in the phase space is insignificant after long-time integration.

\section{Conclusion}
\label{sec:conclusion}
In this work, we first discussed the invariants for incompressible elastodynamics and designed a consistent scheme with respect to the energy and momenta. In particular, the Hamiltonian for fully incompressible materials is identified, in which the potential energy only involves the isochoric part of the energy. Considering that the previously existing studies are restricted to quasi-incompressible materials \cite{Gonzalez2000b,Hauret2006}, this newly identified Hamiltonian can be beneficial in analyzing dynamical systems. Invoking the discrete gradient technique, an algorithmic stress is constructed to maintain critical structures in discrete solutions. Abundant numerical examples justify its effectiveness and demonstrated its advantage over traditional integrators. Additionally, we considered an alternate option based on the scaled mid-point formula. Although it is theoretically appealing \cite{Orden2021}, our analysis reveals that it is non-robust in floating-point calculations. 

Second, the volume ratio is also an invariant in the fully incompressible regime. Preserving this quantity is of great importance in detecting physical instabilities \cite{Auricchio2013}. Inspired by recent advances in computational fluid dynamics, we introduced the grad-div stabilization technique to elastodynamics. This technique is attractive in that it can deliver divergence-free solutions asymptotically \cite{Scott1985a}. From the perspective of structure preservation, the grad-div stabilization term is momentum-preserving and energy-dissipative. Also, the dissipation will vanish as the stabilization parameter approaches infinity. This technique, combined with the smooth generalization of Taylor-Hood element, constitutes a robust, accurate, computationally convenient, and stable discretization technique for finite-strain analysis.

Third, our discussion focused on the material model formulated in principal stretches. Our derivation showed that there are missing terms in the elasticity tensor formula documented in the literature \cite{Holzapfel2000,Simo1991}. Numerical evidence suggests that the missing term can be critical in maintaining the discrete conservation. Also, we adopted a recently developed algorithm for the spectral decomposition. Our numerical example indicates that this algorithm is accurate enough for maintaining the discrete conservation properties, and it is thus superior to the widely-used algorithms based on Cardano's formula, which necessitates extra techniques when used in energy-momentum consistent schemes \cite{Mohr2008}.

There are several directions worth pursuing in the future. First, the volume-preserving algorithm \cite{Feng1995,McLachlan2002} can be applied with the current framework to further enhance the preservation of the volume ratio. This can lead to a new strategy for circumventing the issue in detecting the stability range for nonlinear elasticity \cite{Auricchio2013}. Second, the optimal choice of the stabilization parameter is worth investigating. Third, the newly identified Hamiltonian and the structure-preserving algorithm can be applied to study problems involving additional nonlinearity, such as inelasticity and contact problems.

\section*{Acknowledgements}
This work is supported by the National Natural Science Foundation of China [Grant Numbers 12172160, 12072143], Southern University of Science and Technology [Grant Number Y01326127], and the Guangdong-Hong Kong-Macao Joint Laboratory for Data-Driven Fluid Mechanics and Engineering Applications [Grant Number 2020B1212030001]. Computational resources are provided by the Center for Computational Science and Engineering at the Southern University of Science and Technology.

\appendix
\section{Elasticity tensor for the Ogden model}
\label{sec:appendix-A}
Here, we derive the formula \eqref{eq:CC-spectral-form} based on the spectral form of the isochoric stress \eqref{eq:ogden-S-spectral-form} and the formula for the elasticity tensor \eqref{eq:CC-rate-form}. We first introduce $S_1$ and $S_2$ as follows to facilitate the discussion,
\begin{align}
\label{eq:A1}
S_{\mathrm{ich~a}} = S_1(\lambda_a) S_2(\tilde{\lambda}_a),
\end{align}
in which
\begin{align*}
S_1(\lambda_a) = \frac{1}{\lambda_a^2} \quad \mbox{and} \quad S_2(\tilde{\lambda}_a) = \tilde{\lambda}_a \frac{\partial G_{\mathrm{ich}} }{\partial \tilde{\lambda}_a} -\frac{1}{3}\sum_{b=1}^{3} \tilde{\lambda}_b \frac{ \partial G_{\mathrm{ich}} }{ \partial \tilde{ \lambda}_b } = \sum_{p=1}^{N} \mu_p \left( \tilde{\lambda}_a^{\alpha_p} - \frac13 \sum_{c=1}^{3}\tilde{\lambda}_c^{\alpha_p} \right).
\end{align*}
Consequently, we have
\begin{align}
\label{eq:A2}
\frac{1}{\lambda_b}\frac{\partial S_{\mathrm{ich~a}}}{\partial \lambda_b} =  \frac{1}{\lambda_b} S_1 \frac{\partial S_2}{\partial \lambda_b} + \frac{1}{\lambda_b}\frac{\partial S_1}{\partial \lambda_b} S_2.
\end{align}
Straightforward calculations reveal that
\begin{align}
\frac{1}{\lambda_b} S_1 \frac{\partial S_2}{\partial \lambda_b} &= \frac{1}{\lambda_b} S_1 \sum_{c=1}^{3} \frac{\partial S_2}{\partial	\tilde{\lambda}_c} \frac{\partial \tilde{\lambda}_c}{\lambda_b} = \frac{1}{\lambda_b} S_1 \sum_{c=1}^{3}\frac{\partial S_2}{\partial \tilde{\lambda}_c} \left(J^{-1/3}\delta_{bc} -\frac{1}{3} J^{-1/3}\lambda^{-1}_{b} \lambda_{c} \right) \nonumber \displaybreak[2] \\
&= \frac{1}{\lambda_b^2} S_1 \left( \frac{\partial S_2}{\partial \tilde{\lambda}_b} \tilde{\lambda}_b -\frac{1}{3}\sum_{c=1}^{3}\frac{\partial S_2}{\partial \tilde{\lambda}_c} \tilde{\lambda}_c \right) \nonumber \displaybreak[2] \\
&= \frac{1}{\lambda_b^2} \frac{1}{\lambda_a^2} \left[ \left( \sum_{p=1}^{N}\mu_p \alpha_p \tilde{\lambda}_a^{\alpha_p -1} \tilde{\lambda}_b \delta_{ab} -\frac{1}{3}\sum_{p=1}^{N}\mu_p \alpha_p \tilde{\lambda}_b^{\alpha_p} \right) \right. \nonumber \displaybreak[2] \\
& \left. \qquad -\frac{1}{3}\sum_{c=1}^{3} \left(\mu_p \alpha_p \tilde{\lambda}_a^{\alpha_p-1}\tilde{\lambda}_c \delta_{ac} -\frac{1}{3}\sum_{p=1}^{N}\mu_p \alpha_p \tilde{\lambda}_{c}^{\alpha_p} \right)  \right] \nonumber \displaybreak[2] \\
&= \frac{1}{\lambda_b^2} \frac{1}{\lambda_a^2} \sum_{p=1}^{N}\mu_p \alpha_{p} \left(\tilde{\lambda}_a^{\alpha_p-1}\tilde{\lambda}_{b}\delta_{ab}-\frac{1}{3}\tilde{\lambda}_a^{\alpha_p} -\frac{1}{3}\tilde{\lambda}_b^{\alpha_p} + \frac{1}{9}\sum_{c=1}^{3}\tilde{\lambda}_c^{\alpha_p} \right) \nonumber \displaybreak[2] \\
&=
\label{eq:A25}
\begin{cases}
\frac{1}{\lambda_a^4} \sum\limits_{p=1}^{N}\mu_p \alpha_{p} \left( \frac{1}{3}\tilde{\lambda}_a^{\alpha_p} + \frac{1}{9}\sum\limits_{c=1}^{3}\tilde{\lambda}_c^{\alpha_p} \right) & a=b,  \displaybreak[2] \\
\frac{1}{\lambda_a^2} \frac{1}{\lambda_b^2} \sum\limits_{p=1}^{N}\mu_p \alpha_{p} \left(-\frac{1}{3}\tilde{\lambda}_a^{\alpha_p} -\frac{1}{3}\tilde{\lambda}_b^{\alpha_p} + \frac{1}{9}\sum\limits_{c=1}^{3}\tilde{\lambda}_c^{\alpha_p} \right) & a\neq b.
\end{cases}
\end{align}
The above is in fact the formula for the elasticity tensor documented in \cite[Page 264]{Holzapfel2000} or \cite[p.~285]{Simo1991}. However, there is a second term  in \eqref{eq:A2} which is missing in the above formula. It can be explicity written as
\begin{align}
\label{eq:A3}
\frac{1}{\lambda_b}\frac{\partial S_1}{\partial \lambda_b} S_2 = \frac{1}{\lambda_b}(-2 \lambda_a^{-3} \delta_{ab})S_2 =
\begin{cases}
-2\frac{1}{\lambda_a^4}\sum\limits_{p=1}^{N}\mu_p \left( \tilde{\lambda}_a^{\alpha_p} - \frac{1}{3}\sum\limits_{c=1}^{3}\tilde{\lambda}_c^{\alpha_p} \right) &a=b, \\
0 &a\neq b,
\end{cases}
\end{align}
Adding \eqref{eq:A3} to the conventional formula \eqref{eq:A25} leads to the final corrected one
\begin{align*}
\frac{1}{\lambda_b}\frac{\partial S_{\mathrm{ich~a}}}{\partial \lambda_b} =&
\begin{cases}
\lambda_a^{-4}\sum\limits_{p=1}^{N}\mu_p \alpha_p \left( (\frac{1}{3}-\frac{2}{\alpha_p}) \tilde{\lambda}_a^{\alpha_p} + (\frac{1}{9}+\frac{2}{3\alpha_p}) \sum\limits_{c=1}^3 \tilde{\lambda}_c^{\alpha_p}  \right) &  a=b, \\[1.6em]
\lambda_a^{-2}\lambda_b^{-2}\sum\limits_{p=1}^N \mu_p \alpha_p\left(-\frac{1}{3}
\tilde{\lambda}_b^{\alpha_p}-\frac{1}{3}\tilde{\lambda}_a^{\alpha_p} + \frac{1}{9}\sum\limits_{c=1}^3 \tilde{\lambda}_c^{\alpha_p}\right) & a\neq b.
\end{cases}
\end{align*}

\bibliographystyle{plain}
\bibliography{SOLID_MECH}

\begin{thebibliography}{10}

\bibitem{propeller_geo2023}
\url{https://github.com/guanjiashen/propeller_geometry_files}.

\bibitem{Amestoy2001}
P.R. Amestoy, I.S. Duff, J.-Y. L’Excellent, and J.~Koster.
\newblock {MUMPS}: a general purpose distributed memory sparse solver.
\newblock In {\em Applied Parallel Computing. New Paradigms for HPC in Industry
  and Academia: 5th International Workshop, PARA 2000 Bergen, Norway, June
  18-20, 2000 Proceedings 5}, pages 121--130, 2001.

\bibitem{Armero2006}
F.~Armero.
\newblock Energy-dissipative momentum-conserving time-stepping algorithms for
  finite strain multiplicative plasticity.
\newblock {\em Computer Methods in Applied Mechanics and Engineering},
  195(37-40):4862--4889, 2006.

\bibitem{Armero1998}
F.~Armero and E.~Pet{\H{o}}cz.
\newblock Formulation and analysis of conserving algorithms for frictionless
  dynamic contact/impact problems.
\newblock {\em Computer Methods in Applied Mechanics and Engineering},
  158(3-4):269--300, 1998.

\bibitem{Armero2001a}
F.~Armero and I.~Romero.
\newblock On the formulation of high-frequency dissipative time-stepping
  algorithms for nonlinear dynamics. {P}art {I}: low order methods for two
  model problems and nonlinear elastodynamics.
\newblock {\em Computer Methods in Applied Mechanics and Engineering},
  190:2603--2649, 2001.

\bibitem{Armero2001b}
F.~Armero and I.~Romero.
\newblock On the formulation of high-frequency dissipative time-stepping
  algorithms for nonlinear dynamics. {P}art {II}: second-order methods.
\newblock {\em Computer Methods in Applied Mechanics and Engineering},
  190(51-52):6783--6824, 2001.

\bibitem{Armero2007}
F.~Armero and C.~Zambrana-Rojas.
\newblock Volume-preserving energy-momentum schemes for isochoric
  multiplicative plasticity.
\newblock {\em Computer Methods in Applied Mechanics and Engineering},
  196(41-44):4130--4159, 2007.

\bibitem{Auricchio2010}
F.~Auricchio, L.~Beir{\~a}o da~Veiga, C.~Lovadina, and A.~Reali.
\newblock The importance of the exact satisfaction of the incompressibility
  constraint in nonlinear elasticity: mixed {FEM}s versus {NURBS}-based
  approximations.
\newblock 199:314--323.

\bibitem{Auricchio2005}
F.~Auricchio, L.~Beir{\~a}o da~Veiga, C.~Lovadina, and A.~Reali.
\newblock A stability study of some mixed finite elements for large deformation
  elasticity problems.
\newblock {\em Computer Methods in Applied Mechanics and Engineering},
  194:1075--1092, 2005.

\bibitem{Auricchio2013}
F.~Auricchio, L.~Beir{\~a}o da~Veiga, C.~Lovadina, A.~Reali, R.L. Taylor, and
  P.~Wriggers.
\newblock Approximation of incompressible large deformation elastic problems:
  some unresolved issues.
\newblock {\em Computational Mechanics}, 52:1153--1167, 2013.

\bibitem{Betsch2007}
P.~Betsch and C.~Hesch.
\newblock Energy-momentum conserving schemes for frictionless dynamic contact
  problems.
\newblock In {\em IUTAM Symposium on Computational Methods in Contact
  Mechanics}, pages 77--96. Springer, 2007.

\bibitem{Betsch2018}
P.~Betsch, A.~Janz, and C.~Hesch.
\newblock A mixed variational framework for the design of energy-momentum
  schemes inspired by the structure of polyconvex stored energy functions.
\newblock {\em Computer Methods in Applied Mechanics and Engineering},
  335:660--696, 2018.

\bibitem{Betsch2000}
P.~Betsch and P.~Steinmann.
\newblock Conservation properties of a time {FE} method-part {I}: time-stepping
  schemes for $n$-body problems.
\newblock {\em International Journal for Numerical Methods in Engineering},
  49:599--638, 2000.

\bibitem{Betsch2001}
P.~Betsch and P.~Steinmann.
\newblock Conservation properties of a time {FE} method-part {II}:
  {T}ime-stepping schemes for non-linear elastodynamics.
\newblock {\em International Journal for Numerical Methods in Engineering},
  50(8):1931--1955, 2001.

\bibitem{Bonet2015}
J.~Bonet, A.J. Gil, C.H. Lee, M.~Aguirre, and R.~Ortigosa.
\newblock A first order hyperbolic framework for large strain computational
  solid dynamics. {P}art {I}: {T}otal {L}agrangian isothermal elasticity.
\newblock {\em Computer Methods in Applied Mechanics and Engineering},
  283:689--732, 2015.

\bibitem{Bonet2015a}
J.~Bonet, A.J. Gil, and R.~Ortigosa.
\newblock A computational framework for polyconvex large strain elasticity.
\newblock {\em Computer Methods in Applied Mechanics and Engineering},
  283:1061--1094, 2015.

\bibitem{Bowers2014}
A.L. Bowers, S.~L. Borne, and L.G. Rebholz.
\newblock Error analysis and iterative solvers for {N}avier-{S}tokes projection
  methods with standard and sparse grad-div stabilization.
\newblock 275:1--19.

\bibitem{Bucelli2021}
M.~Bucelli, M.~Salvador, L.~Ded{\`e}, and A.~Quarteroni.
\newblock Multipatch {I}sogeometric {A}nalysis for electrophysiology:
  {S}imulationin a human heart.
\newblock {\em Computer Methods in Applied Mechanics and Engineering},
  376:113666, 2021.

\bibitem{Buffa2011}
A.~Buffa, C.~De Falco, and G.~Sangalli.
\newblock Isogeometric analysis: stable elements for the 2{D} {S}tokes
  equation.
\newblock {\em International Journal for Numerical Methods in Fluids},
  65(11-12):1407--1422, 2011.

\bibitem{Bui2007}
Q.V. Bui.
\newblock On the enforcing energy conservation of time finite elements for
  discrete elasto-dynamics problems.
\newblock {\em International Journal for Numerical Methods in Engineering},
  70:127--162, 2007.

\bibitem{Cao2006}
D.Q. Cao, D.~Liu, and C.H.T. Wang.
\newblock Three-dimensional nonlinear dynamics of slender structures: Cosserat
  rod element approach.
\newblock {\em International Journal of Solids and Structures},
  43(3-4):760--783, 2006.

\bibitem{Case2011}
M.A. Case, V.J. Ervin, A.~Linke, and L.G. Rebholz.
\newblock A connection between {S}cott-{V}ogelius and grad-div stabilized
  {T}aylor-{H}ood {FE} approximations of the {N}avier-{S}tokes equations.
\newblock {\em SIAM Journal on Numerical Analysis}, 49(5):1461--1481, 2011.

\bibitem{Chorin1978}
A.J. Chorin, T.J.R. Hughes, M.F. McCracken, and J.E. Marsden.
\newblock Product formulas and numerical algorithms.
\newblock {\em Communications on Pure and Applied Mathematics}, 31:205--256,
  1978.

\bibitem{Chung1993}
J.~Chung and G.M. Hulbert.
\newblock A time integration algorithm for structural dynamics with improved
  numerical dissipation: the generalized-$\alpha$ method.
\newblock {\em Journal of Applied Mechanics}, 60:371--375, 1993.

\bibitem{Colomes2016}
O.~Colomes, S.~Badia, and J.~Principe.
\newblock Mixed finite element methods with convection stabilization for the
  large eddy simulation of incompressible turbulent flows.
\newblock {\em Computer Methods in Applied Mechanics and Engineering},
  304:294--318, 2016.

\bibitem{SouzaNeto2008}
E.A. de~Souza~Neto, D.~Peri\'c, and D.R.J. Owen.
\newblock {\em Computational Methods for Plasticity: Theory and Applications}.
\newblock Wiley, 2008.

\bibitem{Dittmann2019}
M.~Dittmann, S.~Schu{\ss}, B.~Wohlmuth, and C.~Hesch.
\newblock Weak $c^n$ coupling for multipatch isogeometric analysis in solid
  mechanics.
\newblock {\em International Journal for Numerical Methods in Engineering},
  118(11):678--699, 2019.

\bibitem{Evans2013}
J.A. Evans and T.J.R. Hughes.
\newblock {Isogeometric divergence-conforming B-splines for the
  Darcy-Stokes-Brinkman equations}.
\newblock {\em Mathematical Models and Methods in Applied Sciences},
  23:671--741, 2013.

\bibitem{Feng1995}
K.~Feng and Z.J. Shang.
\newblock Volume-preserving algorithms for source-free dynamical systems.
\newblock {\em Numerische Mathematik}, 71:451--463, 1995.

\bibitem{Franca1988a}
L.P. Franca and T.J.R. Hughes.
\newblock Two classes of mixed finite element methods.
\newblock {\em Computer Methods in Applied Mechanics and Engineering},
  69:89--129, 1988.

\bibitem{Franke2018}
M.~Franke, A.~Janz, M.~Schiebl, and P.~Betsch.
\newblock An energy momentum consistent integration scheme using a
  polyconvexity-based framework for nonlinear thermo-elastodynamics.
\newblock {\em International Journal for Numerical Methods in Engineering},
  115(5):549--577, 2018.

\bibitem{Franke2022}
M.~Franke, R.~Ortigosa, J.~Mart{\'\i}nez-Frutos, A.J. Gil, and P.~Betsch.
\newblock A thermodynamically consistent time integration scheme for non-linear
  thermo-electro-mechanics.
\newblock {\em Computer Methods in Applied Mechanics and Engineering},
  389:114298, 2022.

\bibitem{Ge1988}
Z.~Ge and J.E. Marsden.
\newblock Lie-{P}oisson {H}amilton-{J}acobi theory and {L}ie-{P}oisson
  integrators.
\newblock {\em Physics Letters A}, 133:134--139, 1988.

\bibitem{Goicolea2000}
J.M. Goicolea and J.C.G. Orden.
\newblock Dynamic analysis of rigid and deformable multibody systems with
  penalty methods and energy-momentum schemes.
\newblock {\em Computer Methods in Applied Mechanics and Engineering},
  188(4):789--804, 2000.

\bibitem{Gonzalez1996b}
O.~Gonzalez.
\newblock Time integration and discrete hamiltonian systems.
\newblock {\em Journal of Nonlinear Science}, 6:449--467, 1996.

\bibitem{Gonzalez2000b}
O.~Gonzalez.
\newblock Exact energy and momentum conserving algorithms for general models in
  nonlinear elasticity.
\newblock {\em Computer Methods in Applied Mechanics and Engineering},
  190(13-14):1763--1783, 2000.

\bibitem{Gonzalez1996}
O.~Gonzalez and J.C. Simo.
\newblock On the stability of symplectic and energy-momentum algorithms for
  non-linear {H}amiltonian systems with symmetry.
\newblock {\em Computer Methods in Applied Mechanics and Engineering},
  134:197--222, 1996.

\bibitem{Greenspan1984}
D.~Greenspan.
\newblock Conservative numerical methods for $\ddot x = f(x)$.
\newblock {\em Journal of Computational Physics}, 56:28--41, 1984.

\bibitem{Gross2010}
M.~Gro{\ss} and P.~Betsch.
\newblock Energy-momentum consistent finite element discretization of dynamic
  finite viscoelasticity.
\newblock {\em International Journal for Numerical Methods in Engineering},
  81(11):1341--1386, 2010.

\bibitem{Harari2022}
I.~Harari and U.~Albocher.
\newblock Computation of eigenvalues of a real, symmetric $3\times3$ matrix
  with particular reference to the pernicious case of two nearly equal
  eigenvalues.
\newblock {\em International Journal for Numerical Methods in Engineering},
  pages 1--22, 2022.

\bibitem{Hartmann2003}
S.~Hartmann.
\newblock Computational aspects of the symmetric eigenvalue problem of second
  order tensors.
\newblock {\em Technische Mechanik}, pages 283--294, 2003.

\bibitem{Hauret2006}
P.~Hauret and P.~Le Tallec.
\newblock Energy-controlling time integration methods for nonlinear
  elastodynamics and low-velocity impact.
\newblock {\em Computer Methods in Applied Mechanics and Engineering},
  195(37-40):4890--4916, 2006.

\bibitem{Herrmann1965}
L.R. Herrmann.
\newblock Elasticity equations for incompressible and nearly incompressible
  materials by a variational theorem.
\newblock {\em AIAA Journal}, 3:1896--1900, 1965.

\bibitem{Hilber1977}
H.M. Hilber, T.J.R. Hughes, and R.L. Taylor.
\newblock Improved numerical dissipation for time integration algorithms in
  structural dynamics.
\newblock {\em Earthquake Engineering and Structural Dynamics}, 5:283--292,
  1977.

\bibitem{Holzapfel2000}
G.A. Holzapfel.
\newblock {\em Nonlinear {S}olid {M}echanics: {A} {C}ontinuum {A}pproach for
  {E}ngineering}.
\newblock John Wiley \& Sons, 2000.

\bibitem{Hughes1983}
T.J.R. Hughes.
\newblock {\em Analysis of transient algorithms with particular reference to
  stability behavior}, pages 67--155.
\newblock 1983.

\bibitem{Hughes1987}
T.J.R. Hughes.
\newblock {\em The {F}inite {E}lement {M}ethod: {L}inear {S}tatic and {D}ynamic
  {F}inite {E}lement {A}nalysis}.
\newblock Prentice Hall, Englewood Cliffs, NJ, 1987.

\bibitem{Hughes1978b}
T.J.R. Hughes, T.K. Caughey, and W.K. Liu.
\newblock Finite-element methods for nonlinear elastodynamics which conserve
  energy.
\newblock {\em Journal of Applied Mechanics}, 45(2):366--370, 1978.

\bibitem{Hughes2005}
T.J.R. Hughes, J.A. Cottrell, and Y.~Bazilevs.
\newblock Isogeometric analysis: {CAD}, finite elements, {NURBS}, exact
  geometry and mesh refinement.
\newblock {\em Computer Methods in Applied Mechanics and Engineering},
  194:4135--4195, 2005.

\bibitem{Janz2019}
A.~Janz, P.~Betsch, and M.~Franke.
\newblock Structure-preserving space-time discretization of a mixed formulation
  for quasi-incompressible large strain elasticity in principal stretches.
\newblock {\em International Journal for Numerical Methods in Engineering},
  120(13):1381--1410, 2019.

\bibitem{John2010}
V.~John and A.~Kindl.
\newblock Numerical studies of finite element variational multiscale methodsfor
  turbulent flow simulation.
\newblock {\em Computer Methods in Applied Mechanics and Engineering},
  199:841--852, 2010.

\bibitem{John2017}
V.~John, A.~Linke, C.~Merdon, M.~Neilan, and L.G. Rebholz.
\newblock On the divergence constraint in mixed finite element methods for
  incompressible flows.
\newblock {\em SIAM review}, 59(3):492--544, 2017.

\bibitem{Kane1999}
C.~Kane, J.E. Marsden, and M.~Ortiz.
\newblock Symplectic-energy-momentum preserving variational integrators.
\newblock {\em Journal of Mathematical Physics}, 40:3353--3371, 1999.

\bibitem{Kruger2016}
M.~Kr{\"u}ger, M.~Gro{\ss}, and P.~Betsch.
\newblock An energy-entropy-consistent time stepping scheme for nonlinear
  thermo-viscoelastic continua.
\newblock {\em ZAMM-Journal of Applied Mathematics and Mechanics/Zeitschrift
  f{\"u}r Angewandte Mathematik und Mechanik}, 96:141--178, 2016.

\bibitem{Kuhl1996}
D.~Kuhl and E.~Ramm.
\newblock Constraint energy momentum algorithm and its application to
  non-linear dynamics of shells.
\newblock {\em Computer methods in applied mechanics and engineering},
  136(3-4):293--315, 1996.

\bibitem{Kuhl1999}
D.~Kuhl and E.~Ramm.
\newblock Generalized energy-momentum method for non-linear adaptive shell
  dynamics.
\newblock {\em Computer Methods in Applied Mechanics and Engineering},
  178(3-4):343--366, 1999.

\bibitem{Laursen2001}
T.A. Laursen and X.N. Meng.
\newblock A new solution procedure for application of energy-conserving
  algorithms to general constitutive models in nonlinear elastodynamics.
\newblock {\em Computer Methods in Applied Mechanics and Engineering},
  190(46-47):6309--6322, 2001.

\bibitem{Lew2004}
A.~Lew, J.E. Marsden, M.~Ortiz, and M.~West.
\newblock Variational time integrators.
\newblock {\em International Journal for Numerical Methods in Engineering},
  60(1):153--212, 2004.

\bibitem{Lipton2010}
S.~Lipton, J.A. Evans, Y.~Bazilevs, T.~Elguedj, and T.J.R. Hughes.
\newblock {Robustness of isogeometric structural discretizations under severe
  mesh distortion}.
\newblock {\em Computer Methods in Applied Mechanics and Engineering},
  199:357--373, 2010.

\bibitem{Liu2021}
J.~Liu, M.~Latorre, and A.L. Marsden.
\newblock A continuum and computational framework for viscoelastodynamics: {I}.
  {F}inite deformation linear models.
\newblock {\em Computer Methods in Applied Mechanics and Engineering},
  385:114059, 2021.

\bibitem{Liu2018}
J.~Liu and A.L. Marsden.
\newblock A unified continuum and variational multiscale formulation for
  fluids, solids, and fluid-structure interaction.
\newblock {\em Computer Methods in Applied Mechanics and Engineering},
  337:549--597, 2018.

\bibitem{Liu2019}
J.~Liu and A.L. Marsden.
\newblock A robust and efficient iterative method for hyper-elastodynamics with
  nested block preconditioning.
\newblock {\em Journal of Computational Physics}, 383:72--93, 2019.

\bibitem{Liu2019a}
J.~Liu, A.L. Marsden, and Z.~Tao.
\newblock An energy-stable mixed formulation for isogeometric analysis of
  incompressible hyperelastodynamics.
\newblock {\em International Journal for Numerical Methods in Engineering},
  120:937--963, 2019.

\bibitem{Marsden1998}
J.E. Marsden, G.W. Patrick, and S.~Shkoller.
\newblock Multisymplectic geometry, variational integrators, and nonlinear
  {PDE}s.
\newblock {\em Communications in Mathematical Physics}, 199:351--395, 1998.

\bibitem{Conde2014}
S.~Conde Mart{\'\i}n, J.C.~Garc{\'\i}a Orden, and I.~Romero.
\newblock Energy-consistent time integration for nonlinear viscoelasticity.
\newblock {\em Computational Mechanics}, 54(2):473--488, 2014.

\bibitem{McLachlan2002}
R.I. McLachlan and G.R.W. Quispel.
\newblock Splitting methods.
\newblock {\em Acta Numerica}, 11:341--434, 2002.

\bibitem{Meng2002}
X.N. Meng and T.A. Laursen.
\newblock Energy consistent algorithms for dynamic finite deformation
  plasticity.
\newblock {\em Computer Methods in Applied Mechanics and Engineering},
  191(15-16):1639--1675, 2002.

\bibitem{Miehe1993}
C.~Miehe.
\newblock Computation of isotropic tensor functions.
\newblock {\em Communications in Numerical Methods in Engineering}, 9:889--896,
  1993.

\bibitem{Miehe2001}
C.~Miehe and J.~Schr{\"o}der.
\newblock Energy and momentum conserving elastodynamics of a non-linear
  brick-type mixed finite shell element.
\newblock {\em International Journal for Numerical Methods in Engineering},
  50(8):1801--1823, 2001.

\bibitem{Mohr2008}
R.~Mohr, A.~Menzel, and P.~Steinmann.
\newblock Galerkin-based mechanical integrators for finite elastodynamics
  formulated in principal stretches--pitfalls and remedies.
\newblock {\em Computer Methods in Applied Mechanics and Engineering},
  197(49-50):4444--4466, 2008.

\bibitem{ogden1972}
R.W. Ogden.
\newblock Large deformation isotropic elasticity--on the correlation of theory
  and experiment for incompressible rubberlike solids.
\newblock {\em Proceedings of the Royal Society of London A: Mathematical,
  Physical and Engineering Sciences}, 326:565--584, 1972.

\bibitem{Olshanskii2002}
M.~Olshanskii.
\newblock A low order galerkin finite element method for the {N}avier--{S}tokes
  equations of steady incompressible flow: a stabilization issue and iterative
  methods.
\newblock {\em Computer Methods in Applied Mechanics and Engineering},
  191(47-48):5515--5536, 2002.

\bibitem{Olshanskii2009}
M.~Olshanskii, G.~Lube, T.~Heister, and J.~L{\"o}we.
\newblock Grad-div stabilization and subgrid pressure models for the
  incompressible {N}avier-{S}tokes equations.
\newblock {\em Computer Methods in Applied Mechanics and Engineering},
  198(49-52):3975--3988, 2009.

\bibitem{Orden2021}
J.C.G. Orden.
\newblock A conserving formulation of a simple shear- and torsion-free beam for
  multibody applications.
\newblock 51:21--43.

\bibitem{Ortigosa2018}
R.~Ortigosa, M.~Franke, A.~Janz, A.J. Gil, and P.~Betsch.
\newblock An energy-momentum time integration scheme based on a convex
  multi-variable framework for non-linear electro-elastodynamics.
\newblock {\em Computer Methods in Applied Mechanics and Engineering},
  339:1--35, 2018.

\bibitem{Ortigosa2020}
R.~Ortigosa, A.J. Gil, J.~Mart{\'\i}nez-Frutos, M.~Franke, and J.~Bonet.
\newblock A new energy-momentum time integration scheme for non-linear
  thermo-mechanics.
\newblock {\em Computer Methods in Applied Mechanics and Engineering},
  372:113395, 2020.

\bibitem{Reich1996}
S.~Reich.
\newblock Enhancing energy conserving methods.
\newblock {\em BIT Numerical Mathematics}, 36:122--134, 1996.

\bibitem{Reissner1984}
E.~Reissner.
\newblock On a variational principle for elastic displacements and pressure.
\newblock {\em Journal of Applied Mechanics}, 51:444--445, 1984.

\bibitem{Romero2009}
I.~Romero.
\newblock Thermodynamically consistent time-stepping algorithms for non-linear
  thermomechanical systems.
\newblock {\em International Journal for Numerical Methods in Engineering},
  79:706--732, 2009.

\bibitem{Romero2012}
I.~Romero.
\newblock An analysis of the stress formula for energy-momentum methods in
  nonlinear elastodynamics.
\newblock {\em Computational Mechanics}, 50:603--610, 2012.

\bibitem{Rueberg2012}
T.~R{\"u}berg and F.~Cirak.
\newblock Subdivision-stabilised immersed b-spline finite elements for moving
  boundary flows.
\newblock {\em Computer Methods in Applied Mechanics and Engineering},
  209:266--283, 2012.

\bibitem{Scherzinger2008}
W.M. Scherzinger and C.R. Dohrmann.
\newblock A robust algorithm for finding the eigenvalues and eigenvectors of
  3$\times$ 3 symmetric matrices.
\newblock {\em Computer Methods in Applied Mechanics and Engineering},
  197(45-48):4007--4015, 2008.

\bibitem{Schroder2017}
J.~Schr\"{o}der, N.~Viebahn, P.~Wriggers, F.~Auricchio, and K.~Steeger.
\newblock On the stability analysis of hyperelastic boundary value problems
  using three- and two-field mixed finite element formulations.
\newblock {\em Computational Mechanics}, 60:479--492, 2017.

\bibitem{Scott1985a}
L.R. Scott and M.~Vogelius.
\newblock Conforming finite element methods for incompressible and nearly
  incompressible continua.
\newblock {\em Lectures in Applied Mathematics}, 22:221--244, 1985.

\bibitem{Shariff1997}
M.H.B.M. Shariff.
\newblock An extension of {H}errmann's principle to nonlinear elasticity.
\newblock 21(2):97--107.

\bibitem{Simo1992c}
J.C. Simo.
\newblock Algorithms for static and dynamic multiplicative plasticity that
  preserve the classical return mapping schemes of the infinitesimal theory.
\newblock {\em Computer Methods in Applied Mechanics and Engineering},
  99(1):61--112, 1992.

\bibitem{Simo1992}
J.C. Simo and F.~Armero.
\newblock Geometrically non-linear enhanced strain mixed methods and the method
  of incompatible modes.
\newblock {\em International Journal for Numerical Methods in Engineering},
  33:1413--1449, 1992.

\bibitem{Simo1993}
J.C. Simo and O.~Gonzalez.
\newblock Assessment of energy-momentum and symplectic schemes for stiff
  dynamical systems.
\newblock In {\em {ASME} winter annual meeting}, 1993.

\bibitem{Simo1992b}
J.C. Simo and N.~Tarnow.
\newblock The discrete energy-momentum method. {C}onserving algorithms for
  nonlinear elastodynamics.
\newblock {\em Zeitschrift f{\"u}r angewandte Mathematik und Physik ZAMP},
  43(5):757--792, 1992.

\bibitem{Simo1991}
J.C. Simo and R.L. Taylor.
\newblock Quasi-incompressible finite elasticity in principal stretches.
  continuum basis and numerical algorithms.
\newblock {\em Computer Methods in Applied Mechanics and Engineering},
  85:273--310, 1991.

\bibitem{Simo1985}
J.C. Simo, R.L. Taylor, and K.S. Pister.
\newblock Variational and projection methods for the volume constraint in
  finite deformation elasto-plasticity.
\newblock {\em Computer Methods in Applied Mechanics and Engineering},
  51:177--208, 1985.

\bibitem{Sussman1987}
T.~Sussman and K.J. Bathe.
\newblock A finite element formulation for nonlinear incompressible elastic and
  inelastic analysis.
\newblock {\em Computers \& Structures}, 26:357--409, 1987.

\end{thebibliography}

\end{document}